\documentclass[12pt,twoside,english,reqno,a4paper]{amsart}

\usepackage[utf8]{inputenc}
\usepackage[english]{babel}
\usepackage[lined,boxed]{algorithm2e}
\usepackage{float} %needed to fix the position of a figure
\usepackage{hyperref}
\usepackage{fullpage}
\usepackage{cite} %https://tex.stackexchange.com/questions/160552/bibtex-order-references-as-in-bibtex-file
\usepackage{listings,graphicx,caption,mwe,amsmath,varioref,amscd,amssymb,color,xcolor,bm,amsthm,amsfonts,graphics,lineno,float,comment,soul}
\usepackage{mathrsfs} % mathscr F
\usepackage[foot]{amsaddr}
\usepackage{bm}
\usepackage{tcolorbox}
\usepackage{enumitem}
\usepackage{indentfirst}
\usepackage{array}
\allowdisplaybreaks

\newcommand{\ignore}[1]{}

\theoremstyle{plain}
\newtheorem{definition}{Definition}[section]
\newtheorem{theorem}[definition]{Theorem}
\newtheorem{proposition}[definition]{Proposition}
\newtheorem{lemma}[definition]{Lemma}

\newtheorem{remark}[definition]{Remark}

\numberwithin{equation}{section}

\renewcommand{\theequation}{\thesection.\arabic{equation}}

\def\dis{\displaystyle}
\DeclareMathOperator*{\supp}{supp}
\def\R{\mathbb{R}}

\def\norma#1#2{\|#1\|_{\lower 4pt \hbox{$ \scriptstyle #2$ }}}

\newcommand{\eps}{\varepsilon}

\newcommand{\U}{\mathcal U}
\newcommand{\li}{\mathcal L}
\newcommand{\dd}{\tfrac{\mathrm{d}}{\mathrm{d}t}}
\newcommand{\di}{\mathrm{d}}
\newcommand{\pp}{\underline{p}}
\newcommand{\cc}{\underline{c}}
\newcommand{\vv}{\underline{v}}
\newcommand{\co}{\mathfrak{c}}
\newcommand{\uo}{\mathfrak{u}}
\newcommand{\po}{\mathfrak{p}}
\newcommand{\Di}{\mathrm{D}}

\newcommand{\cv}{\mathbf{c}}
\newcommand{\uv}{\mathbf{u}}
\newcommand{\pv}{\mathbf{p}}

\newcommand{\Aa}{\mathbf A}

\newcommand{\Ll}{\mathbf L}
\newcommand{\muo}{\pmb{\mu}}
\newcommand{\nuo}{\pmb{\nu}}
\newcommand{\PC}{\mathcal{P}_c(C)}
\newcommand{\PBC}{\mathcal{P}_b(C)}

\newcommand{\evt}{\left(\Phi_{(0,t)}^{\mu^0}\right)_{\#}\mu^0}
\newcommand{\ex}{\text{e}}
\newcommand{\G}{C^1_b(C;Z)}
\newcommand{\Das}{\Di_{\Phi_{(s,t)}^{\mu^s}(c)}A}

\newcommand{\K}{\mathcal{K}_{\omega,\tau}}
\newcommand{\cP}{\mathcal{P}}
\newcommand{\cT}{\mathcal{T}}
\newcommand{\de}{\mathrm{d}}

\DeclareMathOperator*{\argmax}{argmax}

\author[S. Almi]{Stefano Almi}
\address[Stefano Almi]{Dipartimento di Matematica e Applicazioni ``R.~Caccioppoli'', Universit\`a di Napoli Federico II, via Cintia, 80126 Napoli, Italy}\email{stefano.almi@unina.it}

\author[R. Durastanti]{Riccardo Durastanti}
\address[Riccardo Durastanti]{Dipartimento di Matematica e Applicazioni ``R.~Caccioppoli'', Universit\`a di Napoli Federico II, via Cintia, 80126 Napoli, Italy}\email{riccardo.durastanti@unina.it}

\author[F. Solombrino]{Francesco Solombrino}
\address[Francesco Solombrino]{Dipartimento di Matematica e Applicazioni ``R.~Caccioppoli'', Universit\`a di Napoli Federico II, via Cintia, 80126 Napoli, Italy} \email{francesco.solombrino@unina.it}

\makeatletter
\@namedef{subjclassname@2020}{%
  \textup{2020} Mathematics Subject Classification}
\makeatother


\begin{document}

\title[PMP for agent-based models with convex state space]{A Pontryagin Maximum Principle for agent-based models with convex state space}

\begin{abstract}
We derive a first order optimality condition for a class of agent-based systems, as well as for their mean-field counterpart. A relevant difficulty of our analysis is that the state equation is formulated on possibly infinite-dimensional convex subsets of Banach spaces. This is a typical feature of many problems in multi-population dynamics, where a convex set of probability measures may account for the population, the degree of influence or the strategy attached to each agent. Due to the lack of a linear structure and of local compactness, the usual tools of needle variations and linearisation procedures used to derive Pontryagin type conditions have to be generalised to the setting at hand. This is done by considering suitable notions of differentials and by a careful inspection of the underlying functional structures.
\end{abstract}

\maketitle
\tableofcontents

\newpage

\section{Introduction}
\subsection*{Presentation of the problem.}
Large systems of interacting agents, the so-called \textit{multi-agent systems}, have enjoyed growing attention from several mathematical communities in recent years. Indeed, they may encompass a broad class of applications, including, e.g., the modelling of autonomous vehicle ensembles \cite{Bullo2009}, swarms and flocking structures in the animal kingdom \cite{CS2}, coordinated animal motion \cite{animal}, opinion dynamics on networks \cite{Bellomo2013}, pedestrian flows \cite{CPT}, cell aggregation and motility \cite{motility2,motility1}, cooperative robots \cite{robots}, and influence of key investors in the stock market \cite[Introduction]{stock}. The modeling of such systems is usually inspired  from Newtonian laws of motion and is usually described by large dynamical systems based on pairwise forces accounting for repulsion/attraction, alignment, self-propulsion/friction in biological, social, or economical interactions. On the one hand, such discrete formulations often give rise to intractable problems because of their very large dimensionality. On the other hand, for many applications one is not interested in the exact pattern of the single agent, which may be indistinguishable from another one, but rather on the \textit{collective behavior} of the system. This point of view is also relevant for \textit{control-theoretic} purposes \cite{Albi2016,AlbiPareschiZanella,ControlKCS}, where a policy maker aims at formulating control laws which are as generic and global as possible in order to steer a given system towards a desired goal.

To overcome the so-called \textit{curse of dimensionality} one is naturally lead to consider infinite-dimensional approximations of the class of multi-agent systems at hand. They usually take the form of a continuity equation for the time-dependent distribution $\mu$ of the agent on the state space (a probability measure), driven by a possibly nonlocal velocity field accounting for particle interactions and external control. This is often referred to as \textit{mean-field approximation} of large particle systems (see \cite{golse} as well as the survey~\cite{Choi2014}).
These equations are studied in the space of probability measures endowed with the \textit{Wasserstein metrics} of optimal transport (see Section \ref{s3} below). The well-posedness of mean-field optimal control problems, as well as their compliance with particle formulations, has been addressed in several recent contributions (see, e.g., \cite{CLOS, FLOS, For-Solo}) hinging on a variational approach.

From an applicative standpoint, the identification of mean-field optimal controls relies on the implementation of suitable optimality conditions rather then on the direct minimisation of a cost functional. Such optimality conditions have been studied in the recent literature \cite{BFRS, BonFra, BonRos, PMPWassConst, Burger-Totzeck, Burger-Tse, Pogodaev2020, Carmona2015}, resulting in a generalisation  to Wasserstein spaces of the celebrated {\em Pontryagin Maximum Principle (PMP)}. The works \cite{BFRS, BonFra, BonRos} are of particular interest for our purposes. Indeed, in such papers the conditions of Pontryagin type are recovered by adapting to the geometric setting of Wasserstein spaces some classical tools coming from optimal control techniques in Euclidean setting, such as the use of needle variations and linearisation of the evolution constraints. The optimal control has to maximise a suitable Hamiltonian field $\mathcal{H}(\nu, u)$, where $\nu$ is a probability measure on $\mathbb{R}^{2d}$ accounting for both state and costate of the system, evolving through the Wasserstein Hamiltonian flow driven by $\mathcal{H}$.

In recent applications, a refined modelling comes into play, where the state space of an agent is no more the Euclidean, but is instead represented by a convex subset $C$ of a possibly infinite-dimensional Banach space $E$. This point of view has been introduced in the context of {\em spatially inhomogeneous evolutionary games} in~\cite{AFMS}. In this case, the state variable is a pair $(x, \lambda) \in C := \mathbb{R}^{d} \times \mathcal{P} (K)$, where~$\mathcal{P}(K)$ is the space of probability measure on a compact metric space of pure strategies $K$. The agents' state is therefore described by their position~$x$ and by their mixed strategy~$\lambda$. A related point of view has been considered in other contributions \cite{ADEMS, Mor1, MS2020} to describe the evolution of systems where each agent has a time-evolving degree of influence decoded by $\lambda$. This is useful in many applications where agents could belong to different populations, such as leaders and followers. In particular, a mean-field selective optimal control problem has been analysed in~\cite{AAMS}, where the action of the policy maker aimed at steering the whole system through a parsimonious intervention on possibly temporary leaders, i.e., via a control law depending on the (time evolving) variable $\lambda$. Notice that also the particle formulation of this systems requires additional care, as the involved operators must satisfy a suitable set of assumptions (see, e.g.,~\cite{MS2020}), in order to guarantee existence and uniqueness of solutions for differential equations constrained to live in a given convex set (see \cite{Brezis}). A detailed account for models in multi-population dynamics which fit into this setting is given in Section \ref{s:example}.

\subsection*{Our results.}
The goal of our paper is to extend the formulation of a PMP to the convex constrained setting. Our interest is motivated by the models we present in Section  \ref{s:example}, for which the optimality conditions in linear spaces available in literarure do not apply. The main technical tools are {\em (i)} a notion of differentiability for functionals defined on a convex subset $C$ of a Banach space $E$, {\em (ii)} a weak notion of local differentiability for functionals defined on $\mathcal{P}(C)$, and {\em (iii)} the introduction of a functional setting suitable for linearisation and for the derivation of the adjoint equation.

We start with the controlled  particle systems, which we recast in the form
\begin{equation}\label{intro-bol1}
\min_{u\in\U} \left\{\int_0^T L(t,c(t),u(t))\di t \right\}
\end{equation}
subject to
\begin{equation}\label{intro-may1}
\begin{cases}
\dd c(t)=A(t,c(t),u(t)) & \text{ in }(0,T], \\
c(0)=c_0\in C.
\end{cases}
\end{equation}
In \eqref{intro-bol1}, $\mathcal{U}$ is a weakly compact space of admissible controls, while the assumptions on the operator $A$ are listed in {\bf (HAode)} assure well-posedness of the state equation \eqref{intro-may1} constrained to $C$ and involve the notion of $C$-differentiability with respect to $c \in C$ (see also \cite{AFMS}). In Theorem \ref{PMP} we prove that the an optimal pair trajectory-control $(\co(t),\po(t))$ solves
\begin{equation*}
\label{intro-bol4}
\begin{cases}
\dis\dd \left(\begin{array}{c} \co(t) \\ \po(t) \end{array}\right) = J\Di_{(\co(t),\po(t))}\mathcal H(t,\co(t),\po(t),\uo(t)) & \text{ in } [0,T), \\
\co(0) = c_0 \in C, \\
\po(T) = \mathbf{0} \in E_C^*,
\end{cases}
\end{equation*}
on $C\times E_{C}^*$, where $E_{C} := \overline{\mathbb{R}( C - C)}$ is the topological closure of the vector subspace $\R(C - C):= \{ \alpha (c_{1} - c_{2}) : \alpha \in \R, \, c_{1}, c_{2} \in C\}$ in $E$. Furthermore, the optimal control maximises the Hamiltonian $\mathcal H(t,c(t),p(t),u(t)) = \left\langle p(t), A(t, c(t), u(t))\right\rangle - L(t, c (t), u(t))$, where the duality product is taken in $E_{C}^{*},E_C$. The proof is based on needle variations and on passing to the limit in the blow-up $(c_{\varepsilon} - c)/\varepsilon$, where $c_{\varepsilon}$ denotes the solution of the perturbed problem. Here some care is required, since a linear structure is not available and can only be partially replaced by convexity. In our setting tangent directions can be approximated only at first order by the perturbed trajectories, so that a careful estimate of the behavior of rest terms is in order. We refer to the proof of Theorems~\ref{teo1} and~\ref{PMP} for full details.

In Section \ref{s3} we turn our attention to the mean-field control problem
\begin{equation}
\label{intro-Bolinf-min}
\min_{u\in\mathcal U} \left\{\int_0^T L(t,\mu(t),u(t))\di t  \right\},
\end{equation}
subject to
\begin{equation}\label{intro-Wmayeq}
\begin{cases}
\dd \mu(t) +\mathrm{div}\left(A(t,\mu(t),\cdot,u(t,\cdot))\mu(t)\right)=0 & \text{ in }(0,T], \\
\mu(0)=\mu^0\in \mathcal{P}_c(C).
\end{cases}
\end{equation}
The set of admissible controls $\mathcal{U}$ has a closed-loop structure with respect to the state variable, which means that an admissible control $u$ may depend on both $t$ and $c$, and satisfies similar assumptions than the one considered in \cite{BonRos} to ensure compactness of controls. The operator $A$ is assumed to satisfy two sets of assumptions, namely {\bf (HA1)} and {\bf (HA2)}. The first is similar to {\bf (HAode)}, while the latter is crucial for writing a linearised equation in form of a nonlocal Cauchy problem~\eqref{lineq}. This requires, besides the already mentioned $C$-differentiability of $A$, a suitable notion of Wasserstein gradient with respect to $\mu$ introduced in Definition \ref{Wmudiff}, which is tailored for the functional setting at hand. Since the dynamics must live in $C$ and $A$ may be not linear in the control variable $u$, both assumptions on $A$ are slightly different from those of the Euclidean setting (see \cite{BonRos, PMPWassConst}).

The core of our arguments is contained in Proposition \ref{linearized}, where we perform the linearisation of the flow associated  to~\eqref{intro-Wmayeq}. Here is where our analysis significantly departs from \cite{BonFra, BonRos}. Indeed the states space lacks linearity and local compactness; furthermore we cannot in principle assume that the dependence of the velocity field $A$ on the closed-loop control $u$ is linear.

Also from a technical point of view the hardest part is contained in the proof of Proposition~\ref{linearized}. A crucial technical tool is a chain-rule with respect to the notion of Wasserstein differentiability, which we prove in Proposition~\ref{Wchainrule}, allowing for a first-order expansion of suitable compositions of flow maps close to the identity. This is heavily exploited in Theorem~\ref{mainres} to analyse the behavior of the perturbed solutions to~\eqref{intro-Wmayeq} arising from needle variations of the optimal control, as well as in the derivation of the  optimality conditions for the Hamiltonian associated to the problem, which is explicitly given in \eqref{Wham}. Our main result is contained in Theorem~\ref{mainresL}, where a PMP for \eqref{intro-Bolinf-min}--\eqref{intro-Wmayeq} is derived, coupling a maximality condition for the control $u$ with a Wasserstein Hamiltonian flow in the space $C \times E^*_{C}$. We remark that the derivation of such principle requires some assumptions on the involved (possibly infinite-dimensional) functional spaces. In particular, we need to assume $E$ to be separable and $E_{C}$ to be reflexive. The compliance of our setting with model examples inspired by~\cite{AAMS, ADEMS, AFMS, BFS, MS2020} and other recent contributions is discussed in Section~\ref{s:example}. The examples include leader-follower type dynamics, or in general multi-agent dynamics with  label switching, where the state space $\R^d \times \mathcal{P}(K)$ takes into account the position and the probability of each agent to belong to one element in the set $K$ of possible populations, and controlled replicator-type dynamics, where $\mathcal{P}(K)$ is a set of mixed strategies. We remark that the PMP provided by our contribution may be, for instance, a useful tool for the numerical results presented in \cite{AAMS}, which were instead based on model predictive control.

We conclude with the obvious remark that in the special case $C=\R^d$ one retrieves the known results obtained for instance in \cite{BonRos} for closed-loop controls and in \cite{BonFra} for open-loop controls.

\subsection*{Outlook.}
This work provides a general functional setting for identifying optimality conditions for agent-based systems with convex state space, where the evolution of the system is described by the continuity equation~\eqref{intro-Wmayeq}. Possible generalisations of the present framework may take into account evolutions driven by differential inclusions or stochastic effects, which have recently been considered in the Euclidean setting (see~\cite{ContInc} and~\cite{BonBon}, respectively). As pointed out in~\cite{LipReg} for systems with unconstrained state space, optimality conditions for the finite particle control problem~\eqref{intro-bol1}--\eqref{intro-may1} in combination with convexity assumptions on the Hamiltonian $\mathcal{H}$ may lead to Lipschitz regularity of optimal controls for the mean-field problem~\eqref{intro-Bolinf-min}--\eqref{intro-Wmayeq}. Such topics will be the focus of future investigation.

\section{Preliminaries and notation}
\label{s:preliminaries}

In a Banach space $(E, \| \cdot\|_{E})$, let $C$ be a closed convex subset of~$E$. For $R>0$ we define $B_R: = \{ e\in E : \|e\|_E\leq R\}$ and $B_{R,C}=B_R\cap C$. For any vector subspace $W$ of $E$ we denote by $\li(W;E)$ the Banach space of linear bounded operator from $W$ to $E$ and by $W^*$ the dual space of $W$, i.e. $W^*=\li(W;\R)$. Recall that, by density, each operator in $\li(W;E)$ uniquely extends to an operator in $\li(\overline{W};E)$. By $\langle \cdot, \cdot\rangle$ we denote the duality pairing between a vector space and its dual.
We denote by $E_C$ the topological closure of the vector subspace $\R(C-C):=\{ \alpha(c_1-c_2): \, \alpha\in \R, \, c_1,c_2\in C\}$ in $E$. We notice that $\R(C - C)$ is a vector space in view of the convexity of $C$. For $c\in C$, $E_c$ is the convex cone of directions defined as
$$
E_c:=\R_+(C-c)\subset E_C \subset E.
$$
We have the following definition of $C$-differentiability.

\begin{definition}
\label{d:cdiff}
Let $(F, \| \cdot\|_{F})$ be a normed space and let $f \colon C \to F$. We say that $f$ is~$C$-differentiable in $c \in C$ if there exists a linear operator $M \in \mathcal{L} (E_{C}; F)$ such that
\begin{displaymath}
\lim_{C\ni c' \to c}\, \frac{f(c') - f(c) - M[c' - c]}{\| c' - c\|_{E}} = 0\,.
\end{displaymath}
\end{definition}
We will denote the $C$-differential in $c \in C$ by $D_{c}f$. Notice that such map is a priori well-defined in $E_{c}$ and can be uniquely extended by linearity and density to a map in~$E_{C}$.

Given $(X,d)$ a separable Radon metric space, we denote by $\mathcal{P}(X)$ the family of all Borel probability measures on~$X$. For $p \geq 1$ we further consider
$$
\mathcal{P}_p(X):=\left\{\mu \in \mathcal{P}(X) : \int_X d(x,\bar{x})^p \di\mu(x)<+\infty \text{ for some }\bar{x}\in X  \right\}
$$
and $\mathcal{P}_c(X)$ the subset of $\mathcal{P}(X)$ of measures with compact support in $X$ recalling that the support is the closed set
$$
\supp(\mu)=\{x\in X : \mu(V)>0 \text{ for each neighborhood }V\text{ of }x  \}.
$$
If $X$ is contained in some Banach space $Y$, we define the $p$ momentum of $\mu\in\mathcal{P}(X)$ as
$$
m_p(\mu):=\left(\int_X \|x\|_Y^p \,\di \mu(x)\right)^\frac1p \qquad \mbox{for }p\geq 1.
$$
Let $(X_1,d_1)$ and $(X_2,d_2)$ be separable Radon metric spaces, we define for every $\mu^1\in \mathcal{P}(X_1)$ and $\mu^2\in\mathcal{P}(X_2)$ the transport plans with marginals $\mu^1$ and $\mu^2$
$$
\Gamma(\mu^1,\mu^2):=\left\{{\bm\gamma}\in \mathcal{P}(X_1\times X_2) : \pi_{\#}^{i}{\bm\gamma}=\mu^i \text{ for }i=1,2 \right\},
$$
where $\pi^i \colon X_1\times X_2 \to X_i$ is the projection on $X_i$ and $\pi_{\#}^{i}{\bm\gamma}\in \mathcal{P}(X_i)$ is the pushforward of ${\bm\gamma}$ through $\pi^i$. Note that $\Gamma(\mu^1,\mu^2)$ is a not empty and compact subset of $\mathcal{P}(X_1\times X_2)$ (see Remark 5.2.3 of \cite{AGS}).
We define the $p$-Wasserstein distance between two probability measures $\mu^1$ and $\mu^2$ in $\mathcal{P}_p(X)$ by
$$
W_p^p(\mu^1,\mu^2)= \min\left\{\int_{X\times X}d(x_1,x_2)^p \di{\bm\gamma}(x_1,x_2) : {\bm\gamma}\in \Gamma(\mu^1,\mu^2)\right\},
$$
it follows from Proposition 7.15 of \cite{AGS} that $\mathcal{P}_p(X)$ endowed with the $p$-Wasserstein distance is a separable metric space which is complete if $X$ is complete.

\section{Pontryagin maximum principle for optimal control problems in convex spaces}
\label{s1}

Let the following assumptions hold:
\begin{tcolorbox}
\begin{itemize}
\item $(E, \|\cdot\|_E)$ is a Banach space;
\item $C\subset E$ is a closed and convex subset of $E$;
\item $(U,d_U)$ is a compact metric space of admissible controls;
\item $\U:=\left\{u\colon[0,T]\to U \text{ such that } u(\cdot) \text{ is Lebesgue measurable} \right\}$.
\end{itemize}
\end{tcolorbox}

We introduce here the main set of assumptions for the operator $A$ appearing in~\eqref{intro-may1}.

\begin{tcolorbox}
{\bf (HAode)} Let $A(t,\cdot,\cdot) \colon C \times U \to E, t\in [0,T]$, be a family of operators satisfying the following properties:
\begin{itemize}
\item[$(i)$] for every $R>0$ there exists a constant $L_R > 0$ such that for every $c,\tilde{c}\in B_{R,C}$, $t\in [0,T]$ and $u\in U$
\begin{equation}\label{Clip}
\|A(t,c,u)-A(t,\tilde{c},u)\|_E\leq L_R\|c-\tilde{c}\|_E;
\end{equation}
\item[$(ii)$] there exists $M>0$ such that for every $c\in C$, $t\in[0,T]$ and $u\in U$ it holds
$$
\|A(t,c,u)\|_E \leq  M(1+\|c\|_E);
$$
\item[$(iii)$] for every $c\in C$ and $u\in U$ the map $t\mapsto A(t,c,u)$ belongs to $L^1([0,T];E)$;
\item[$(iv)$] for every $t\in [0,T]$ and $c\in C$ the map $u\mapsto A(t,c,u)$ is such that
$$
\|A(t,c,u)-A(t,c,\tilde{u})\|_E \leq \omega(d_U(u,\tilde{u})),
$$
where $\omega:[0,+\infty) \to [0,+\infty)$ is not dependent on $t$ and with $\lim_{s\to 0^+} \omega(s)=0$;
\item[$(v)$] for every $R>0$ there exists $\theta>0$ such that for every $t\in [0,T]$ and $u\in U$
$$
c\in C, \|c\|_E \leq R \Rightarrow c+\theta A(t,c,u)\in C;
$$
\item[$(vi)$] for every $c\in C$, $t\in [0,T]$ and $u\in U$, $A$ is $C$-differentiable at $c\in C$, i.e. there exists $B_{(t,c,u)}\in \li(E_C;E_C)$ such that
\begin{equation} \label{ACdiff}
\lim_{C\ni c'\to c} \dis\frac{\|A(t,c',u)-A(t,c,u)-B_{(t,c,u)}[c'-c]\|_E}{\|c'-c\|_E} = 0;
\end{equation}
\end{itemize}
\end{tcolorbox}

Note that $(v)$ implies $A(t,\cdot,u)\colon C \to E_c$ and that by \eqref{ACdiff} we have for every $e\in E_c$
\begin{equation}\label{Gat}
B_{(t,c,u)}[e]=\lim_{h\to 0^+} \dis\frac{A(t,c+he,u)-A(t,c,u)}{h}.
\end{equation}
For every $t\in [0,T]$ and $u\in U$, we denote the $C$-differential of $A$ in $c\in C$ by
$$
\Di_{c}A:=\{B_{(t,c,u)}\in \li(E_C;E_C) : B_{(t,c,u)} \text{ fullfils } \eqref{ACdiff}\}.
$$
Let us give some results related to the notion of $C$-differentiability given in $(vi)$.
\begin{lemma}
\label{ext-bnd}
Let $A$ satisfy $(vi)$. Then the map $\Di A \colon C \to \li(E_C;E_C)$ with $\Di A(c):=\Di_cA$ is single valued. Moreover, if $(i)$ holds, $\Di_{c(t)}A:=B_{(t,c(t),u(t))}\in L^\infty([0,T];\li(E_C;E_C))$ for every $c\in AC([0,T];C)$ and $u\in\U$.
\end{lemma}

\begin{proof}
First we prove that $\Di A$ is single valued. If $e\in E_c$, then, by \eqref{Gat}, $\Di_cA[e]$ is a singleton. Moreover, if $v\in \mathbb{R}(C-C)$, then $v=\alpha(c_1-c_2)=\alpha(c_1-c)-\alpha(c_2-c)$ for some $c_1,c_2\in C$ and $\alpha\in\R$. Hence, by linearity, we deduce
$$
\Di_cA[v]=\alpha \Di_cA[c_1-c]-\alpha \Di_cA[c_2-c],
$$
where $c_1-c, c_2-c\in E_c$. Therefore $\Di_cA[v]$ too is singleton and, by density, $\Di A$ is single valued. \\
Now let $c \in AC([0,T];C)$ and $u\in\U$. By continuity we have that $c(t)\in B_{R,C}$ for some $R>0$. Since for $h$ small enough we have $c(t)+he\in C\cap B_{2R}$ with $e\in E_{c(t)}$ (recall that $C$ is convex), it follows from \eqref{Clip} and \eqref{Gat} that $\|\Di_{c(t)}A[e]\|_E \leq l\|e\|_E$ for some $l\geq 0$ depending on $R$. This property can be extended for $v\in\mathbb{R}(C-C)$. Indeed, using the decomposition $v=\alpha(c_1-c_2)=\alpha(c_1-c(t))-\alpha(c_2-c(t))$, we have
\begin{eqnarray*}
\|\Di_{c(t)}A[v]\|_E &=&|\alpha|\left\|\Di_{c(t)}A[c_1-c(t)]-\Di_{c(t)}A[c_2-c(t)]\right\|_{E} \\
&\stackrel{\eqref{Gat}} =&|\alpha|\left\|\lim_{h\to 0^+}\frac{A(t,c(t)+h(c_1-c(t)),u)-A(t,c(t)+h(c_2-c(t)),u)}{h} \right\|_{E}\\
&=& |\alpha|\left\|\lim_{h\to 0^+}\frac{A(t,(1-h)c(t)+hc_1,u)-A(t,(1-h)c(t)+hc_2,u)}{h} \right\|_{E}\\
&\stackrel{(i)}\leq& l|\alpha|\|c_1-c_2\|_E =l\|v\|_E.
\end{eqnarray*}
Hence, using again a density argument and recalling that $E_C=\overline{\mathbb{R}(C-C)}$, we deduce that $\|\Di_{c(t)}A\|_{\li(E_C;E_C)}\leq l$ in $[0,T]$.
\end{proof}

From now on we use $\Di_cA$ instead of $B_{(t,c,u)}$. We define the adjoint operator $\Di_c^*A \in\li(E_C^*,E_C^*)$ of $\Di_c A$ as
$$
\langle \Di_c^*A[p],v\rangle :=\langle p, \Di_cA[v]\rangle \quad \forall p\in E_C^*, v\in E_C.
$$
\begin{remark}
\label{adjbnd}
By its definition
$$
\|\Di_c^*A\|_{\li(E_C^*,E_C^*)}=\|\Di_cA\|_{\li(E_C;E_C)}.
$$
Then it follows from the second statement of Lemma \ref{ext-bnd} that
\begin{equation}\label{duale}
\Di_{c(t)}^*A\in L^\infty([0,T];\li(E_C^*,E_C^*)),
\end{equation}
for every $c\in AC([0,T]; C)$ and $u\in \U$.
\end{remark}

\begin{tcolorbox}
{\bf (H$\varphi$ode)} Let $\varphi\colon  C \to \R$ be $C$-differentiable for every $c\in C$ with $C$-differential $\Di_c\varphi\in E_C^*$ defined as in $(vi)$, i.e.
$$
\lim_{C\ni c'\to c} \dis\frac{|\varphi(c')-\varphi(c)-\langle \Di_c\varphi,c'-c\rangle|}{\|c'-c\|_E} = 0.
$$
\end{tcolorbox}

Under the sets of assumptions {\bf (HAode)} and {\bf (H$\varphi$ode)}, we aim at finding necessary conditions for optimal solutions $(\co(t),\uo(t))\in AC([0,T]; C) \times \U$ of the following Mayer problem with free terminal point:
\begin{equation}\label{maxMay}
\min_{u\in \U} \varphi(c(T,u))
\end{equation}
subject to
\begin{equation}\label{may1}
\begin{cases}
\dd c(t)=A(t,c(t),u(t)) & \text{ in }(0,T], \\
c(0)=c_0\in C.
\end{cases}
\end{equation}
Notice that from the assumptions on $A$ and from Theorem I.4 of \cite{Brezis} we deduce that \eqref{may1} admits a unique  weak solution $c \in AC([0, T; C)$ for every control~$u$ and every initial condition $c_{0} \in C$.

We have the following result:
\begin{theorem}\label{teo1}
Consider the optimal control problem \eqref{maxMay}-\eqref{may1}, under the assumptions {\bf (HAode)} and {\bf (H$\varphi$ode)}. Let $\uo\in \U$ be an admissible control whose corresponding trajectory $\co\in AC([0,T]; C)$ is optimal. Let $p\colon [0,T] \to E_C^*$ the weak solution of the adjoint equation
\begin{equation*}\label{may2}
\begin{cases}
\dd p(t)=-\Di_{\co(t)}^*A[p(t)] & \text{ in } [0,T), \\
p(T)= -\Di_{\co(T)}\varphi.
\end{cases}
\end{equation*}
Then the maximality condition
\begin{equation*}\label{may3}
\langle p(t),A(t,\co(t),\uo(t))\rangle=\dis\max_{\omega\in U}\left\{\langle p(t),A(t,\co(t),\omega)\rangle\right\}
\end{equation*}
holds for almost everywhere $t$ in $[0,T]$.
\end{theorem}

\begin{proof}
Fix any time $\tau\in (0,T]$ and any admissible control value $\omega\in U$. For $\varepsilon>0$ sufficiently small, we consider the following function called needle variation:
$$
u_\eps(t)=\begin{cases} \omega & \text{ if }t\in[\tau-\eps,\tau], \\
\uo(t) & \text{ otherwise}.
\end{cases}
$$
It follows from the assumptions on $A$ and from Theorem I.4 of \cite{Brezis} that
\begin{equation}\label{may52bis}
\begin{cases}
\dd c(t)=A(t,c(t),u_\eps(t)) & \text{ in }(0,T] \\
c(0)=c_0\in C
\end{cases}
\end{equation}
has a unique weak solution $c_\eps(t)\in C$ for a.e. $t\in [0,T]$. By its definition $c_\eps(t)=\co(t)$ for a.e. $t\in [0,\tau-\eps]$. Now we prove that $c_\eps(t)$ converges to $\co(t)$ uniformly in $[0,T]$. By assumptions on $A$ we deduce that $A(t,\co(t),\uo(t))$, $A(t,\co(t),u_\eps(t))$ and $A(t,c_\eps(t), u_\eps(t))$ are integrable on $[0,T]$. Moreover $R>0$ depending only on $c_0$ and $T$ exists such that $\|\co(t)\|_E + \|c_\eps(t)\|_E \leq R$ in $[0,T]$. Hence, we have
\begin{eqnarray}\label{may52}
\|\co(t)-c_\eps(t)\|_E &\stackrel{\eqref{may1},\eqref{may52bis}}\leq& \int_0^T \|A(s,\co(s),\uo(s))-A(s,\co(s),u_\eps(s))\|_E \di s \nonumber \\
&& + \int_0^T \|A(s,\co(s),u_\eps(s))-A(s,c_\eps(s),u_\eps(s))\|_E \di s \nonumber \\
&\stackrel{\eqref{Clip}}\leq & \int_0^T \|A(s,\co(s),\uo(s))-A(s,\co(s),u_\eps(s))\|_E \di s \nonumber \\
&&+ L_R\int_{0}^{T} \|\co(s)-c_\eps(s)\|_E \di s.
\end{eqnarray}
Since $u_\eps$ converges to $\uo$ in $L^1([0,T];U)$ as $\eps$ tends to $0^+$, by $(iv)$ we obtain that $A(t,\co(t),u_\eps(t))$ converges to $A(t,\co(t),\uo(t))$ almost everywhere in $[0,T]$ as $\eps$ tends to $0^+$. Using $(ii)$ we have that $\|A(s,\co(s),\uo(s))-A(s,\co(s),u_\eps(s))\|_E \leq 2M(1+R)$ for some $M>0$. Then applying the Lebesgue theorem we obtain that $A(t,\co(t),u_\eps(t))$ converges to $A(t,\co(t),\uo(t))$ in $L^1([0,T];E)$ as $\eps$ tends to $0^+$. This, together with \eqref{may52}, implies
\begin{equation*}
\|\co(t)-c_\eps(t)\|_E \leq \delta_\eps + L_R\int_{0}^{T} \|\co(s)-c_\eps(s)\|_E \di s,
\end{equation*}
where $\delta_\eps \to0$ as $\eps \to0^+$. Applying the integral formulation of the Gr\"onwall inequality we get
$$
\|\co(t)-c_\eps(t)\|_E \leq \delta_\eps \mathrm{e}^{L_RT},
$$
whence
\begin{equation}\label{may52tris}
\lim_{\eps\to 0^+} c_\eps(t) = \co(t) \qquad \text{uniformly in }[0,T].
\end{equation}

Note that, since $A(t,\co(t),\omega)$ and $A(t,\co(t),\uo(t))$ are integrable in $[0,T]$, then a.e. time $\tau$ in $[0,T]$ is one of their Lebesgue points. From now on $\tau$ is a time with such property. Notice that such set of Lebesgue points can be made independent of~$\omega \in U$, since $U$ is a compact metric space. Indeed, $U$ is also separable. Hence, it would be enough to fix a countable dense set $D$ of~$U$, construct the Lebesgue points $\tau \in [0, T]$ of $t\mapsto A(t, \co(t), \omega)$ for every $\omega \in D$. Such points would be of Lebesgue type also for $t \mapsto A(t, \co(t), \omega)$ for any $\omega \in U$ by the assumption {\bf (HAode)}-$(iv)$. It follows that a positive $\delta_\eps$ tending to $0$ as $\eps$ tends to $0$ exists such that
\begin{eqnarray*}
&&\dis\left\|\frac{c_\eps(\tau)-\co(\tau)}{\eps} - \left(A(\tau,\co(\tau),\omega) - A(\tau,\co(\tau),\uo(\tau))\right)\right \|_E  \\ &&\leq \frac{1}{\eps}\int_{\tau-\eps}^{\tau}\|A(s,c_\eps(s),\omega)-A(s,\co(s),\omega)\|_E \di s + \left\|\frac{1}{\eps}\int_{\tau-\eps}^{\tau}A(s,\co(s),\omega)\di s - A(\tau,\co(\tau),\omega)\right\|_E  \\
&& +\left\| A(\tau,\co(\tau),\uo(\tau))- \frac{1}{\eps}\int_{\tau-\eps}^{\tau}A(s,\co(s),\uo(s))\di s \right\|_E \\
&& \stackrel{\eqref{Clip}}\leq L_R \|\co-c_\eps\|_{L^\infty([0,T];E)} + \delta_\eps.
\end{eqnarray*}
Therefore, by \eqref{may52tris}, we obtain
\begin{equation}\label{may53}
\lim_{\eps\to 0^+} \dis\frac{c_\eps(\tau)-\co(\tau)}{\eps}= A(\tau,\co(\tau),\omega) - A(\tau,\co(\tau),\uo(\tau)),
\end{equation}
whence we get $A(\tau,\co(\tau),\omega) - A(\tau,\co(\tau),\uo(\tau)) \in \overline{E_{\co(\tau)}}$.

Now we consider the following equation a priori defined in $E_C$
\begin{equation}\label{may4}
\begin{cases}
\dd v(t)=\Di_{\co(t)}A[v(t)] & \text{ in }(\tau,T], \\
v(\tau)=A(\tau,\co(\tau),\omega) - A(\tau,\co(\tau),\uo(\tau)).
\end{cases}
\end{equation}
By assumption $(vi)$ and Lemma~\ref{ext-bnd}, using Theorem I.4 of \cite{Brezis} we infer the existence of a unique weak solution $v\in AC([\tau,T]; E_C)$ such that $\dd v\in L^1([\tau,T];E_C)$. We want to prove that
\begin{equation}\label{may5}
v(t)=\lim_{\eps\to 0^+} \frac{c_\eps(t)-\co(t)}{\eps} \quad \text{uniformly in }[\tau,T],
\end{equation}
which in particular implies that $v(t)\in \overline{E_{\co(t)}}$ a.e. $t\in [\tau,T]$. Let us start by noting that $c_\eps(t)$ and $\co(t)$ in $[\tau,T]$ solve the same equation (i.e. \eqref{may1} with $u(t)$ replaced by $\uo(t)$) with $c_\eps(\tau)$ and $\co(\tau)$ respectively as initial datum. We define
$$
v_\eps(t):=\frac{c_\eps (t)-\co(t)}{\eps} \quad \text{for }t\in [\tau,T].
$$
Hence by its definition $v_\eps(t)\in E_{\co(t)}$ for every $t\in [\tau,T]$. Thus \eqref{may5} is equivalent to
\begin{equation}\label{may51}
\lim_{\eps\to 0^+}\left\|v_\eps (t)- v(t) \right\|_{L^{\infty}([\tau,T];E)}=0.
\end{equation}
Moreover by \eqref{may53} $v_\eps(\tau)$ converges to $v(\tau)$ in $E$. Now, recalling that $\|c_\eps(t)\|_E+\|\co(t)\|_E \leq R$ in $[0,T]$ and denoting by $L$ the constant given by \eqref{Clip} (we omit the subscript R), by continuity with respect to the initial data (see Theorem 1 of \cite{AAMS}) we obtain for $t\in [\tau,T]$
\begin{equation}
\label{may54}
\|v_\eps(t)\|_E \leq \mathrm{e}^{LT}\|v_{\eps}(\tau)\|_E \stackrel{\eqref{may53}}\leq 2\mathrm{e}^{LT}\|v(\tau)\|_E.
\end{equation}
Now we note that by its definition $v_\eps(t)$ satisfies
\begin{equation}\label{may54bis}
\dd v_\eps(t)= \frac{1}{\eps}\left(A(t,c_\eps(t),\uo(t))-A(t,\co(t),\uo(t))\right)=\Di_{\co(t)}A[v_\eps(t)]+r_\eps(t),
\end{equation}
where
$$
r_\eps(t):= \frac{1}{\eps}\left(A(t,c_\eps(t),\uo(t))-A(t,\co(t),\uo(t))\right)-\Di_{\co(t)}A[v_\eps(t)].
$$
It follows from \eqref{may52tris}, \eqref{ACdiff} and \eqref{may54} that for every $t\in (\tau,T]$
$$
\lim_{\varepsilon \to 0} \|r_\eps(t)\|_E=\lim_{\varepsilon \to 0}\frac{\|A(t,c_\eps(t),\uo(t))-A(t,\co(t),\uo(t))-\Di_{\co(t)}A[c_\eps(t)-\co(t)]\|_E}{\|c_\eps(t)-\co(t)\|_E}\|v_\eps(t)\|_E = 0.
$$
Moreover by Lemma \ref{ext-bnd} we have
\begin{eqnarray*}
\|r_\eps(t)\|_E &\stackrel{\eqref{may54bis}}\leq & \frac{1}{\eps}\|A(t,c_\eps(t),\uo(t))-A(t,\co(t),\uo(t))\|_E + \|\Di_{\co(t)}A[v_\eps(t)]\|_E \\
&\stackrel{\eqref{Clip}} \leq & \left(L + \|\Di_{\co(t)}A\|_{L^\infty([0,T];\li(E_C;E_C))}\right)\|v_\eps(t)\|_E.
\end{eqnarray*}
Hence using \eqref{may54} and applying the Lebesgue theorem we obtain that $r_\eps \to 0$ in $L^1([\tau,T];E)$ as $\eps\to 0^+$. Therefore, recalling that $v$ solves \eqref{may4} and using again Lemma \ref{ext-bnd}, we obtain
\begin{eqnarray}\label{may56}
\|v_\eps(t)-v(t)\|_E &\leq& \|v_\eps(\tau)-v(\tau)\|_E + \left\| \int_\tau^t \Di_{\co(s)}A[v_\eps(s)-v(s)] \di s \right\|_E + \left\|\int_\tau^t r_\eps(s)\di s\right\|_E \nonumber \\
&\stackrel{\eqref{may53}}\leq & \delta_\eps + \|\Di_{\co(t)}A\|_{L^\infty([0,T];\li(E_C;E_C))}\int_\tau^t \|v_\eps(s)-v(s)\|_E \di s,
\end{eqnarray}
where $\delta_\eps: = \|v_\eps(\tau)-v(\tau)\|_E + \|r_\eps\|_{L^1([\tau,T];E)} \to 0$ as $\eps \to 0^+$. Since $\|v_\eps(t)-v(t)\|_E$ is a continuous function from $[\tau,T]$ to $\R$, by \eqref{may56}, applying the integral form of the Gr\"onwall inequality and recalling that $\|\Di_{\co(t)}A\|_{L^\infty([0,T];\li(E_C;E_C))}\leq L$ it follows that
$$
\|v_\eps(t)-v(t)\|_E\leq \delta_\eps \mathrm{e}^{LT}.
$$
This implies \eqref{may51} and, consequentially, \eqref{may5}.

Since the control $\uo(t)$ is optimal in $\U$, we deduce for every $\eps >0$
\begin{equation}\label{may8}
\varphi(\co(T))\leq \varphi(c_\eps(T)).
\end{equation}
Moreover by \eqref{may5} we have $c_\eps(T)=\co(T)+\eps v(T) + o(\eps)$. Hence, recalling that $\varphi$ is $C$-differentiable and that
$$
c_\eps(T)-\co(T)=\eps v(T) + o(\eps)\in E_{\co(t)}
$$
we have
\begin{eqnarray}\label{may9}
\varphi(c_\eps(T))-\varphi(\co(T)) &= &\varphi(\co(T)+\eps v(T) + o(\eps))-\varphi(\co(T))\nonumber \\
& = &\langle \Di_{\co(T)}\varphi, \eps v(T) + o(\eps) \rangle + o(\eps).
\end{eqnarray}
Therefore
\begin{equation}\label{may7}
0 \stackrel{\eqref{may8}}\leq \lim_{\eps\to 0^+}\dis\frac{\varphi(c_\eps(T))- \varphi(\co(T))}{\eps} \stackrel{\eqref{may9}}= \langle \Di_{\co(T)}\varphi,v(T)\rangle.
\end{equation}

Now we consider the adjoint equation of \eqref{may4} which transports $p(T)=\Di_{\co(T)}\varphi\in E_C^* $ backward in time, i.e.
\begin{equation}\label{may6}
\begin{cases}
\dd p(t) = -\Di_{\co(t)}^*A[p(t)] & \text{ in }[\tau,T), \\
p(T) = -\Di_{\co(T)}\varphi\in E_C^*.
\end{cases}
\end{equation}
It follows from \eqref{duale} and Theorem I.4 of \cite{Brezis} that $p(t)\in AC([\tau,T];E_C^*)$ is the unique weak solution of \eqref{may6}. Then $\dd p\in L^1([\tau,T];E_C^*)$. The density of $C^1([\tau,T];E_C)$ and $C^1([\tau,T];E_C^*)$ in $W^{1,1}([\tau,T];E_C)$ and $W^{1,1}([\tau,T];E_C^*)$ respectively implies that the function $\langle p(\cdot),v(\cdot)\rangle \colon [\tau,T] \to \R$ is weak differentiable and its weak derivative is
$$
\dd \langle p(t),v(t)\rangle =\langle \dd p(t),v(t)\rangle+ \langle p(t),\dd v(t)\rangle.
$$
Therefore, using equations \eqref{may4} and \eqref{may6} and the definition of adjoint operator, we obtain
$$
\dd \langle p(t),v(t)\rangle = 0 \quad \text{ a.e. } t\in [\tau,T].
$$
This implies that $\langle p(t),v(t)\rangle$ is constant in $[\tau,T]$, more precisely
\begin{eqnarray*}
0 &\stackrel{\eqref{may7},\eqref{may6}} \geq &\langle p(T),v(T)\rangle =\langle p(t),v(t)\rangle =\langle p(\tau),v(\tau)\rangle \\
&\stackrel{\eqref{may4}}=&\langle p(\tau),A(\tau,\co(\tau),\omega) - A(\tau,\co(\tau),\uo(\tau))\rangle
\end{eqnarray*}
for a.e. $\tau\in [0,T]$. Since $\omega\in U$ is arbitrary, we have that
$$
\langle p(\tau),A(\tau,\co(\tau),\uo(\tau))\rangle =\dis\max_{\omega\in U}\left\{\langle p(\tau), A(\tau,\co(\tau),\omega)\rangle \right\} \quad \text{ for a.e. }\tau\in [0,T].
$$
\end{proof}

Now we focus on the Bolza problem with running cost $L$, i.e.
\begin{equation}\label{bol1}
\min_{u\in\U} \left\{\int_0^T L(t,c(t),u(t))\di t \right\}
\end{equation}
subject to \eqref{may1}. We assume:

\begin{tcolorbox}
{\bf (HLode)} Let $L\colon [0,T]\times C \times U \to \R$ be a map satisfying:
\begin{itemize}
\item[$(a)$] for every $R>0$ there exists a constant $L_R> 0$ such that for every $c,\tilde{c} \in B_{R,C}$, $t\in [0,T]$ and $u\in U$
\begin{equation*}
|L(t,c,u)-L(t,\tilde{c},u)|\leq L_R\|c-\tilde{c}\|_E;
\end{equation*}
\item[$(b)$] for every $c\in C$ and $u\in U$ the map $t\mapsto L(t,c,u)$ belongs to $L^1([0,T];\R)$;
\item[$(c)$] for every $t\in [0,T]$ and $c\in C$ the map $u\mapsto L(t,c,u)$ is such that
$$
\|L(t,c,u)-L(t,c,\tilde{u})\|_E \leq \omega(d_U(u,\tilde{u})),
$$
where $\omega \colon [0,+\infty) \to [0,+\infty)$ is not dependent on $t$ and with $\lim_{s\to 0^+} \omega(s)=0$;
\item[$(d)$] for every $c\in C$, $t\in [0,T]$ and $u\in U$, $L$ is $C$-differentiable at $c\in C$ with $C$-differential $\Di_cL\in \li(E_C;\R)$.
\end{itemize}
\end{tcolorbox}

Under these assumptions on $L$ with the same argument contained in Lemma \ref{ext-bnd} we have that $\Di_cL$ is single valued in $E_C$ and $\Di_{c(t)}L\in L^\infty([0,T];E_C^*)$ for every $c\in AC([0,T];C)$ and $u\in\U$. \\

We can rewrite the Bolza problem in Mayer form introducing the auxiliary variable
$$
c_{au}(t)=\int_0^t L(s,c(s),u(s))\di s,
$$
with $c_{au}(0)=0$. We introduce the following notations:
\begin{eqnarray}\label{notbol}
&& \cc:= \left(
\begin{array}{c}
c \\
c_{au}
\end{array}
\right)\in C\times \R,
\quad
\vv:= \left(
\begin{array}{c}
v \\
v_{au}
\end{array}
\right)\in E_C\times \R, \quad
\pp:= \left(
\begin{array}{c}
p \\
p_{au}
\end{array}
\right)\in E_C^*\times \R, \nonumber \\
&&
\underline{A}(t,\cc,u) := \left(
\begin{array}{c}
A(t,c,u) \\
L(t,c,u)
\end{array}
\right).
\end{eqnarray}
Since $A$ and $L$ are $C$-differentiable, it follows that $\underline{A}\colon  [0,T]\times C\times \R \times U \to E_C$ is $C$-differentiable and
\begin{equation}
\label{DAau}
\Di_{\cc}\underline{A}:=  \left(
\begin{array}{cc}
\Di_{c}A & 0 \\
\Di_{c}L & 0
\end{array}
\right)\in \li(E_C\times \R; E_C\times \R).
\end{equation}
Finally, we denote by $\Di_{\cc}^*\underline{A}$ the adjoint operator of $\Di_{\cc}\underline{A}$. Note that for every $c\in AC([0,T];C)$ and $u\in\U$, since $\Di_{\cc(t)}\underline{A}\in L^\infty([0,T];\li(E_C\times \R; E_C\times \R))$, with the same argument of Remark \ref{adjbnd} we obtain $\Di_{\cc(t)}^*\underline{A}\in L^\infty([0,T];\li(E_C^*\times \R;E_C^*\times \R)$. Recalling that $A(t,c,u)\in E_c \subset E_C$, we define $\mathcal H\colon [0,T] \times C \times E_C^* \times U \to \R$ as
\begin{equation}\label{defHam}
\mathcal{H}(t,c,p,u)=\langle p, A(t,c,u) \rangle - L(t,c,u).
\end{equation}
By its definition and since $A$ and $L$ are $C$-differentiable, it follows that $\mathcal H$ is $C$-differentiable at $c\in C$ and its $C$-differential $\Di_c\mathcal H\in E_C^*$ has the form
\begin{equation*}
\Di_c\mathcal H[v]=\langle p, \Di_cA[v] \rangle - D_cL[v] = \langle \Di_c^*A[p],v \rangle - D_cL[v] \qquad \forall v\in E_C.
\end{equation*}
Moreover $\mathcal H$ is Fr\'echet-differentiable with respect to $p$ and $\Di_p\mathcal H = A(t,c,u)\in E_C \subset E_C^{**}$. It follows that $\mathcal H$ is differentiable in $C\times E_C^*$ and
\begin{equation}
\label{DHdoppio}
\Di_{(c,p)}\mathcal H(t,c,p,u) = \left(\begin{array}{c} \Di_c^*A[p] - D_cL \\ A(t,c,u) \end{array}\right) \in E_C^* \times E_C.
\end{equation}
In this way the Bolza problem \eqref{bol1}-\eqref{may1} is equivalent to the Mayer problem
\begin{equation}\label{bol2}
\min_{u\in\U} c_{au}(T,u)
\end{equation}
subject to
\begin{equation}\label{bol3}
\begin{cases}
\dd \cc(t) = \underline{A}(t,\cc(t),u(t)) & \text{ in }(0,T], \\
\cc(0)=\left(\begin{array}{c} c_0 \\0 \end{array}\right)\in C\times \R .
\end{cases}
\end{equation}
Thanks to the assumptions on $A$ and $L$ we can apply Theorem \ref{teo1} to problem \eqref{bol2}-\eqref{bol3} obtaining the following result.

\begin{theorem}
\label{PMP}
Under the assumptions {\bf (HAode)} and {\bf (HLode)}, let $\uo\in \U$ be an admissible control whose corresponding trajectory $\co\in AC([0,T]; C)$ is optimal for problem \eqref{bol1}-\eqref{may1}. Then there exists $\po\in AC([0,T]; E_C^*)$ such that $(\co,\po)$ solves in distributional sense
\begin{equation*}\label{bol4}
\begin{cases}
\dis\dd \left(\begin{array}{c} \co(t) \\ \po(t) \end{array}\right) = J\Di_{(\co(t),\po(t))}\mathcal H(t,\co(t),\po(t),\uo(t)) & \text{ in } [0,T), \\
\co(0) = c_0 \in C, \\
\po(T) = \mathbf{0} \in E_C^*,
\end{cases}
\end{equation*}
where $\mathcal H$ is defined by \eqref{defHam} and $J\colon E_C^* \times E_C \ni (T,e) \mapsto (e,-T)\in E_C\times E_C^*$. Moreover the maximality condition
\begin{equation*}\label{bol5}
\mathcal{H}(t,\co(t),\po(t),\uo(t)) =\dis\max_{\omega\in U} \mathcal{H}(t,\co(t),\po(t),\omega)
\end{equation*}
holds for almost everywhere $t$ in $[0,T]$.
\end{theorem}

\begin{proof}
Let us start by noting that, if $\uo\in \U$ be an admissible control whose corresponding trajectory $\co\in AC([0,T]; C)$ is optimal for problem \eqref{bol1}-\eqref{may1}, then by \eqref{notbol} and \eqref{bol1} we have that $\underline\co = (\co, \co_{au}) \in AC([0,T]; C\times \R)$ is optimal for problem \eqref{bol2}-\eqref{bol3}. Since $\varphi(\cc(T,u))=c_{au}(T,u)$, we deduce that $\displaystyle \Di_{\underline\co(T)}\varphi =\left(\begin{array}{c} \mathbf 0 \\ 1 \end{array}\right)\in E_C^*\times \R $. Hence it follows from Theorem \ref{teo1} that
\begin{eqnarray}
\label{bol6}
&&\left\langle \underline\po(t), \underline{A}(t,\underline\co(t),\uo(t)) \right\rangle \stackrel{\eqref{notbol}}=\langle \po(t),A(t,\co(t),\uo(t))\rangle + \po_{au}(t)L(t,\co(t),\uo(t)) \nonumber \\
&&=\dis\max_{\omega\in U}\left\{\langle \po(t),A(t,\co(t),\omega)\rangle + \po_{au}(t)L(t,\co(t),\omega)\right\}
\end{eqnarray}
holds for almost everywhere $t$ in $[0,T]$, where $\underline\po$ is the weak solution of
\begin{equation}\label{bol4bis}
\begin{cases}
\dd \pp(t)= -\Di_{\underline\co(t)}^*\underline{A}[\pp(t)] & \text{ in } [0,T), \\
\pp(T)= - \Di_{\underline\co(T)}\varphi = \left(\begin{array}{c} \mathbf 0 \\ -1 \end{array}\right)\in E_C^*\times \R.
\end{cases}
\end{equation}
Then for $(v,0)\in E_C \times \R$ it holds
\begin{eqnarray}\label{bol6bis}
\langle \dd \po(t), v\rangle & = &\left\langle \dd \underline\po(t), \left(\begin{array}{c} v \\ 0 \end{array}\right) \right\rangle \stackrel{\eqref{bol4bis}}= \left\langle -\Di_{\underline\co(t)}^*\underline{A}[\underline\po(t)], \left(\begin{array}{c} v \\ 0 \end{array}\right) \right\rangle  \nonumber \\
&=& - \left\langle \underline\po(t), \Di_{\underline\co(t)}\underline{A}\left[\begin{array}{c} v \\ 0 \end{array}\right] \right\rangle \stackrel{\eqref{notbol},\eqref{DAau}} = -\left \langle  \left(\begin{array}{c} \po(t) \\ \po_{au}(t) \end{array}\right), \left(\begin{array}{c} \Di_{\co(t)}A[v] \\  \Di_{\co(t)}L[v] \end{array}\right) \right\rangle  \nonumber \\
& = & -\langle \po(t), \Di_{\co(t)}A[v] \rangle - \po_{au}(t)\Di_{\co(t)}L[v] \nonumber \\
& = & -\langle \Di_{\co(t)}^*A[\po(t)],v\rangle - \po_{au}(t)\Di_{\co(t)}L[v].
\end{eqnarray}
Therefore, recalling \eqref{defHam} and \eqref{DHdoppio}, in order to have the result it remains to prove that $\po_{au}(t)= -1$ almost everywhere $t$ in $[0,T]$ or, equivalently, since $\po_{au}(T)= -1$, to prove that $\dd \po_{au}(t)=0$ almost everywhere $t$ in $[0,T]$. Using \eqref{bol4bis} and the definition of adjoint operator we obtain that
\begin{eqnarray*}
\dd \po_{au}(t) & = &\left\langle \dd \underline\po(t), \left(\begin{array}{c} \mathbf{0} \\ 1 \end{array}\right) \right\rangle = \left\langle -\Di_{\underline\co(t)}^*\underline{A}[\underline\po(t)], \left(\begin{array}{c} \mathbf{0} \\ 1 \end{array}\right) \right\rangle \\
&=& - \left\langle \underline\po(t), \Di_{\underline\co(t)}\underline{A}\left[\begin{array}{c} \mathbf{0} \\ 1 \end{array}\right] \right\rangle \stackrel{\eqref{DAau}}\equiv 0.
\end{eqnarray*}
Hence, by substituting $\po_{au}(t)=-1$ in \eqref{bol6} and \eqref{bol6bis}, the result follows from \eqref{defHam} and \eqref{DHdoppio}.
\end{proof}

\subsection{A generalization to  finite particle control problems}
\label{s2}

In this subsection we write the Pontryagin maximum principle for the Bolza problem associated to multi-agent systems. This is the vector version of Theorem~\ref{PMP}, where we pass from $N=1$ agent represented by $c\in C$ to $N > 1$ agents represented by $\cv\in C^N$, $N\in \mathbb{N}$. We refer to Section~\ref{s:example} for some model examples in leader-follower dynamics and for replicator-type models with entropy regularisation. We further remark that existence for control problems of the form~\eqref{FinMin}--\eqref{FinEQ} below has been discussed in~\cite{AAMS}, while specific model assumptions for the velocity field~$\Aa$ may be found in \cite{MS2020, ADEMS}.

Before stating the assumptions on the velocity field and the cost functions, we introduce some notation. We denote by $\cv$ a generic element $(c_1,\dots,c_N)\in C^N$ and by $\cv(t)$ a generic element $(c_1(t),\dots,c_N(t))\in L^1([0,T];C^N)$, by $\pv$ an element $(p_1,\dots,p_N)\in (E_C^*)^N$ and by $\pv(t)$ an element $(p_1(t),\dots,p_N(t))\in L^1([0,T];(E_C^*)^N)$, by $\uv$ an element $(u_1,\dots,u_N)\in U^N$ and by $\uv(t)$ an element $(u_1
(t),\dots,u_N(t))\in \U^N$ respectively.

We make the following assumptions on the velocity field~$\Aa$ and on the cost functional~$\Ll$, which are the finite-particle counterparts of {\bf(HAode)} and of {\bf (HLode)}, respectively.

\begin{tcolorbox}
{\bf (HAsym)} Let $\Aa\colon  [0,T] \times C^N \times C \times U \to E$ be an operator satisfying the following properties:
\begin{itemize}
\item[$(i)$] for every $c\in C$, $\cv \in C^N$, $t\in[0,T]$ and $u\in U$ it holds
$$
\Aa(t,\cv,c,u) = \Aa(t,\sigma(\cv),c,u) \quad \text{for every permutation }\sigma\colon C^N \to C^N;
$$
\item[$(ii)$] for every $R>0$ there exists a constant $L_R > 0$ such that for every $c,\tilde c\in B_{R,C}$, $t\in [0,T]$, $\cv,\tilde\cv \in B_{R,C}^N$ and $u\in U$
\begin{equation*}\label{LIPA}
\|\Aa(t,\cv,c,u)-\Aa(t,\tilde \cv,\tilde c,u)\|_E\leq L_R\left(\|c-\tilde c\|_E + W_1\left(\scriptstyle{\frac{1}{N}\sum_{i=1}^N \delta_{c_i},\frac{1}{N}\sum_{i=1}^N \delta_{\tilde{c}_i}}\right)\right);
\end{equation*}
\item[$(iii)$] there exists $M>0$ such that for every $c\in C$, $\cv \in C^N$, $t\in[0,T]$ and $u\in U$ it holds
\begin{equation*}\label{FinBnd}
\|\Aa(t,\cv,c,u)\|_E \leq  M\left( 1+ \|c\|_E + m_1\left(\scriptstyle{\frac{1}{N}\sum_{i=1}^N \delta_{c_i}} \right)\right);
\end{equation*}
\item[$(iv)$] for every $c\in C$, $\cv\in C^N$ and $u\in U$ the map $t\mapsto \Aa(t,\cv,c,u)$ belongs to $L^1([0,T];E)$;
\item[$(v)$] for every $t\in [0,T]$, $c\in C$ and $\cv\in C^N$ the map $u\mapsto \Aa(t,\cv,c,u)$ is such that
$$
\|\Aa(t,\cv,c,u)-\Aa(t,\cv,c,\tilde{u})\|_E \leq \omega(d_U(u,\tilde{u})),
$$
where $\omega \colon [0,+\infty) \to [0,+\infty)$ is not dependent on $t$ and with $\lim_{s\to 0^+} \omega(s)=0$;
\item[$(vi)$] for every $R>0$ there exists $\theta>0$ such that for every $t\in [0,T]$, $\cv \in B_{R,C}^N$ and $u\in U$
\begin{equation*}\label{BrezisCon}
c\in C, \|c\|_E \leq R \Rightarrow c+\theta \Aa(t,\cv,c,u)\in C;
\end{equation*}
\item[$(vii)$] for every $c\in C$, $\cv=(c_1,\dots,c,\dots,c_N) \in C^N$, $t\in [0,T]$ and $u\in U$, $\Aa$ is $C$-differentiable at $c\in C$, i.e. there exists $\Di_c\Aa\in \li(E_C;E_C)$
\begin{equation*} \label{FinDiff}
\lim_{C\ni c'\to c} \dis\frac{\|\Aa(t,\cv',c',u)-\Aa(t,\cv,c,u)-\Di_c\Aa[c'-c]\|_E}{\|c'-c\|_E} = 0.
\end{equation*}
\end{itemize}
\end{tcolorbox}

\begin{tcolorbox}
{\bf (HLsym)} Let $\Ll\colon [0,T]\times C^N \times U^N \to \R$ be a continuous function such that
\begin{itemize}
\item[$(a)$] for every $\cv \in C^N$, $t\in[0,T]$ and $\uv\in U^N$ it holds
$$
\Ll(t,\cv,\uv) = \Ll(t,\sigma(\cv),\iota(\uv)) \quad \text{for every permutations }\sigma\colon C^N \to C^N, \iota\colon U^N\to U^N;
$$
\item[$(b)$] for every $R>0$ there exists a constant $L_R > 0$ such that for every $\cv, \tilde \cv\in (B_R^C)^N$, $t\in[0,T]$ and $\uv\in U^N$
\begin{equation*}
|\Ll(t,\cv,\uv)-\Ll(t,\tilde\cv,\uv)|\leq L_R W_1\left(\scriptstyle{\frac{1}{N}\sum_{i=1}^N \delta_{c_i},\frac{1}{N}\sum_{i=1}^N \delta_{\tilde{c}_i}}\right);
\end{equation*}
\item[$(c)$] for every $t\in [0,T]$ and $\cv \in C^N$ the map $u\mapsto \Ll(t,\cv,u)$ is such that
$$
\|\Ll(t,\cv,\uv)-\Ll(t,\cv,\tilde{\uv})\|_E \leq \omega(d_{U^N}(\uv,\tilde{\uv})),
$$
where $\omega \colon [0,+\infty) \to [0,+\infty)$ is not dependent on $t$ and with $\lim_{s\to 0^+} \omega(s)=0$;
\item[$(d)$] for every $c\in C$, $\cv=(c_1,\dots,c,\dots,c_N) \in C^N$, $t\in [0,T]$ and $u\in U$, $\Ll$ is $C$-differentiable at $c\in C$ with $C$-differential $\Di_c\Ll\in \li(E_C;\R)$.
\end{itemize}
\end{tcolorbox}

We consider the following Bolza problem
\begin{equation}\label{FinMin}
\min_{\uv(t)\in\U^N}\left\{\int_0^T \Ll(t,\cv(t),\uv(t))\di t \right\}
\end{equation}
subject to the system
\begin{equation}\label{FinEQ}
\begin{cases}
\dd c_i(t)=\Aa(t,\cv(t),c_i(t),u_i(t)) & \text{ in }(0,T], \\
c_i(0)=c_{0,i}\in C,
\end{cases}
\qquad\text{for } i=1,\dots,N.
\end{equation}

Thanks to the assumptions on $\Aa$ and by Theorem I.4 of \cite{Brezis} for every $\uv\in \U^N$ there exists a unique weak solution $\cv \in AC([0,T];C^N)$ of system \eqref{FinEQ}.

Assume that the minimum control problem \eqref{FinMin}-\eqref{FinEQ} admits a solution $(\co  ,\uo)  \in AC([0,T];C^N) \times \U^N$. Then, after an application of Theorem \ref{PMP}, the following result follows.

\begin{theorem}
\label{PMPsym}
Let $\Aa$ and $\Ll$ satisfy {\bf (HAsym)} and {\bf (HLsym)} respectively. Let $\uo \in \mathcal{U}^{N}$ be an admissible control whose corresponding trajectory $\co \in AC([0, T]; C^{N})$ is optimal for problem \eqref{FinMin}-\eqref{FinEQ}. Then there exists a family of co-state curves $\po \in AC([0,T];(E_C^*)^N)$ such that $(\co,\po,\uo)$ solve in distributional sense the following system
\begin{equation*}\label{FinSys}
\begin{cases}
\dis\dd \left(\begin{array}{c} \co_i(t) \\ \po_i(t) \end{array}\right) = J\Di_{(\co_i(t),\po_i(t))}\mathcal H(t,\co (t),\po (t),\uo (t)) & \text{ in } [0,T), \\
\co_i(0) = c_{0,i} \in C, \\
\po_i(T) = \mathbf{0} \in E_C^*, \\
\dis \uo (t)\in \argmax_{\uv\in U^N} \mathcal H(t,\co (t),\po (t),\uv),
\end{cases}
\quad \text{for every }i=1,\dots,N,
\end{equation*}
where the Hamiltonian $\mathcal H\colon  [0,T]\times C^N \times  (E_C^*)^N \times U^N \to \R$ is defined by
\begin{equation*}
\label{FinHam}
\mathcal H(t,\cv,\pv,\uv)= \sum_{i=1}^N \langle p_i,\Aa(t,\cv,c_i,u_i)\rangle - \Ll(t,\cv,\uv).
\end{equation*}
\end{theorem}

\section{Pontryagin maximum principle in the Wasserstein space of convex spaces}
\label{s3}

We will henceforth need stronger assumptions than those in Section \ref{s1}. Let the following assumptions hold:
\begin{tcolorbox}
\label{Hpspaces}
\begin{itemize}
\item $(E, \|\cdot\|_E)$ is a separable Banach space;
\item $C\subset E$ is a closed and convex subset of $E$;
\item $E_C$ is reflexive;
\item $(Z, \|\cdot\|_Z)$ is a separable Banach space;
\item $\U=L^1([0,T];U)$ where $U$ is a not empty compact subset of the normed space
$$
(C^1_b(C;Z),\|\cdot\|_{C^1_b}):=\left\{u\in C^1(C;Z) : \|u\|_{C^1_b}:=\scriptstyle {\sup_{c\in C}\|u(c)\|_{Z} + \sup_{c\in C}\|\Di_c u\|_{\li(E_C;Z)}}<+\infty\right\},
$$
where $C^1(C;Z)$ means continuous $C$-differentiability.
\end{itemize}
\end{tcolorbox}

From now on we will often use for any $u\in\U$ the identification $u(t)(c)=u(t,c)$. \\

\begin{remark}
\label{Radon}
Since $C$ is closed in $E$ (which is Banach and separable) we deduce that $C$ with the induced metric is a complete separable metric space. This implies that $C$ is a separable Radon space (it is actually a Polish space). Similarly, since $E_C$ is reflexive and separable, then $E_C^*$ is a separable Banach space, and so a separable Radon space. It follows that $C\times E_C^*$ is a separable Radon space and that $\mathcal{P}_p(C)$, $\mathcal{P}_p(E_C^*)$ and $\mathcal{P}_p(C\times E_C^*)$ are complete separable metric space for every $p\geq 1$.
\end{remark}

We define
$$
\mathcal{P}_c(C):=\left\{\mu \in \mathcal{P}(C) : \supp(\mu) \text{ is compact in }C \right\}
$$
and
$$
\mathcal{P}_b(C):=\left\{\mu \in \mathcal{P}(C) : \supp(\mu) \text{ is bounded in }C \right\}.
$$
From now on, unless otherwise specified, when we write $\mathcal{P}_c(C)$ and $\mathcal{P}_b(C)$ we mean the separable metric spaces $\left(\mathcal{P}_c(C), W_1\right)$ and $\left(\mathcal{P}_b(C), W_1\right)$ respectively, both seen as subset of the separable complete metric space $\left(\mathcal{P}_1(C), W_1\right)$. \\

Note that $E_{C\times E_C^*}=\overline{\R(C\times E_C^*-C\times E_C^*)}=E_C \times E_C^*$ and $(E_C \times E_C^*)^*=E_C^* \times E_C^{**}=E_C^*\times E_C$ (recall that $E_C$ is reflexive). We define the linear and continuous operator
$$
\begin{array}{lll}
J\colon E_C^*\times E_C & \to & E_C\times E_C^*  \\
\phantom{J:}(T,e) & \mapsto & (e,-T)
\end{array}.
$$

We define the \emph{R-fattening} of the support of a measure $\mu\in\mathcal{P}_c(C)$ as
$$
B_\mu(R):= \displaystyle\bigcup_{c\in\supp(\mu)} B_{R,C}(c),
$$
where $B_{R,C}(c)$ is the closed subset $\{e\in E : \|e-c\|_E\leq R\}\cap C$. Note that, since $\mu$ has compact support, if $\eta\in\mathcal{P}(B_\mu(R))$ then $\eta \in \PBC$. \\

From now on let $\phi\colon \mathcal{P}_1(C)\to Y$ be such that $\mathcal{P}_b(C) \subset D(\phi):=\{\mu\in\mathcal{P}_1(C) : \|\phi(\mu)\|_Y <+\infty\}$, with $(Y,\|\cdot\|_Y)$ a Banach space. We introduce a definition of local differentiability of $\phi$ at $\mu\in \mathcal{P}_c(C)$.

\begin{definition}
\label{Wmudiff}
A functional $\phi \colon \mathcal{P}_{1} (C) \to Y$ is locally differentiable at $\mu\in \mathcal{P}_c(C)$ if there exists a map $\nabla_\mu \phi(\mu)\in L^2_{\mu}(C; \li(E_C;Y))$ in Bochner sense such that for every $R>0$ and for every $\nu\in \mathcal{P}(B_\mu(R))$ it holds
$$
\phi(\nu)-\phi(\mu)=\int_{C\times C} \nabla_\mu\phi(\mu)(c_1)[c_2-c_1] \di{\bm\gamma}(c_1,c_2) + o_R(W_{2,{\bm\gamma}}(\mu,\nu))
$$
for any ${\bm\gamma}\in \Gamma(\mu,\nu)$, where
$$
W_{2,{\bm\gamma}}^2(\mu,\nu):=\int_{C\times C} \|c_1-c_2\|_E^2 \di{\bm\gamma}(c_1,c_2).
$$
\end{definition}

This new notion of local differentiability for functional with values in Banach spaces enjoys the following chain rule.
\begin{proposition}
\label{Wchainrule}
Let $\mu\in\mathcal{P}_c(C)$ and $V$ be a closed and bounded subset of $C$ such that $\supp(\mu)\subset V$. Suppose that $\phi\colon \mathcal{P}(V)\to Y$ is locally differentiable over $\mathcal{P}_c(V)$. Let $\psi\colon C\to C$ be a $C$-differentiable map with $C$-differential $D_c\psi$ such that
\begin{itemize}
\item[$(i)$] $D\psi \in  L^{\infty}_{\mu}(C;\li(E_C;E_C))$ where $D\psi(c):=D_c\psi$;
\item[$(ii)$] $\supp(\psi_{\#}\mu)\subset V$.
\end{itemize}
Let $\overline\varepsilon >0 $ and $\mathcal{G}\colon (-\overline\varepsilon,\overline\varepsilon)\times V \to C$ be a map such that:
\begin{itemize}
\item[$(iii)$] $\mathcal{G}(0,\cdot)=Id_V$ and the map $\eps \mapsto \mathcal{G}(\eps,c)\in C\subset E$ is Fr\'echet-differentiable at $\eps=0$ uniformly in $V$;
\item[$(iv)$] the map $\mathcal{F}\colon  V\to E_C$ which is defined by $c\mapsto \mathcal{F}(c):=\frac{\di}{\di\varepsilon} \mathcal{G}(\varepsilon,c)\big|_{\eps=0}$ belongs to $L^{\infty}_{\mu}(V;E_C)$;
\item[$(v)$] $\supp((\psi \circ\mathcal{G}(\eps,\cdot))_{\#}\mu)\subset V$ for every $\eps\in(-\overline\varepsilon,\overline\eps)$.
\end{itemize}
Then the map $\eps\mapsto \phi((\psi \circ\mathcal{G}(\eps,\cdot))_{\#}\mu)$ is Fr\'echet-differentiable at $\eps=0$ and
\begin{equation}
\label{Wchrule}
\frac{\di}{\di\eps}\phi((\psi \circ\mathcal{G}(\eps,\cdot))_{\#}\mu)\big|_{\eps=0}=\int_C \nabla_{\psi_{\#}\mu}\phi(\psi_{\#}\mu)(\psi(c))\left[D_c\psi[\mathcal{F}(c)]\right]\di\mu(c).
\end{equation}
\end{proposition}

\begin{proof}
For notational convenience set $\psi_\eps:=(\psi \circ\mathcal{G}(\eps,\cdot))\colon C\to C$. First we observe that, by $(ii)$, since $\psi\colon C\to C$ is continuous and $\mu\in\PC$, then $\supp(\psi_{\#}\mu)=\psi(\supp(\mu))$ is compact in $V$ (see formula (5.2.6) of \cite{AGS}). Hence $\psi_{\#}\mu\in\mathcal{P}_c(V)$ and, by ({\em v}), $(\psi_{\eps})_{\#}\mu \in \mathcal{P}(B_\mu(R))$ for $R>0$ large enough. Then, using that ${\bm \mu}_\eps:=\left(\psi, \psi_\eps\right)_{\#}\mu\in \Gamma\left(\psi_{\#}\mu, (\psi_\eps)_{\#}\mu\right)$ and that $\phi$ is locally differentiable over $\mathcal{P}_c(V)$, from Definition \ref{Wmudiff} we obtain
\begin{eqnarray}
\label{ch1}
&&\frac{\phi((\psi_\eps)_{\#}\mu)-\phi(\psi_{\#}\mu)}{\eps} \nonumber \\
&& = \int_{C\times C} \nabla_{\psi_{\#}\mu}\phi(\psi_{\#}\mu)(c)\left[\frac{\tilde c-c}{\eps}\right] \di{\bm\mu}_\eps(c,\tilde c) + \frac{1}{\eps}o_R\left(W_{2,{\bm\mu}_\eps}(\psi_{\#}\mu,(\psi_\eps)_{\#}\mu)\right) \nonumber \\
&& = \int_C \nabla_{\psi_{\#}\mu}\phi(\psi_{\#}\mu)(\psi(c))\left[\frac{\psi_\eps(c)-\psi(c)}{\eps}\right]\di\mu(c) \nonumber \\
&& \phantom{ = }+ \frac{1}{\eps} o_R\left(\left(\int_C \left\|\psi_\eps(c)-\psi(c)\right\|^2_E \di\mu(c) \right)^\frac{1}{2}\right).
\end{eqnarray}
Using that $\mathcal{G}$ is Frechet differentiable at $\eps=0$ uniformly in $V$  and since $\psi$ is $C$-differentiable, we have
\begin{equation}
\label{ch2}
\psi_\eps(c)=\psi(\mathcal{G}(\eps,c))=\psi\left(c+\eps\mathcal{F}(c)+r_{\mathcal{F}}(\eps)\right)=\psi(c)+\eps D_c\psi[\mathcal{F}(c)] + r(\eps,c),
\end{equation}
where
$$
r(\eps,c)= D_c\psi[r_{\mathcal{F}}(\eps)] + o(\|\eps\mathcal{F}(c)+r_{\mathcal{F}}(\eps)\|)
$$
with $r_{\mathcal{F}}(\eps)=o(\eps)$ not depending on $c$. It follows from $(i)$ and $(iv)$ that
\begin{equation}
\label{ch3}
\|r(\eps,c)\|_{L^{\infty}_{\mu}(C;E_C)}\leq 2\|D\psi\|_{L^{\infty}_{\mu}(C;\li(E_C;E_C))} \|r_{\mathcal{F}}(\eps)\|_E=o(\eps).
\end{equation}
Now we can pass to the limit in \eqref{ch1} as $\eps\to 0$. Indeed, by \eqref{ch2}-\eqref{ch3}, we have
$$
\lim_{\eps\to 0}\frac{\psi_\eps(c)-\psi(c)}{\eps} = D_c\psi[\mathcal{F}(c)]  \quad \text{a.e. }c\in C
$$
and
$$
\left\|\frac{\psi_\eps(c)-\psi(c)}{\eps}\right\|_{L^{\infty}_{\mu}(C;E_C)} \leq 2 \|D\psi\|_{L^{\infty}_{\mu}(C;\li(E_C;E_C))} \|\mathcal{F}\|_{L^{\infty}_{\mu}(C;E_C)}.
$$
Hence we deduce that
\begin{equation*}
o_R\left(\left(\int_C \left\|\psi_\eps(c)-\psi(c)\right\|^2_E \di\mu(c) \right)^\frac{1}{2}\right) = o_R(\eps) \qquad \text{as }\eps\to 0.
\end{equation*}
Therefore, recalling that $\nabla_{\psi_{\#}\mu}\phi(\psi_{\#}\mu) \in L^2_{\mu}(C; \li(E_C;Y))$, after an application of the Lebesgue theorem we obtain \eqref{Wchrule}.
\end{proof}

We start by considering the following Mayer optimal control problem:
\begin{equation}\label{Wmaymin}
\min_{u\in \U} \varphi(\mu(T))
\end{equation}
subject to
\begin{equation}\label{Wmayeq}
\begin{cases}
\dd \mu(t) +\mathrm{div}\left(A(t,\mu(t),\cdot,u(t,\cdot))\mu(t)\right)=0 & \text{ in }(0,T], \\
\mu(0)=\mu^0\in \mathcal{P}_c(C).
\end{cases}
\end{equation}

We assume that the non-local velocity field $A\colon [0,T]\times \mathcal{P}_b(C)\times C\times C^0_b(C;Z) \to E_C $ satisfies the following.

\begin{tcolorbox}
{\bf (HA1): Assumptions on $A$ for the well-posedness of \eqref{Wmayeq}}
\begin{itemize}
\item[$(i)$] there exists a constant $L> 0$ such that for every $t\in [0,T]$, $\mu,\tilde\mu \in \mathcal{P}_c(C)$, $c,\tilde c\in C$ and $u,\tilde u\in C^0_b(C;Z)$
\begin{equation*}
\|A(t,\mu,c,u)-A(t,\tilde\mu,\tilde{c},\tilde u)\|_E\leq L\left(\|c-\tilde c\|_E + W_1(\mu,\tilde\mu) + \|u-\tilde u\|_{C^0_b}\right);
\end{equation*}
\item[$(ii)$] there exists $M>0$ such that for every $t\in[0,T]$, $\mu \in \mathcal{P}_c(C)$, $c\in C$ and $u\in U$ there holds
\begin{equation*}\label{WFinBnd}
\|A(t,\mu,c,u)\|_E \leq  M\left( 1+ \|c\|_E + m_1(\mu)\right);
\end{equation*}
\item[$(iii)$] for every $\mu\in\mathcal{P}_c(C)$, $c\in C$ and $u\in U$ the map $t\mapsto A(t,\mu,c,u)$ belongs to $L^1([0,T];E)$;
\item[$(iv)$] for every $R>0$ there exists $\theta>0$ such that for every $t\in [0,T]$, $\mu \in \mathcal{P}_c(C)$ and $u\in U$
\begin{equation*}
c\in C, \|c\|_E \leq R \Rightarrow c+\theta A(t,\mu,c,u)\in C.
\end{equation*}
\end{itemize}
\end{tcolorbox}
Under {\bf (HA1)}, it follows from Theorem 3.3 of \cite{AFMS} that,  for every $u \in \U$, there exists a unique solution $\mu\in AC([0,T]; \PC)$ of \eqref{Wmayeq}. The curve $\mu(t)$ can be represented as
\begin{equation}
\label{solrep}
\mu(t)=\evt \qquad \text{for }t\in [0,T],
\end{equation}
where, for $0\leq s\leq t\leq T$, $\Phi_{(s,t)}^{\mu(s)}\colon C \to C$ denotes the family of non-local flows defined by
\begin{equation}
\label{flow}
\Phi_{(s,t)}^{\mu^s}(c):= c + \int_s^t A\left(t, \mu(\sigma), \Phi_{(s,\sigma)}^{\mu^s}(c), u\left(\sigma, \Phi_{(s,\sigma)}^{\mu^s}(c) \right) \right) \di\sigma, \qquad \text{with }\mu^s:=\mu(s).
\end{equation}
Moreover, by Theorem 1 of \cite{AAMS}, for every $0\leq s\leq t\leq T$ we have
\begin{equation}
\label{flow-cont}
\Phi_{(s,t)}^{\mu^s} \in C_b^0(C;C).
\end{equation}

In order to find necessary conditions for the optimal control problem \eqref{Wmaymin}-\eqref{Wmayeq} we need further assumptions.

\begin{tcolorbox}
{\bf (HA2): Assumptions on $A$ for the well-posedness of \eqref{Wsist}}
\begin{itemize}
\item[$(v)$] for every $t\in [0,T]$, $\mu\in\mathcal{P}_c(C)$ and $c\in C$ the map $u\mapsto A(t,\mu,c,u)$ belongs to $C^1(C^1_b(C;Z);E)$;
\item[$(vi)$] for every $t\in [0,T]$, $\mu\in\mathcal{P}_c(C)$ and $u\in U$ the map $c\mapsto A(t,\mu,c,u(t,c))$ is $C$-differentiable with $C$-differential $\Di_cA$ and the application $\Di A\colon \PC \times C \times U \to \li(E_C;E_C)$ defined by $(\mu,c,u)\mapsto \Di_cA:=\Di_cA(t,\mu,c,u(t,c))$ is continuous for every $t\in[0,T]$;
\item[$(vii)$] for every $t\in [0,T]$, $c\in C$ and $u\in U$ the map $\PC \ni \mu \mapsto A(t,\mu,c,u)\in E_C$ is locally differentiable at any $\mu$ in the sense of the Definition \ref{Wmudiff} with differential $\nabla_\mu A:=\nabla_\mu A(t,\mu,c,u)$ and the application $\PC\times C\times U \times C \ni (\mu,c,u,\tilde c) \mapsto \nabla_\mu A(\tilde c) \in \li(E_C;E_C)$ is continuous for every $t\in[0,T]$.
\end{itemize}
\end{tcolorbox}

With an extended but similar argument to the one used to prove Lemma \ref{ext-bnd}, we can get a boundedness result for $\Das$. We state such result in the following lemma which will be proved in Appendix A.

\begin{lemma}
\label{ext-bnd-inf}
Under the assumptions {\bf (HA1)}-$(i)$ and {\bf (HA2)}-$(vi)$, it holds that
\begin{equation}
\label{Das-bnd}
\left\|\Das\right\|_{\li(E_C;E_C)} \leq L \qquad \text{for every } 0\leq s\leq t\leq T \text{ and } c\in C,
\end{equation}
where $L$ is a positive constant which only depends on the Lipschitz constant of~$A$ and on~$U$. In particular,
$$
\text{the map}\quad (t,c)\mapsto \Das \quad \text{belongs to } L^{\infty}_{\li \times \mu^s}\left([s,T]\times C; \li(E_C;E_C)\right).
$$
\end{lemma}

Thanks to Lemma \ref{ext-bnd-inf} we can deduce a regularity result for the family of non-local flows defined in \eqref{flow}. In the next lemma we show that this family is $C$-differentiable and characterize its $C$-differential. Since the techniques used in the proof are similar to those seen in Section \ref{s1}, we postpone the proof in Appendix A.

\begin{lemma}
\label{Dphi-bnd}
Under the assumptions {\bf (HA1)} and {\bf (HA2)}-$(vi)$, it holds that for every $0 \leq s \leq t \leq T$ the map $\Phi_{(s,t)}^{\mu^s}\colon C \to C$ is $C$-differentiable with $C$-differential $\Di_c \Phi_{(s,t)}^{\mu^s}$ at $c\in C$ and
$$
\text{the map}\quad (t,c) \mapsto\Di_{c}\Phi_{(s,t)}^{\mu^s} \quad \text{belongs to } L^\infty_{\li\times \mu^s}\left([s,T] \times C;\li(E_C;E_C)\right).
$$
Moreover, for every $c\in C$ and $f\in E_C$, $\Di_{c}\Phi_{(s,t)}^{\mu^s}[f]$ is the unique solution to the linear differential equation
\begin{equation}\label{Dphi-edo}
\begin{cases}
\dd z(t,c)=\Di_{\Phi_{(s,t)}^{\mu^s}(c)}A[z(t,c)]   & \text{ in }(s,T], \\
z(s,c)=f.
\end{cases}
\end{equation}
\end{lemma}

To adapt the needle variation technique to the infinite-dimensional case we need another lemma describing how the flow \eqref{flow} and the solution of \eqref{Wmayeq} behave when we do an infinitesimal variation of the control. Also in this case we postpone the proof to Appendix A. For simplicity of notation we give the result for initial time $s=0$.

\begin{lemma}
\label{conv-flussi-mis}
Assume {\bf (HA1)} and {\bf (HA2)}. Let $\eps>0$. Let $u,u_\eps \in \U$ be such that  $u_\eps \to u$ in $L^1([0,T];(U,\|\cdot\|_{C_b^1}))$ as $\eps \to 0$. Let $\mu^0,\mu^0_\eps \in \mathcal{P}_c(C)$ be such that $W_1(\mu^0_\eps,\mu^0) \to 0$ as $\eps \to 0$. Let $\mu$ be the solution of \eqref{Wmayeq} with initial datum $\mu^0$ and corresponding family of flows $\Phi_{(0,t)}^{\mu^0}$. Let $\mu_\eps$ be the solution of \eqref{Wmayeq} in which $u$ is replaced by~$u_\eps$ with initial datum~$\mu^0_\eps$ and corresponding family of flows $\Phi^{\eps,\mu^0_\eps}_{(0,t)}$. Then the following hold as $\eps \to 0$:
\begin{itemize}
\item[$(a)$] $\Phi^{\eps,\mu^0_\eps}_{(0,t)}\to \Phi^{\mu^0}_{(0,t)}$ in $C_b^0([0,T]\times C;C)$;
\item[$(b)$] $W_1(\mu_\eps(t),\mu(t))\to 0$ uniformly in $[0,T]$;
\item[$(c)$] $\Di_c\Phi^{\eps,\mu^0_\eps}_{(0,t)} \to \Di_c \Phi^{\mu^0}_{(0,t)}$ in $\li(E_C;E_C)$ for every $c\in C$ and $t\in [0,T]$.
\end{itemize}
\end{lemma}

As for the final cost $\varphi\colon \PC \to \R $ we assume:
\begin{tcolorbox}
\begin{itemize}
\item[{\bf (H$\varphi$):}]
$\varphi$ is Lipschitz continuous w.r.t. the $W_1$-metric and locally differentiable over $\PC$ with differential $\nabla_\mu\varphi$. Moreover, the map $C\ni c \mapsto \nabla_\mu \varphi(c)\in E_C^*$ is continuous.
\end{itemize}
\end{tcolorbox}

We define the Hamiltonian $\mathcal{H}\colon [0,T]\times \mathcal{P}_b(C\times E_C^*) \times U \to \R$ of the Mayer problem \eqref{Wmaymin}-\eqref{Wmayeq} as
\begin{equation}
\label{Wham}
\mathcal{H}(t,\nu,\omega)=\int_{C\times E_C^*} \langle p, A(t,\pi_{\#}^1\nu,c,\omega(c))\rangle \di \nu(c,p).
\end{equation}
Note that, by {\bf (HA1)}-$(ii)$, $\mathcal{H}(t,\nu,\omega)$ is finite for every $\nu \in \mathcal{P}_b(C\times E_C^*)$. Moreover, thanks to the assumptions {\bf (HA1)}-{\bf (HA2)} we can apply Lemma \ref{calc-diff-Ham} to obtain that $\mathcal H$ is locally differentiable at any $\nu\in \mathcal{P}_c(C\times E_C^*)$ in the sense of Definition \ref{Wmudiff}. The following explicit formula for its differential $\nabla_\nu \mathcal H(t,\nu,u)\colon  C \times E_C^* \to E_C^*\times E_C$ holds:
\begin{equation}
\label{form-diff-H}
\nabla_\nu \mathcal H(t,\nu,u)(c,p) = \left(\begin{array}{l} D_c^{*}A(t,\pi_{\#}^1\nu,c,u)[p] + \displaystyle\int_{C\times E_C^*}\nabla_{\pi_{\#}^1\nu}^*A(t,\pi_{\#}^1\nu,\tilde c,u)(c)[\tilde p] \di\nu(\tilde c,\tilde p) \\
A(t,\pi_{\#}^1\nu,c,u)
\end{array}\right),
\end{equation}
where $D_c^{*}A(t,\pi_{\#}^1\nu,c,u)$ is the adjoint operator of $D_cA(t,\pi_{\#}^1\nu,c,u)$ and $\nabla_{\pi_{\#}^1\nu}^*A(t,\pi_{\#}^1\nu,\tilde c,u)(c)$ is the adjoint operator of $\nabla_{\pi_{\#}^1\nu}A(t,\pi_{\#}^1\nu,\tilde c,u)(c)$. \\

We state the Pontryagin maximum principle for the infinite-dimensional Mayer problem \eqref{Wmaymin}--\eqref{Wmayeq}.
\begin{theorem}
\label{mainres}
Let $(\muo,\uo)\in AC([0,T];\mathcal{P}_c(C))\times\U$ be an optimal pair for \eqref{Wmaymin}-\eqref{Wmayeq}. Then there exists $\nuo\in AC([0,T]; \mathcal{P}_c(C\times E_C^*))$ which solves in distributional sense
\begin{equation}\label{Wsist}
\begin{cases}
\dis\dd \nuo(t) = - \mathrm{div}_{(c,p)}\left(\left(J\nabla_{\nuo(t)}\mathcal{H}(t,\nuo(t),\uo(t))(\cdot,\cdot)\right)\nuo(t)\right)   & \text{ in } [0,T), \\
\pi^1_{\#}\nuo(t)=\muo(t), \\
\nuo(T)= \left(Id , -\nabla_{\muo(T)}\varphi(\muo(T))\right)_{\#}\muo(T),
\end{cases}
\end{equation}
where $\mathcal H$ and $\nabla_\nu \mathcal H$ are defined by \eqref{Wham} and \eqref{form-diff-H} respectively. Moreover the following maximality condition holds
\begin{equation}
\label{Wmaxcond}
\mathcal{H}(t,\nuo(t),\uo(t))=\max_{\omega\in U}\{\mathcal{H}(t,\nuo(t),\omega)\} \qquad \text{for a.e. }t\in [0,T].
\end{equation}
\end{theorem}

In what follows we are concerned with the proof of Theorem \ref{mainres}. \\
First we give a boundedness result for $\nabla_\mu A$ which will be useful to derive the linearised non-local equation.
\begin{lemma}
\label{grad-mu-bnd}
Let $A$ satisfy {\bf (HA1)}-$(i)$ and {\bf (HA2)}-$(vii)$. Then
$$
\|\nabla_\mu A(\tilde c)\|_{\li(E_C;E_C)} \leq L
$$
for every $(\mu,c,u)\in \mathcal{P}_c(C) \times C \times \G$ and a.e. $t\in[0,T]$ and $\mu$-a.e. $\tilde c\in \supp(\mu)$, where $L$ is the Lipschitz constant of $A$. In particular,
$$
\nabla_\mu A\in L^\infty_{\mathcal L\times \mu \times \mu}([0,T]\times C\times \supp(\mu);\li(E_C;E_C)).
$$
\end{lemma}

\begin{proof}
Let $\mu\in \PC$, $c\in C$ and $u\in \G$. Fix $c_0\in \supp(\mu)$ and $c_1,c_2,\bar{c}\in C$. We define for $r>0$, $0<\eps<1$ and $i=1,2$
\begin{equation}
\label{Wapp22bis}
\varphi_\eps^i(\tilde c):=\begin{cases}
c_i & \text{ for } \tilde c\in B_r(c_0)\cap C,\\
{\left(1-\tfrac{\|\tilde c\|-r}{\eps}\right)}c_i + \left(\tfrac{\|\tilde c\|-r}{\eps}\right)\bar{c} & \text{ for }\tilde c\in \left(B_{r+\eps}(c_0)\setminus B_r(c_0)\right)\cap C, \\
\bar{c} & \text{ otherwise}.
\end{cases}
\end{equation}
Then $\varphi_\eps^{i}$ are continuous functions from $C$ to $C$ for $i=1,2$. Since $C$ is convex, we deduce that $(1-\eps)Id+\eps\varphi_\eps^{i}\colon C \to C$, and therefore $\mu_\eps^i:=\left((1-\eps)Id+\eps\varphi_\eps^i\right)_{\#}\mu$ are in $\PC$. It follows choosing in Definition \ref{Wmudiff}
$$
\bm\gamma_i= \left((1-\eps)Id+\eps\varphi_\eps^i ,  Id\right)_{\#}\mu \in \Gamma(\mu_\eps^i,\mu)
$$
that
\begin{eqnarray*}
 && A(t,\mu_\eps^2,c,u)-A(t,\mu_\eps^1,c,u) = A(t,\mu_\eps^2,c,u)-A(t,\mu,c,u) -\left(A(t,\mu_\eps^1,c,u)-A(t,\mu,c,u)\right) \\
&& = \eps\int_C \nabla_{\mu}A{\scriptstyle\left(t,\mu,c,u\right)}(\tilde c)\left[\varphi_\eps^2(\tilde c)- \varphi_\eps^1(\tilde c) \right] \di \mu(\tilde c) \nonumber \\
&& + o_R\left(\eps\left(\int_C \|\varphi_\eps^2(\tilde c)-\tilde c\|_E^2 \di \mu(\tilde c)\right)^\frac12\right) + o_R\left(\eps\left(\int_C \|\varphi_\eps^1(\tilde c)-\tilde c\|_E^2 \di \mu(\tilde c)\right)^\frac12\right).
\end{eqnarray*}
Hence, by \eqref{Wapp22bis}, using {\bf(HA1)}-$(i)$ and recalling that $\supp(\mu)$ is compact in $C$, we have
\begin{eqnarray}
\label{Wapp22}
&& \left\| \int_{B_{r+\eps}(c_0)\cap C} \nabla_{\mu}A{\scriptstyle\left(t,\mu,c,u\right)}(\tilde c)\left[\varphi_\eps^2(\tilde c)-\varphi_\eps^1(\tilde c) \right] \di \mu(\tilde c) \right\|_E \nonumber \\
&& \leq \frac{1}{\eps}\left\| A(t,\mu_\eps^2,c,u)-A(t,\mu_\eps^1,c,u) \right\|_E + \left\|\frac{1}{\eps}o_R\left(M\eps\right)\right\|_E \nonumber \\
&& \leq \frac{L}{\eps} W_1(\mu_\eps^1,\mu_\eps^2) + \left\|\frac{1}{\eps}o_R\left(M\eps\right)\right\|_E \nonumber \\
&& \leq L \|c_1-c_2\|_E \mu\left(B_{r+\eps}(c_0)\cap C\right) + \left\|\frac{1}{\eps}o_R\left(M\eps\right)\right\|_E,
\end{eqnarray}
where $M$ is a positive constant dependent on $\bar c, c_1, c_2$ and $\supp(\mu)$.
Now, by Definition~\ref{Wmudiff}, $\nabla_{\mu}A{\scriptstyle\left(t,\mu,c,u\right)}\in L_\mu^2(C;\li(E_C;E_C))$. Moreover,
$$
\|\varphi_\eps^2(\tilde c)-\varphi_\eps^1(\tilde c)\|_E\leq \left(\|c_1\|_E+\|c_2\|_E+\|\bar c\|_E\right)\in L^2_\mu(C;E_C)
$$
and
$$
\varphi_\eps^2(\tilde c)-\varphi_\eps^1(\tilde c) \to (c_2-c_1)\chi_{\{B_r(c_0)\cap C\}} \quad\text{pointwise in }E.
$$
Therefore, after an application of the Lebesgue theorem for general measures, we infer that as $\eps \to 0^+$
\begin{equation}
\label{Wapp23}
\left\| \frac{1}{\mu(B_{r}(c_0)\cap C)}\int_{B_{r}(c_0)\cap C} \nabla_{\mu}A{\scriptstyle\left(t,\mu,c,u\right)}(\tilde c)\left[c_2-c_1\right] \di \mu(\tilde c) \right\|_E \leq L\|c_2-c_1\|_E.
\end{equation}
Finally, by {\bf (HA2)}-$(vii)$, $\nabla_{\mu}A{\scriptstyle\left(t,\mu,c,u\right)}\in C^0(C ;\li(E_C;E_C))$ for a.e. $t\in[0,T]$. Then, applying Lebesgue differentiation theorem to \eqref{Wapp23}, we deduce as $r\to 0^+$ that
$$
\|\nabla_{\mu}A{\scriptstyle\left(t,\mu,c,u\right)}(c_0)\left[c_2-c_1\right]\|_E \leq L\|c_2-c_1\|_E
$$
for any $c,c_1,c_2\in C, \mu\in \PC, u\in \G$ and $c_0\in\supp(\mu)$. Thanks to the arbitrariness of $c_2-c_1\in E_C$, the result follows.
\end{proof}

Now we derive the expression of the linearised Cauchy problem associated to the non-local continuity equation \eqref{Wmayeq}.

\begin{proposition}
\label{linearized}
Let $V$ be a closed and bounded subset of $C$ and $\mu^0 \in \mathcal{P}_c(V)$. Assume {\bf (HA1)}-{\bf (HA2)}. Fix $u\in\U$ and let $\mu(t)$ and $\Phi_{(0,t)}^{\mu^0}$ be the solution of \eqref{Wmayeq} and the associated non-local flow defined in \eqref{flow} respectively. \\
Let $\overline\varepsilon >0 $ and $\mathcal{G}\colon (-\overline\varepsilon,\overline\varepsilon)\times V \to C$ be a continuous map such that:
\begin{itemize}
\item[$(i)$] $\mathcal{G}(0,\cdot)=Id_V$ and the map $\eps \mapsto \mathcal{G}(\eps,c)\in C\subset E$ is Fr\'echet-differentiable at $\eps=0$ uniformly in $V$;
\item[$(ii)$] the map $\mathcal{F}\colon  V\to E_C$ which is defined by $c\mapsto \mathcal{F}(c):=\frac{\di}{\di\varepsilon} \mathcal{G}(\varepsilon,c)\big|_{\eps=0}$ belongs to $L^{\infty}_{\mu^0}(V;E_C)$;
\item[$(iii)$] $\supp((\Phi_{(0,t)}^{\mu^0} \circ\mathcal{G}(\eps,\cdot))_{\#}\mu^0)\subset V$ for every $\eps\in(-\overline\varepsilon,\overline\eps)$.
\end{itemize}
Set $\mu_\eps^0:=\mathcal{G}(\eps,\cdot)_{\#}\mu^0$. Then, the map $(-\overline{\varepsilon},\overline\varepsilon)\ni \varepsilon \mapsto \Phi_{(0,t)}^{\mu_\eps^0}(c)\in C$ is Fr\'echet-differentiable at $\eps=0$ for all $(t,c)\in [0,T] \times V$. Moreover, its differential $v(t,c)$ is continuous with respect to $c$ and it is the unique solution of the non-local Cauchy problem
\begin{equation}
\label{lineq}
\begin{cases}
\frac{\di}{\di t} v(t,c) = \Di_{\Phi_{(0,t)}^{\mu^0}(c)} A {\scriptstyle \left(t,\mu(t),\Phi_{(0,t)}^{\mu^0}(c),u\left(t,\Phi_{(0,t)}^{\mu^0}(c)\right)\right)}[v(t,c)] \\
\,\,\,\,\,\,\,\,+\displaystyle \int_C \nabla_{\mu(t)}A{\scriptstyle \left(t,\mu(t),\Phi_{(0,t)}^{\mu^0}(c),u\left(t,\Phi_{(0,t)}^{\mu^0}(c)\right)\right)}\left(\Phi_{(0,t)}^{\mu^0}(\tilde c)\right)\left[\Di_{\tilde c} \Phi_{(0,t)}^{\mu^0}(\tilde c)\mathcal{F}(\tilde c) + v(t,\tilde c) \right] \di\mu^0(\tilde c),  \\
v(0,c)=0.
\end{cases}
\end{equation}
\end{proposition}

\begin{proof}
For notational convenience we define $\mathcal{G}^\varepsilon(c):=\mathcal{G}(\varepsilon,c)$ for every $(\eps,c)\in (-\overline\varepsilon,\overline\varepsilon)\times V$ and $\mu_\eps(t):=\left(\Phi_{(0,t)}^{\mu^0_\eps}\right)_{\#}\mu^0_\eps$. Note that $\mu_\eps(t)$ is the unique solution of \eqref{Wmayeq} with initial datum $\mu^0_\eps$. Moreover, by {\bf (HA2)}, we can apply Lemma \ref{ODE-continua} obtaining that there exists a unique distributional solution $v(t,c)$ of \eqref{lineq} which is continuous with respect to $c$. \\
Let $v_\eps\colon  [0,T]\times V \to E_C$ be
\begin{equation}
\label{Wapp3}
v_\eps(t,c):= \frac{\Phi_{(0,t)}^{\mu_\eps^0}(c) - \Phi_{(0,t)}^{\mu^0}(c)}{\eps} \qquad \mbox{for } \eps \in (-\overline\varepsilon,\overline\varepsilon).
\end{equation}
Then the statement is equivalent to
\begin{equation}
\label{goalin}
\lim_{\eps\to 0} v_\eps(t,c)=v(t,c) \qquad \mbox{for any }(t,c)\in [0,T] \times V.
\end{equation}
To prove \eqref{goalin} we divide the proof into 4 steps. \\

{\bf Step 1.} In this step we prove that $v_\eps(t,c)$ is uniformly bounded in $[0,T] \times V$ for $|\eps|$ small enough. \\
By \eqref{flow} we have for a positive constant $L_V$ only dependent on $V$ that
\begin{eqnarray*}
\|v_\eps(t,c)\|_E &\leq & \frac{1}{|\eps|}\int_0^t \left\|A{\scriptstyle \left(s, \mu_\eps(s), \Phi_{(0,s)}^{\mu_\eps^0}(c), u\left(s, \Phi_{(0,s)}^{\mu_\eps^0}(c)\right)\right)}-A{\scriptstyle\left(s, \mu(s), \Phi_{(0,s)}^{\mu^0}(c), u\left(s, \Phi_{(0,s)}^{\mu^0}(c)\right) \right)}\right\|_E \di s \nonumber \\
& \stackrel{{(\bf HA1)}-(i)}\leq & \frac{L}{|\eps|} \int_0^t \left\{ \left\|\Phi_{(0,s)}^{\mu_\eps^0}(c) - \Phi_{(0,s)}^{\mu^0}(c)\right\|_E + W_1\left(\mu_\eps(s),\mu(s)\right) \right\}\di s \nonumber \\
& \stackrel{\eqref{solrep}}\leq & L \int_0^t \|v_\eps(s,c)\|_E \, \di s + \frac{L}{|\eps|}\int_0^t W_1\left(\left(\Phi_{(0,s)}^{\mu_\eps^0}\right)_{\#}\mu_\eps^0,\left(\Phi_{(0,s)}^{\mu^0}\right)_{\#}\mu^0\right) \di s \\
& \leq & L \int_0^t \|v_\eps(s,c)\|_E \, \di s + \frac{L W_1\left(\mu_\eps^0,\mu^0 \right)}{|\eps|}\int_0^t \ex^{2L_V s} \di s.
\end{eqnarray*}
where in the last inequality we have applied Theorem 4.1 of \cite{AFMS}. Therefore we deduce that
\begin{equation}
\label{Wapp1}
\|v_\eps(t,c)\|_E \leq L \int_0^t \|v_\eps(s,c)\|_E \, \di s \, +\, \left(\ex^{2LT}-1\right)\frac{W_1\left(\mu_\eps^0,\mu^0 \right)}{2|\eps|}.
\end{equation}
To estimate the second term on the right-hand side of \eqref{Wapp1}, we use the definition of $1$-Wasserstain distance recalling that $(\mathcal{G}^\eps , Id)_{\#}\mu^0\in\Gamma(\mu_\eps^0,\mu^0)$ to obtain
\begin{eqnarray*}
W_1\left(\mu_\eps^0,\mu^0 \right) & = & W_1\left(\mathcal{G}^\eps_{\#}\mu^0,\mu^0 \right) \leq \int_V \left\|\mathcal{G}(\eps,\tilde c)-\tilde c\right\|_E \di\mu^0(\tilde c) \\
& \stackrel{(i)-(ii)}= &\int_V \|\eps\mathcal{F}(\tilde c) +o(\eps)\|_E \di \mu^0(\tilde c) \stackrel{(ii)}\leq 2|\eps| \|\mathcal{F}\|_{L^{\infty}_{\mu^0}(V;E_C)}.
\end{eqnarray*}
Hence, in view of \eqref{Wapp1}, we infer that
$$
\|v_\eps(t,c)\|_E \leq L \int_0^t \|v_\eps(s,c)\|_E \, \di s \, +\,\left(\ex^{2LT}-1\right)\|\mathcal{F}\|_{L^{\infty}_{\mu^0}(V;E_C)},
$$
which implies, applying the Gr\"onwall inequality, that for $|\eps|$ small enough
\begin{equation}
\label{Wapp2}
\|v_\eps(t,c)\|_E \leq \ex^{3LT}\|\mathcal{F}\|_{L^{\infty}_{\mu^0}(V;E_C)} \qquad \mbox{uniformly in } [0,T]\times V.
\end{equation}

{\bf Step 2.} In this step we explicitly write $v_\eps(t,c)-v(t,c)$. \\
By \eqref{Wapp3} and \eqref{lineq}, we have, noting that $v_\eps(0,c)=0$, that for every $(t,c)\in [0,T]\times V$
\begin{eqnarray}
\label{Wapp4}
&&v_\eps(t,c) - v(t,c) = \int_0^t \Di_{\Phi_{(0,s)}^{\mu^0}(c)} A {\scriptstyle \left(s,\mu(s),\Phi_{(0,s)}^{\mu^0}(c),u\left(s,\Phi_{(0,s)}^{\mu^0}(c)\right)\right)}[v_\eps(s,c)-v(s,c)] \di s \nonumber \\
&&\,\,\,\,\,\,\,\,+ \int_0^t \!\!\int_C \nabla_{\mu(s)}A{\scriptstyle \left(s,\mu(s),\Phi_{(0,s)}^{\mu^0}(c),u\left(s,\Phi_{(0,s)}^{\mu^0}(c)\right)\right)}\left(\Phi_{(0,s)}^{\mu^0}(\tilde c)\right)\left[v_\eps(s,\tilde c) - v(s,\tilde c) \right] \di\mu^0(\tilde c) \di s \nonumber \\
&&\,\,\,\,\,\,\,\, + \int_0^t r_\eps(s,c) \di s,
\end{eqnarray}
where
\begin{eqnarray}
\label{Wapp5}
&&r_\eps(t,c)  = \frac{1}{\eps}\left\{A{\scriptstyle \left(t, \mu_\eps(t), \Phi_{(0,t)}^{\mu_\eps^0}(c), u\left(t, \Phi_{(0,t)}^{\mu_\eps^0}(c)\right)\right)}-A{\scriptstyle\left(t, \mu(t), \Phi_{(0,t)}^{\mu^0}(c), u\left(t, \Phi_{(0,t)}^{\mu^0}(c)\right) \right)}\right\} \nonumber \\
&& - \Di_{\Phi_{(0,t)}^{\mu^0}(c)} A {\scriptstyle \left(t,\mu(t),\Phi_{(0,t)}^{\mu^0}(c),u\left(t,\Phi_{(0,t)}^{\mu^0}(c)\right)\right)}[v_\eps(t,c)]  \\
&& - \int_C \nabla_{\mu(t)}A{\scriptstyle \left(t,\mu(t),\Phi_{(0,t)}^{\mu^0}(c),u\left(t,\Phi_{(0,t)}^{\mu^0}(c)\right)\right)}\left(\Phi_{(0,t)}^{\mu^0}(\tilde c)\right)\left[\Di_{\tilde c} \Phi_{(0,t)}^{\mu^0}(\tilde c)\mathcal{F}(\tilde c) +v_\eps(s,\tilde c)\right] \di\mu^0(\tilde c). \nonumber
\end{eqnarray}
By adding and subtracting to \eqref{Wapp5} the term
$$
A{\scriptstyle \left(t, \mu(t), \Phi_{(0,t)}^{\mu_\eps^0}(c), u\left(t, \Phi_{(0,t)}^{\mu_\eps^0}(c)\right)\right)} + A{\scriptstyle \left(t, \left(\Phi_{(0,t)}^{\mu^0}\circ \mathcal{G}^\eps\right)_{\#}\mu^0, \Phi_{(0,t)}^{\mu_\eps^0}(c), u\left(t, \Phi_{(0,t)}^{\mu_\eps^0}(c)\right)\right)}
$$
we obtain
\begin{equation}
\label{Wapp6}
\|r_\eps(t,c)\|_E \leq \mbox{I} + \mbox{II} + \mbox{III},
\end{equation}
where
\begin{eqnarray}
\label{Wapp7}
\mbox{I} &:= &\left\|\frac{1}{\eps}\left\{A{\scriptstyle \left(t, \mu(t), \Phi_{(0,t)}^{\mu_\eps^0}(c), u\left(t, \Phi_{(0,t)}^{\mu_\eps^0}(c)\right)\right)}-A{\scriptstyle\left(t, \mu(t), \Phi_{(0,t)}^{\mu^0}(c), u\left(t, \Phi_{(0,t)}^{\mu^0}(c)\right) \right)}\right\}\right. \nonumber \\
&& \,\,\, - \left.\Di_{\Phi_{(0,t)}^{\mu^0}(c)} A {\scriptstyle \left(t,\mu(t),\Phi_{(0,t)}^{\mu^0}(c),u\left(t,\Phi_{(0,t)}^{\mu^0}(c)\right)\right)}[v_\eps(t,c)]
      \right\|_E,
\end{eqnarray}
\begin{eqnarray}
\label{Wapp8}
\mbox{II} &:= & \left\|\frac{1}{\eps}\left\{A{\scriptstyle \left(t, \left(\Phi_{(0,t)}^{\mu^0}\circ \mathcal{G}^\eps\right)_{\#}\mu^0, \Phi_{(0,t)}^{\mu_\eps^0}(c), u\left(t, \Phi_{(0,t)}^{\mu_\eps^0}(c)\right)\right)} - A{\scriptstyle \left(t, \mu(t), \Phi_{(0,t)}^{\mu_\eps^0}(c), u\left(t, \Phi_{(0,t)}^{\mu_\eps^0}(c)\right)\right)}\right\}\right. \nonumber \\
&& \!\!\!\!\!\! - \left.\int_C\nabla_{\mu(t)}A{\scriptstyle\left(t,\mu(t),\Phi_{(0,t)}^{\mu^0}(c),u\left(t,\Phi_{(0,t)}^{\mu^0}(c)\right)\right)}\left(\Phi_{(0,t)}^{\mu^0}(\tilde c)\right)\left[\Di_{\tilde c}\Phi_{(0,t)}^{\mu^0}(\tilde c)\mathcal{F}(\tilde c)\right] \di\mu^0(\tilde c) \right\|_E,
\end{eqnarray}
and
\begin{eqnarray}
\label{Wapp9}
\mbox{III} &:= & \left\|\frac{1}{\eps}\left\{A{\scriptstyle \left(t, \mu_\eps(t), \Phi_{(0,t)}^{\mu_\eps^0}(c), u\left(t, \Phi_{(0,t)}^{\mu_\eps^0}(c)\right)\right)}-A{\scriptstyle \left(t, \left(\Phi_{(0,t)}^{\mu^0}\circ \mathcal{G}^\eps\right)_{\#}\mu^0, \Phi_{(0,t)}^{\mu_\eps^0}(c), u\left(t, \Phi_{(0,t)}^{\mu_\eps^0}(c)\right)\right)}\right\}\right. \nonumber \\
&& \,\,\, - \left.\int_C \nabla_{\mu(t)}A{\scriptstyle \left(t,\mu(t),\Phi_{(0,t)}^{\mu^0}(c),u\left(t,\Phi_{(0,t)}^{\mu^0}(c)\right)\right)}\left(\Phi_{(0,t)}^{\mu^0}(\tilde c)\right)\left[v_\eps(t,\tilde c)\right] \di\mu^0(\tilde c) \right\|_E.
\end{eqnarray}

{\bf Step 3.} In this step we prove that $r_\eps(t,c)\to 0$ as $\eps \to 0$ in $E$ for any $(t,c)\in[0,T]\times V$. \\
To estimate I, we rewrite the right-hand side of \eqref{Wapp7} in the following way
\begin{eqnarray*}
I &=& \frac{1}{\|\Phi_{(0,t)}^{\mu_\eps^0}(c)-\Phi_{(0,t)}^{\mu^0}(c)\|_E}\left\{\left\|A{\scriptstyle \left(t, \mu(t), \Phi_{(0,t)}^{\mu_\eps^0}(c), u\left(t, \Phi_{(0,t)}^{\mu_\eps^0}(c)\right)\right)}-A{\scriptstyle\left(t, \mu(t), \Phi_{(0,t)}^{\mu^0}(c), u\left(t, \Phi_{(0,t)}^{\mu^0}(c)\right) \right)}\right.\right. \nonumber \\
&& \,\,\, - \left.\left.\Di_{\Phi_{(0,t)}^{\mu^0}(c)} A {\scriptstyle \left(t,\mu(t),\Phi_{(0,t)}^{\mu^0}(c),u\left(t,\Phi_{(0,t)}^{\mu^0}(c)\right)\right)}[\Phi_{(0,t)}^{\mu_\eps^0}(c)-\Phi_{(0,t)}^{\mu^0}(c)]
      \right\|_E\right\}\|v_\eps(t,c)\|_E,
\end{eqnarray*}
then, it follows from {\bf (HA2)}-$(vi)$ and using \eqref{Wapp2} that $I\to 0$ as $\eps \to 0$ for any $(t,c)\in [0,T]\times V$.

The fact of II $\to 0$ as $\eps \to 0$ follows from Lemma \ref{Dphi-bnd} and assumption {\bf (HA2)}-$(vii)$ which allow us to apply Proposition \ref{Wchainrule} in \eqref{Wapp8} with $\phi(\mu)=A(t,\mu,c,u)$ and $\psi=\Phi_{(0,t)}^{\mu^0}$. Thus we have
\begin{equation}
\label{Wapp20bis}
\lim_{\eps\to 0} \left\{\mbox{I}+\mbox{II}\right\}=0.
\end{equation}

The estimate of term III is trickier. By {\bf (HA2)}-$(vii)$, recalling Definition \ref{Wmudiff} and that $\mu_\eps(t)=\left(\Phi_{(0,t)}^{\mu_\eps^0}\circ \mathcal{G}^\eps\right)_{\#}\mu^0$, we have
\begin{align}
\label{Wapp10}
& A{\scriptstyle \left(t,\mu_\eps(t), \Phi_{(0,t)}^{\mu_\eps^0}(c), u\left(t, \Phi_{(0,t)}^{\mu_\eps^0}(c)\right)\right)} = A{\scriptstyle \left(t, \left(\Phi_{(0,t)}^{\mu^0}\circ \mathcal{G}^\eps\right)_{\#}\mu^0, \Phi_{(0,t)}^{\mu_\eps^0}(c), u\left(t, \Phi_{(0,t)}^{\mu_\eps^0}(c)\right)\right)} \nonumber \\
& \,\,\, + \int_{C\times C} \nabla_{\left(\Phi_{(0,t)}^{\mu^0}\circ \mathcal{G}^\eps\right)_{\#}\mu^0}A{\scriptstyle \left(t, \left(\Phi_{(0,t)}^{\mu^0}\circ \mathcal{G}^\eps\right)_{\#}\mu^0, \Phi_{(0,t)}^{\mu_\eps^0}(c), u\left(t, \Phi_{(0,t)}^{\mu_\eps^0}(c)\right)\right)}       (c_1)[c_2-c_1] \di{\bm\gamma}(c_1,c_2) \nonumber \\
& \,\,\, + o\left(W_{2,{\bm\gamma}}\left(\left(\Phi_{(0,t)}^{\mu_\eps^0}\circ \mathcal{G}^\eps\right)_{\#}\mu^0, \left(\Phi_{(0,t)}^{\mu^0}\circ \mathcal{G}^\eps\right)_{\#}\mu^0    \right)\right) \nonumber \\
&\stackrel{\eqref{Wapp3}}= A{\scriptstyle \left(t, \left(\Phi_{(0,t)}^{\mu^0}\circ \mathcal{G}^\eps\right)_{\#}\mu^0, \Phi_{(0,t)}^{\mu_\eps^0}(c), u\left(t, \Phi_{(0,t)}^{\mu_\eps^0}(c)\right)\right)} \nonumber \\
&\,\,\, + \eps\int_C \nabla_{\left(\Phi_{(0,t)}^{\mu^0}\circ \mathcal{G}^\eps\right)_{\#}\mu^0}A{\scriptstyle \left(t, \left(\Phi_{(0,t)}^{\mu^0}\circ \mathcal{G}^\eps\right)_{\#}\mu^0, \Phi_{(0,t)}^{\mu_\eps^0}(c), u\left(t, \Phi_{(0,t)}^{\mu_\eps^0}(c)\right)\right)}\left(\Phi_{(0,t)}^{\mu^0}(\mathcal{G}^\eps(\tilde c))\right)\left[v_\eps(t,\mathcal{G}^\eps(\tilde c))\right] \di \mu^0(\tilde c) \nonumber \\
&\,\,\, + o\left(|\eps|\left(\int_C \|v_\eps(t,\mathcal{G}^\eps(\tilde c))\|_E^2 \di \mu^0(\tilde c)\right)^{\frac12}\right),
\end{align}
where in the last equality we have chosen $\bm\gamma=\left(\Phi_{(0,t)}^{\mu_\eps^0}\circ \mathcal{G}^\eps , \Phi_{(0,t)}^{\mu^0}\circ \mathcal{G}^\eps  \right)_{\#}\mu^0$. Combining~\eqref{Wapp9} and~\eqref{Wapp10}, we deduce
\begin{align}
\label{Wapp11}
\nonumber\mbox{III} & \stackrel{\eqref{solrep}}= \left\|\int_C \nabla_{\left(\Phi_{(0,t)}^{\mu^0}\circ \mathcal{G}^\eps\right)_{\#}\mu^0}A{\scriptstyle \left(t, \left(\Phi_{(0,t)}^{\mu^0}\circ \mathcal{G}^\eps\right)_{\#}\mu^0, \Phi_{(0,t)}^{\mu_\eps^0}(c), u\left(t, \Phi_{(0,t)}^{\mu_\eps^0}(c)\right)\right)}\left(\Phi_{(0,t)}^{\mu^0}(\mathcal{G}^\eps(\tilde c))\right)\left[v_\eps(t,\mathcal{G}^\eps(\tilde c))\right] \di \mu^0(\tilde c) \right. \nonumber \\
\nonumber& \,\,\, -\int_C\nabla_{\left(\Phi_{(0,t)}^{\mu^0}\right)_{\#}\mu^0}A{\scriptstyle\left(t,\left(\Phi_{(0,t)}^{\mu^0}\right)_{\#}\mu^0,\Phi_{(0,t)}^{\mu^0}(c),u\left(t,\Phi_{(0,t)}^{\mu^0}(c)\right)\right)}\left(\Phi_{(0,t)}^{\mu^0}(\tilde c)\right)\left[v_\eps(t,\tilde c)\right] \di\mu^0(\tilde c) \nonumber \\
& \,\,\, \left.+\frac{1}{\eps}o\left(|\eps|\left(\int_C \|v_\eps(t,\mathcal{G}^\eps(\tilde c))\|_E^2 \di \mu^0(\tilde c)\right)^{\frac12}\right) \right\|_E.
\end{align}
By adding and subtracting to the terms within the norm in \eqref{Wapp11} the term
$$
\int_C\nabla_{\left(\Phi_{(0,t)}^{\mu^0}\right)_{\#}\mu^0}A{\scriptstyle\left(t,\left(\Phi_{(0,t)}^{\mu^0}\right)_{\#}\mu^0,\Phi_{(0,t)}^{\mu^0}(c),u\left(t,\Phi_{(0,t)}^{\mu^0}(c)\right)\right)}\left(\Phi_{(0,t)}^{\mu^0}(\tilde c)\right)\left[v_\eps(t,\mathcal{G}^\eps(\tilde c))\right] \di\mu^0(\tilde c)
$$
we obtain
\begin{equation}
\label{Wapp12}
\mbox{III}\leq \mbox{III}' + \mbox{III}'' + \left\|\frac{1}{\eps}o\left(|\eps|\left(\int_C \|v_\eps(t,\mathcal{G}^\eps(\tilde c))\|_E^2 \di \mu^0(\tilde c)\right)^{\frac12}\right) \right\|_E,
\end{equation}
where
\begin{align*}
\mbox{III}' &:= \left\|\int_C \left(\nabla_{\left(\Phi_{(0,t)}^{\mu^0}\circ \mathcal{G}^\eps\right)_{\#}\mu^0}A{\scriptstyle \left(t, \left(\Phi_{(0,t)}^{\mu^0}\circ \mathcal{G}^\eps\right)_{\#}\mu^0, \Phi_{(0,t)}^{\mu_\eps^0}(c), u\left(t, \Phi_{(0,t)}^{\mu_\eps^0}(c)\right)\right)}\left(\Phi_{(0,t)}^{\mu^0}(\mathcal{G}^\eps(\tilde c))\right)\right.\right.  \\
& \,\,\,\,\,\left.\left.-\nabla_{\left(\Phi_{(0,t)}^{\mu^0}\right)_{\#}\mu^0}A{\scriptstyle\left(t,\left(\Phi_{(0,t)}^{\mu^0}\right)_{\#}\mu^0,\Phi_{(0,t)}^{\mu^0}(c),u\left(t,\Phi_{(0,t)}^{\mu^0}(c)\right)\right)}\left(\Phi_{(0,t)}^{\mu^0}(\tilde c)\right)\right)\left[v_\eps(t,\mathcal{G}^\eps(\tilde c))\right] \di\mu^0(\tilde c)\right\|_E, \nonumber
\end{align*}
and
\begin{align*}
&\mbox{III}'' :=
\\
&
\left\|\int_C\nabla_{\left(\Phi_{(0,t)}^{\mu^0}\right)_{\#}\mu^0}A{\scriptstyle\left(t,\left(\Phi_{(0,t)}^{\mu^0}\right)_{\#}\mu^0,\Phi_{(0,t)}^{\mu^0}(c),u\left(t,\Phi_{(0,t)}^{\mu^0}(c)\right)\right)}\left(\Phi_{(0,t)}^{\mu^0}(\tilde c)\right)\left[v_\eps(t,\mathcal{G}^\eps(\tilde c))-v_\eps(t,\tilde c)\right] \di\mu^0(\tilde c) \right\|_E. \nonumber
\end{align*}

First we focus on III$'$. Using \eqref{Wapp2} and applying H\"older inequality, we obtain for a positive constant $M$
\begin{align}
\label{Wapp15}
\mbox{III}' &\leq M\left(\int_C \left\| \nabla_{\left(\Phi_{(0,t)}^{\mu^0}\circ \mathcal{G}^\eps\right)_{\#}\mu^0}A{\scriptstyle \left(t, \left(\Phi_{(0,t)}^{\mu^0}\circ \mathcal{G}^\eps\right)_{\#}\mu^0, \Phi_{(0,t)}^{\mu_\eps^0}(c), u\left(t, \Phi_{(0,t)}^{\mu_\eps^0}(c)\right)\right)}\left(\Phi_{(0,t)}^{\mu^0}(\mathcal{G}^\eps(\tilde c))\right) \right.\right. \nonumber \\
& \left.\left.-\nabla_{\left(\Phi_{(0,t)}^{\mu^0}\right)_{\#}\mu^0}A{\scriptstyle\left(t,\left(\Phi_{(0,t)}^{\mu^0}\right)_{\#}\mu^0,\Phi_{(0,t)}^{\mu^0}(c),u\left(t,\Phi_{(0,t)}^{\mu^0}(c)\right)\right)}\left(\Phi_{(0,t)}^{\mu^0}(\tilde c)\right)\right\|_{\li(E_C;E_C)}^2 \!\!\!\!\!\!\! \di\mu^0(\tilde c)   \right)^{\frac12}.
\end{align}
By the continuity assumption {\bf (HA2)}-$(vii)$ and recalling that in Step 1 we have seen that $W_1\left(\left(\Phi_{(0,t)}^{\mu^0}\circ \mathcal{G}^\eps\right)_{\#}\mu^0, \left(\Phi_{(0,t)}^{\mu^0}\right)_{\#}\mu^0 \right)\to 0$ as $\eps \to 0$, we deduce
\begin{eqnarray*}
&& \lim_{\eps\to 0} \left\{\nabla_{\left(\Phi_{(0,t)}^{\mu^0}\circ \mathcal{G}^\eps\right)_{\#}\mu^0}A{\scriptstyle \left(t, \left(\Phi_{(0,t)}^{\mu^0}\circ \mathcal{G}^\eps\right)_{\#}\mu^0,\Phi_{(0,t)}^{\mu_\eps^0}(c),u\left(t,\Phi_{(0,t)}^{\mu_\eps^0}(c)\right)\right)}\left(\Phi_{(0,t)}^{\mu^0}(\mathcal{G}^\eps(\tilde c))\right)\right\} \\
&&\stackrel{\li(E_C;E_C)}=\nabla_{\left(\Phi_{(0,t)}^{\mu^0}\right)_{\#}\mu^0}A{\scriptstyle\left(t,\left(\Phi_{(0,t)}^{\mu^0}\right)_{\#}\mu^0,\Phi_{(0,t)}^{\mu^0}(c),u\left(t,\Phi_{(0,t)}^{\mu^0}(c)\right)\right)}\left(\Phi_{(0,t)}^{\mu^0}(\tilde c)\right)
\end{eqnarray*}
uniformly for $\tilde c\in \supp(\mu^0)$ (recall that $\supp(\mu^0)$ is compact in $C$). Moreover, by Lemma~\ref{grad-mu-bnd}, it follows that
\begin{eqnarray*}
&&\left\| \nabla_{\left(\Phi_{(0,t)}^{\mu^0}\circ \mathcal{G}^\eps\right)_{\#}\mu^0}A{\scriptstyle \left(t, \left(\Phi_{(0,t)}^{\mu^0}\circ \mathcal{G}^\eps\right)_{\#}\mu^0, \Phi_{(0,t)}^{\mu_\eps^0}(c), u\left(t, \Phi_{(0,t)}^{\mu_\eps^0}(c)\right)\right)}\left(\Phi_{(0,t)}^{\mu^0}(\mathcal{G}^\eps(\tilde c))\right) \right. \nonumber \\
&&\,\,\,\left.-\nabla_{\left(\Phi_{(0,t)}^{\mu^0}\right)_{\#}\mu^0}A{\scriptstyle\left(t,\left(\Phi_{(0,t)}^{\mu^0}\right)_{\#}\mu^0,\Phi_{(0,t)}^{\mu^0}(c),u\left(t,\Phi_{(0,t)}^{\mu^0}(c)\right)\right)}\left(\Phi_{(0,t)}^{\mu^0}(\tilde c)\right) \right\|_{\li(E_C;E_C)} \leq 2L.
\end{eqnarray*}
Thus, after an application of the Lebesgue theorem for general measures to the right-hand side of \eqref{Wapp15}, we infer that
\begin{equation}
\label{Wapp16}
\lim_{\eps\to 0} \mbox{III}'=0.
\end{equation}

As for III$''$, using again Lemma \ref{grad-mu-bnd}, we infer
\begin{equation}
\label{Wapp18}
\mbox{III}''\leq M \int_C \|v_\eps(t,\mathcal{G}^\eps(\tilde c))-v_\eps(t,\tilde c)\|_E \di\mu^0(\tilde c),
\end{equation}
for a positive constant $M$. Moreover, by assumption $(i)$-$(ii)$ and thanks to the property of uniform $C$-differentiability of $\Phi_{(0,t)}^{\mu^0}$ and $\Phi_{(0,t)}^{\mu^0_\eps}$ on $\supp(\mu^0)$ given by Lemma \ref{Dphi-bnd}, we have as $\eps \to 0$
\begin{eqnarray*}
v_\eps(t,\mathcal{G}^\eps(\tilde c))-v_\eps(t,\tilde c) &\stackrel{\eqref{Wapp3}}=& \frac{\Phi_{(0,t)}^{\mu_\eps^0}(\mathcal{G}^\eps(\tilde c)) - \Phi_{(0,t)}^{\mu^0}(\mathcal{G}^\eps(\tilde c))}{\eps}-\frac{\Phi_{(0,t)}^{\mu_\eps^0}(\tilde c) - \Phi_{(0,t)}^{\mu^0}(\tilde c)}{\eps} \\
&= &\frac{1}{\eps}\left\{\Phi_{(0,t)}^{\mu_\eps^0}(\mathcal{G}^\eps(\tilde c)) - \Phi_{(0,t)}^{\mu_\eps^0}(\tilde c)-\Phi_{(0,t)}^{\mu^0}(\mathcal{G}^\eps(\tilde c)) + \Phi_{(0,t)}^{\mu^0}(\tilde c)\right\} \\
& = & \frac{1}{\eps}\left\{\eps\Di_{\tilde c}\Phi_{(0,t)}^{\mu_\eps^0}\left[\mathcal{F}(\tilde c)\right] - \eps\Di_{\tilde c}\Phi_{(0,t)}^{\mu^0}\left[\mathcal{F}(\tilde c)\right] + o(\eps) \right\} \\
& = & \left(\Di_{\tilde c}\Phi_{(0,t)}^{\mu_\eps^0}-\Di_{\tilde c}\Phi_{(0,t)}^{\mu^0}\right)\left[\mathcal{F}(\tilde c)\right] + \frac{o(\eps)}{\eps} \stackrel{E}\longrightarrow 0,
\end{eqnarray*}
where in the limit we have applied Lemma \ref{conv-flussi-mis}. Thanks to \eqref{Wapp2} and applying the Lebesgue theorem to the right-hand side of \eqref{Wapp18}, we obtain
\begin{equation}
\label{Wapp19}
\lim_{\eps\to 0} \mbox{III}''=0.
\end{equation}

Finally, using again \eqref{Wapp2}, we infer that
$$
\lim_{\eps \to 0} \left\{\frac{1}{\eps}o\left(|\eps|\left(\int_C \|v_\eps(t,\mathcal{G}^\eps(\tilde c))\|_E^2 \di \mu^0(\tilde c)\right)^{\frac12}\right)\right\} = 0.
$$
Hence, combining the last limit with \eqref{Wapp12}, \eqref{Wapp16} and \eqref{Wapp19}, we conclude that
$$
\lim_{\eps\to 0} \mbox{III} = 0.
$$
This, together with \eqref{Wapp6} and \eqref{Wapp20bis}, implies that
\begin{equation}
\label{Wapp20}
\lim_{\eps\to 0} r_\eps(t,c) = 0 \qquad \forall (t,c)\in [0,T]\times V.
\end{equation}

{\bf Step 4.} In this last step we prove the result. \\
Using {\bf(HA1)}-$(i)$, the fact that $W_1(\mu_\eps(t),\mu(t))\leq M|\eps|$ (as seen in Step 1) and \eqref{Wapp3} in the first term of the right-hand side of \eqref{Wapp5}, inequalities \eqref{Das-bnd} and \eqref{Wapp2} in the second term and Lemma \ref{grad-mu-bnd}, Lemma \ref{Dphi-bnd}, assumption $(ii)$ and again \eqref{Wapp2} in the third term  respectively, we deduce for a positive constant $M$ that
$$
\|r_\eps(t,c)\|_E\leq M \qquad \forall (t,c)\in [0,T]\times V.
$$
Therefore, by \eqref{Wapp20} and applying the Lebesgue theorem, we infer that
\begin{equation}
\label{Wapp21}
r_\eps(t,c)\to 0 \qquad \mbox{in }L^1_{\mathcal{L}\times \mu^0}([0,T]\times C;E) \qquad\mbox{as }\eps\to 0.
\end{equation}
Finally, we estimate \eqref{Wapp4}. By \eqref{Das-bnd} and Lemma \ref{grad-mu-bnd}, we deduce
\begin{eqnarray}
\label{Wapp24}
\|v_\eps(t,c) - v(t,c)\|_E  &\leq&  L\int_0^t \|v_\eps(s,c)-v(s,c)\|_E \di s \nonumber \\
&&+L \int_0^t \int_C \|v_\eps(s,\tilde c) - v(s,\tilde c)\|_E \di\mu^0(\tilde c) \di s  \nonumber\\
&&+ \int_0^t \|r_\eps(s,c)\|_E \di s.
\end{eqnarray}
Integrating over $C$, recalling that $\mu^0(C)=1$ and applying Fubini theorem, we have
\begin{eqnarray*}
\int_C\|v_\eps(t,c) - v(t,c)\|_E \di \mu^0(c) &\leq&  2L\int_0^t \int_C \|v_\eps(s,c)-v(s,c)\|_E \di \mu^0(c)\di s \nonumber \\
&&+ \int_0^t \int_C\|r_\eps(s,c)\|_E \di\mu^0(c) \di s,
\end{eqnarray*}
which implies, applying Gr\"onwall inequality and using \eqref{Wapp21}, that
$$
\int_C\|v_\eps(t,c) - v(t,c)\|_E \di \mu^0(c) \leq \delta_\eps \ex^{2Lt}, \qquad \delta_\eps:=\int_0^T \int_C\|r_\eps(s,c)\|_E \di\mu^0(c) \di s \rightarrow  0 \text{ as }\eps \to 0.
$$
Inserting this last inequality in \eqref{Wapp24}, we obtain for every $(t,c)\in [0,T]\times V$
$$
\|v_\eps(t,c) - v(t,c)\|_E \leq  L\int_0^t \|v_\eps(s,c)-v(s,c)\|_E \di s + \delta_\eps \ex^{2LT} + \int_0^t \|r_\eps(s,c)\|_E \di s.
$$
Hence, setting
$$
\hat\delta_\eps:= \delta_\eps \ex^{2LT} + \int_0^T \|r_\eps(s,c)\|_E \di s,
$$
recalling that, by \eqref{Wapp21}, $\hat\delta_\eps \to 0$ uniformly in $V$ as $\eps \to 0$, and applying again Gr\"onwall inequality, we infer
$$
\|v_\eps(t,c) - v(t,c)\|_E \leq \hat\delta_\eps \ex^{Lt},
$$
which gives \eqref{goalin}.
\end{proof}

\begin{proof}[Proof of Theorem \ref{mainres}]
We divide the proof into 4 steps. \\

{\bf Step 1: Needle variations.} Fix any time $\tau\in (0,T]$ and any admissible control value $\omega\in U$. Up to a null set we can assume that $\tau$ is a Lebesgue point of $t \mapsto A(t,\muo(t),c,z)$ for all $z\in Z$. Notice that this is possible by separability of $Z$ and the uniform Lipschitz continuity of $A$ in the last variable. For $\varepsilon\in [0,\bar{\varepsilon})$ with $\bar{\varepsilon}>0$ sufficiently small, we consider the following needle variation:
\begin{equation}
\label{Wapp32}
u_\eps(t)=\begin{cases} \omega & \text{ if }t\in[\tau-\eps,\tau], \\
\uo(t) & \text{ otherwise}.
\end{cases}
\end{equation}
Thanks to {\bf (HA1)}, by Theorem 3.3 of \cite{AFMS}, there exists a unique solution $\mu_\eps\in AC([0,T]; \PC)$ of
$$
\begin{cases}
\dd \mu(t) +\mathrm{div}\left(A(t,\mu(t),\cdot,u_\eps(t,\cdot))\mu(t)\right)=0 & \text{ in }(0,T], \\
\mu(0)=\mu^0\in \mathcal{P}_c(C).
\end{cases}
$$
The curve $\mu_\eps(t)$ can be represented using the associated family of non-local flows as
$$
\mu_\eps(t)= \left(\Phi_{(0,t)}^{\varepsilon,\mu^0}\right)_{\#}\mu^0 \qquad \text{for }t\in [0,T].
$$
Moreover, by Proposition 4 of \cite{AAMS}, $R>0$ exists such that $\supp(\muo(t)) \cup \supp(\mu_\eps(t)) \subset B_{R,C}$ and $\|\Phi^{\varepsilon, \mu^{0}}_{(0, t)}\|_{C^0_b(C;C)} \leq R$ for every $(\varepsilon,t)\in [0,\bar{\varepsilon})\times [0,T]$.

Let $t\in(\tau,T]$. We want to compute for every $(t,c)\in (\tau, T]\times B_{R,C}$
\begin{equation}
\label{Wapp31}
\lim_{\varepsilon\to 0^+} \frac{\Phi_{(0,t)}^{\varepsilon,\mu^0} \circ   \Phi_{(t,0)}^{\muo(t)}(c) - \text{Id}(c)}{\varepsilon},
\end{equation}
where $\text{Id}\colon C \to C$ is the identity function. It follows from \eqref{Wapp32} and from the definition of $\Phi_{(s,t)}^{\varepsilon,\mu_\eps(s)}$ and $\Phi_{(s,t)}^{\muo(s)}$ ($0\leq s<t\leq T$) that
\begin{equation}
\label{Wapp33}
\Phi_{(0,t)}^{\varepsilon,\mu^0} \circ   \Phi_{(t,0)}^{\muo (t)}(c)  = \Phi_{(\tau,t)}^{\mu_\eps(\tau)}\circ \Phi_{(\tau-\eps,\tau)}^{\varepsilon,\muo(\tau-\eps)}\circ\Phi_{(\tau,\tau-\eps)}^{\muo(\tau)}(c_\tau)
\end{equation}
where $c_\tau:=\Phi_{(t,\tau)}^{\muo(t)}(c)$. Using again {\bf (HA1)}, we deduce that the map
$$
\sigma\mapsto A{\scriptstyle\left(\sigma,\muo(\sigma), \Phi_{(\sigma,\tau-\eps)}^{\muo(\tau)}(c_\tau),\uo\left(\sigma,\Phi_{(\sigma,\tau-\eps)}^{\muo(\tau)}(c_\tau) \right)\right)} \in L^\infty([0,T];E).
$$
Then, after an application of the Lebesgue differentiation theorem for vector-valued functions, we obtain
\begin{eqnarray*}
\Phi_{(\tau,\tau-\eps)}^{\muo(\tau)}(c_\tau) &\stackrel{\eqref{flow}}=& c_\tau - \int_{\tau-\eps}^\tau A{\scriptstyle\left(\sigma,\muo(\sigma),\Phi_{(\sigma,\tau-\eps)}^{\muo(\tau)}(c_\tau),\uo\left(\sigma,\Phi_{(\sigma,\tau-\eps)}^{\muo(\tau)}(c_\tau) \right)\right)} \di\sigma  \nonumber \\
&=& c_\tau -\eps A\left(\tau,\muo(\tau),c_\tau,\uo\left(\tau,c_\tau\right)\right) + \delta_\eps^1,
\end{eqnarray*}
with $\frac{\delta_\eps^1}{\eps}\to 0$ in $E$ as $\eps \to 0^+$.  Now, choose $d=c_\tau -\eps A\left(\tau,\muo(\tau),c_\tau,\uo\left(\tau,c_\tau\right)\right) + \delta_\eps^1$. Observe that
\[
\frac1{\eps}\int_{\tau-\eps}^\tau\left[W_1(\mu_\eps(\sigma), \muo(\tau))+\left\|\Phi_{(\tau-\eps,\sigma)}^{\eps,\muo(\tau-\eps)}(d)-c_\tau\right\|_E+ \left\|\omega\left(\Phi_{(\tau-\eps,\sigma)}^{\eps,\muo(\tau-\eps)}(d)\right)-\omega(c_\tau)\right\|_Z \right]\di\sigma \to 0\,.
\]
Using {\bf (HA1)}-$(i)$ we deduce that
\[
\frac1{\eps}\int_{\tau-\eps}^\tau A{\scriptstyle\left(\sigma,\mu_\eps(\sigma),\Phi_{(\tau-\eps,\sigma)}^{\eps,\muo(\tau-\eps)}(d), \omega\left(\Phi_{(\tau-\eps,\sigma)}^{\eps,\muo(\tau-\eps)}(d)\right)\right)}- A{\scriptstyle\left(\sigma,\mu_\eps(\sigma),\Phi_{(\tau-\eps,\sigma)}^{\eps,\muo(\tau-\eps)}(d),\omega\left(c_\tau\right)\right)}  \di\sigma \to 0\,.
\]
With the above equalities and the Lebesgue differentiation theorem we get
\begin{eqnarray*}
\Phi_{(\tau-\eps,\tau)}^{\varepsilon,\muo(\tau-\eps)}(d) &\stackrel{\eqref{flow}}=& d + \int_{\tau-\eps}^\tau A{\scriptstyle\left(\sigma,\mu_\eps(\sigma),\Phi_{(\tau-\eps,\sigma)}^{\eps,\muo(\tau-\eps)}(d),\omega\left(\Phi_{(\tau-\eps,\sigma)}^{\eps,\muo(\tau-\eps)}(d)\right)\right)} \di\sigma  \nonumber \\
&=& c_\tau -\eps A\left(\tau,\muo(\tau),c_\tau,\uo\left(\tau,c_\tau\right)\right) +\eps A\left(\tau,\muo(\tau),c_\tau,\omega(c_\tau)\right) + \delta_\eps^2,
\end{eqnarray*}
with $\frac{\delta_\eps^2}{\eps}\to 0$ in $E$ as $\eps \to 0^+$. Recalling the definition of $d$, this amounts to
\begin{equation}
\label{Wapp34}
\Phi_{(\tau-\eps,\tau)}^{\varepsilon,\muo(\tau-\eps)}\circ\Phi_{(\tau,\tau-\eps)}^{\muo( \tau)}(c_\tau) = c_\tau +\eps\left( A\left(\tau,\muo(\tau),c_\tau,\omega(c_\tau)\right) - A\left(\tau,\muo(\tau),c_\tau,\uo\left(\tau,c_\tau\right)\right)\right) + \delta_\eps,
\end{equation}
with $\frac{\delta_\eps}{\eps}\to 0$ in $E$ as $\eps \to 0^+$. Recalling that $\muo(t)$ and $\mu_\eps(t)$ can be represented using the associated non-local flows, we have the following expression which links $\mu_\eps(\tau)$ and $\muo(\tau)$:
\begin{eqnarray}
\label{Wapp36}
\mu_\eps(\tau) &=& \left(\Phi_{(\tau-\eps,\tau)}^{\varepsilon,\muo(\tau-\eps )}\right)_{\#}\muo(\tau-\eps) = \left(\Phi_{(\tau-\eps,\tau)}^{\varepsilon,\muo(\tau-\eps)} \circ \Phi_{(\tau,\tau-\eps)}^{\muo(\tau)}\right)_{\#}\muo(\tau) \nonumber \\
&\stackrel{\eqref{Wapp34}}=& \left( \text{Id} +\eps\left(A\left(\tau,\muo(\tau),\cdot,\omega(\cdot)\right) - A\left(\tau,\muo(\tau),\cdot,\uo\left(\tau,\cdot\right)\right)\right) + \delta_\eps  \right)_{\#}\muo(\tau).
\end{eqnarray}
We then define for fixed $c\in B_{R,C}$ the function $\mathcal{G}\colon  [0,\bar\eps) \to E$ as
\begin{equation}
\label{Wapp37tris}
\mathcal{G}(\eps,\tau,c):= \text{Id} +\eps\left( A\left(\tau,\muo(\tau),c,\omega(c)\right) - A\left(\tau,\muo(\tau),c,\uo\left(\tau,c\right)\right)\right) + \delta_\eps,
\end{equation}
and, since $\mathcal{G}$ is Fr\'echet-differentiable from the right with differential
\begin{equation}
\label{Wapp37bis}
\mathcal{F}(\tau,c):= A\left(\tau,\muo(\tau),c,\omega(c)\right) - A\left(\tau,\muo(\tau),c,\uo\left(\tau,c\right)\right),
\end{equation}
applying Lemma 2.11 of \cite{Shva} we can extend $\mathcal{G}$ from $[0,\bar\eps)$ to $(-\bar\eps,\bar\eps)$ preserving the Fr\'echet-differentiability at $\eps=0$. Consequently, by {\bf (HA1)}-$(ii)$, we have for $\eps\in(-\bar\eps,\bar\eps)$
\begin{equation}
\label{Wapp37}
\begin{cases}
\mathcal{G}(0,\tau,\cdot)=\text{Id},\\
\frac{\di}{\di\eps}\mathcal{G}(\eps,\tau,c)\big|_{\eps=0}=\mathcal{F}(\tau,c)\in L^{\infty}_{\muo(\tau)}(B_{R,C};E), \\
\mu_\eps(\tau)\stackrel{\eqref{Wapp36}}= \mathcal{G}(\eps,\tau,\cdot)_{\#}\muo(\tau).
\end{cases}
\end{equation}
Therefore, combining \eqref{Wapp31}, \eqref{Wapp33}, \eqref{Wapp34} and \eqref{Wapp37tris}, we infer that
\begin{equation}
\label{Wapp38}
\frac{\Phi_{(0,t)}^{\varepsilon,\mu^0} \circ  \Phi_{(t,0)}^{\muo(t)}(c)  - \text{Id}(c)}{\varepsilon} = \frac{\Phi_{(\tau,t)}^{\mu_\eps(\tau)}\left(\mathcal{G}(\eps,\tau,c_\tau)\right) - \Phi_{(\tau,t)}^{\muo(\tau)}(c_\tau)}{\eps}
\end{equation}
for every $(\eps,t,c)\in (-\bar\eps,\bar\eps)\times (\tau,T]\times B_{R,C}$. Thanks to \eqref{Wapp37} we can apply Proposition \ref{linearized} and Lemma \ref{Dphi-bnd} obtaining
\begin{eqnarray}
\label{Wapp39}
\!\!\!\!\!\!\Phi_{(\tau,t)}^{\mu_\eps(\tau)}\left(\mathcal{G}(\eps,\tau,c_\tau)\right) &=& \Phi_{(\tau,t)}^{\muo(\tau)}\left(\mathcal{G}(\eps,\tau,c_\tau)\right) + \eps v(t,\mathcal{G}(\eps,\tau,c_\tau)) + \hat{\delta}_\eps \nonumber \\
&=& \Phi_{(\tau,t)}^{\muo(\tau)}(c_\tau) + \eps \left(\Di_{c_\tau}\Phi_{(\tau,t)}^{\muo(\tau)}[\mathcal{F}(\tau,c_\tau)] + v(t,\mathcal{G}(\eps,\tau,c_\tau))\right) + \tilde{\delta}_\eps,
\end{eqnarray}
where $v(t,\cdot)$ is the Fr\'echet derivative at $\eps=0$ of $\Phi_{(\tau,t)}^{\mu_\eps(\tau)}(\cdot)$ and $\frac{\hat{\delta}_\eps}{\eps},\frac{\tilde{\delta}_\eps}{\eps}\to 0$ as $\eps\to 0^+$. Hence, by the continuity of $\mathcal{G}$ with respect to $\eps$ and $v$ with respect to $c$ given by Lemma~\ref{ODE-continua}, we conclude that for every $(t,c)\in (\tau,T] \times B_{R,C}$
\begin{equation*}
\lim_{\varepsilon\to 0^+} \frac{\Phi_{(0,t)}^{\varepsilon,\mu^0} \circ  \Phi_{(t,0)}^{\muo( t) }(c)  - \text{Id}(c)}{\varepsilon} \stackrel{\eqref{Wapp38},\eqref{Wapp39}} = \Di_{c_\tau}\Phi_{(\tau,t)}^{\muo(\tau)}[\mathcal{F}(\tau,c_\tau)] + v(t,c_\tau).
\end{equation*}
We define $\mathcal{F}\colon (\tau,T]\times B_{R,C} \to E_C$ as
\begin{eqnarray}
\label{Wapp41}
\mathcal{F}(t,c) &:=& \Di_{c}\Phi_{(\tau,t)}^{\muo(\tau)}[\mathcal{F}(\tau,c)] + v(t,c) \nonumber \\
&\stackrel{\eqref{Wapp37bis}}=& \Di_{c}\Phi_{(\tau,t)}^{\muo(\tau)}[A\left(\tau,\muo(\tau),c,\omega(c)\right) - A\left(\tau,\muo(\tau),c,\uo\left(\tau,c\right)\right)] + v(t,c),
\end{eqnarray}
thus obtaining
\begin{equation}
\label{Wapp40}
\lim_{\varepsilon\to 0^+} \frac{\Phi_{(0,t)}^{\varepsilon,\mu^0} \circ  \Phi_{(t,0)}^{\muo(t)}(c)  - \text{Id}(c)}{\varepsilon}
\stackrel{\eqref{Wapp41}} = \mathcal{F}(t,  \Phi_{(t,\tau)}^{\muo(t)}(c)  ).
\end{equation}
Finally, again applying Lemma \ref{Dphi-bnd} and Proposition \ref{linearized}, we characterize $\mathcal{F}(t,c)$ as the unique solution of the following linear ODE defined in $E_C$ for $t\in (\tau,T]$:
\begin{equation}
\label{Wapp42}
\begin{cases}
\dd \mathcal{F}(t,c) = \Di_{\Phi_{(\tau,t)}^{\muo(\tau)}(c)} A{\scriptstyle \left(t,\muo(t),\Phi_{(\tau,t)}^{\muo(\tau)}(c),\uo\left(t,\Phi_{(\tau,t)}^{\muo(\tau)}(c)\right)\right)}[\mathcal{F}(t,c)] \\
\,\,\,\,\,\,\,\,+\displaystyle \int_C \nabla_{\muo(t)}A{\scriptstyle \left(t,\muo(t),\Phi_{(\tau,t)}^{\muo(\tau)}(c),\uo\left(t,\Phi_{(\tau,t)}^{\muo(\tau)}(c)\right)\right)}\left(\Phi_{(\tau,t)}^{\muo(\tau)}(\tilde c)\right)\left[\mathcal{F}(t,\tilde c) \right] \di\muo(\tau)(\tilde c),  \\
\mathcal{F}(\tau,c)\stackrel{\eqref{Wapp37bis}}= A\left(\tau,\muo(\tau),c,\omega(c)\right) - A\left(\tau,\muo(\tau),c,\uo\left(\tau,c\right)\right).
\end{cases}
\end{equation}
\\

{\bf Step 2: Optimality condition.} Note that, by the definitions of $\muo(t)$ and $\mu_\eps(t)$ and using \eqref{Wapp40}, we deduce
\begin{equation}
\label{Wapp43}
\mu_\eps(t) = \left(\Phi_{(0,t)}^{\varepsilon,\mu^0} \circ  \Phi_{(t,0)}^{\muo(t) }  \right)_{\#}\muo(t) = \left(\text{Id}+\eps\mathcal{F}(t,  \Phi_{(t,\tau)}^{\muo( t) }  (\cdot))+\delta_\eps\right)_{\#}\muo(t),
\end{equation}
for every $t\in (\tau,T]$, where $\frac{\delta_\eps}{\eps}\to 0$ as $\eps \to 0^+$. Thus, it follows from the assumption {\bf (H$\varphi$)} and using \eqref{Wapp43} with $t=T$ that we can apply Proposition \ref{Wchainrule} obtaining
\begin{equation}
\label{Wapp44}
\varphi(\mu_\eps(T)) = \varphi(\muo(T)) + \eps \int_C \left\langle \nabla_{\muo(T)}\varphi(\muo(T))(c), \mathcal{F}(T,  \Phi_{(T,\tau)}^{\muo(T)}(c)) \right\rangle \di\muo(T) (c) + \delta_\eps^1,
\end{equation}
where $\frac{\delta_\eps^1}{\eps}\to 0$ as $\eps \to 0^+$. By the optimality condition on $\uo\in \mathcal{U}$, we know that
$$
\frac{\varphi(\mu_\eps(T)) - \varphi(\muo(T))}{\eps} \geq 0 \qquad \text{for every }\eps>0,
$$
hence, using \eqref{Wapp44} and letting $\eps$ go to $0$, we infer
\begin{equation}
\label{Wapp45}
\int_C \left\langle \nabla_{\muo(T)}\varphi(\muo(T))(c), \mathcal{F}(T,  \Phi_{(T,\tau)}^{\muo( T)}  (c)) \right\rangle \di\muo(T) (c) \geq 0.
\end{equation}
\\

{\bf Step 3: Adjoint problem.} In this Step we follow Step 3 of Section 3.2 in \cite{BonRos}. \\
We consider the unique distributional solution $\psi_t(c,p)\colon  [0,T]\times C \times E_C^* \to E_C^*$ of the adjoint equation of \eqref{Wapp42} (which is a linear ODE defined on $E_C^*$):
\begin{equation}
\label{Wapp46}
\begin{cases}
\dd w(t,c,p) = -\Di_{\Phi_{(T,t)}^{\muo(T)}(c)}^* A {\scriptstyle \left(t,\muo(t),\Phi_{(T,t)}^{\muo(T)}(c),\uo\left(t,\Phi_{(T,t)}^{\muo(T)}(c)\right)\right)}[w(t,c,p)] \\
\,\,\,-\displaystyle \int_{C\times E_C^*} \nabla_{\muo(t)}^*A{\scriptstyle \left(t,\muo(t),\Phi_{(T,t)}^{\muo(T)}(\tilde c),\uo\left(t,\Phi_{(T,t)}^{\muo(T)}(\tilde c)\right)\right)}\left(\Phi_{(T,t)}^{\muo(T)}(c)\right)\left[w(t,\tilde c,\tilde p) \right] \di\left(\scriptstyle{\left(\text{Id}, -\nabla_{\muo(T)}\varphi(\muo(T))\right)_{\#}\muo(T)}\right)(\tilde c,\tilde p),  \\
w(T,c,p) = p\in E_C^*.
\end{cases}
\end{equation}
Note that, since $D_cA(t,\mu,c,u)$ and $\nabla_{\mu}A(t,\mu,\tilde c,u)(c)$ are uniformly bounded operators in $\li(E_C;E_C)$, then their adjoint operators are uniformly bounded in $\li(E_C^*;E_C^*)$. Hence, applying Lemma \ref{ODE-continua}, there exists a unique distributional solution $\psi_t(c,p)\in AC([0,T];E_C^*)$ of \eqref{Wapp46}. Moreover, $\psi_t(c,p)$ is continuous with respect to $c$ and $p$. Thus, we define a curve of measures in $E_C^*$ as
\begin{equation}
\label{Wapp47}
\sigma_c : t\mapsto \psi_t(c,\cdot)_{\#}\delta_{-\nabla_{\muo(T)}\varphi(\muo(T))(c)}.
\end{equation}
Since $\delta_{-\nabla_{\muo(T)}\varphi(\muo(T))(c)}\in \mathcal{P}_c(E_C^*)$, we have $\sigma_c \in AC([0,T];\mathcal{P}_c(E_C^*))$. Now we extend this curve of measures on $C\times E_C^*$. First we define
\begin{equation}
\label{Wapp48}
\nu_T: t\mapsto \int_C \sigma_c(t) \di \muo(T)(c),
\end{equation}
which belongs to $AC([0,T];\mathcal{P}_c(C\times E_C^*))$ since $\mu(T)$ and $\sigma_c(t)$ have compact support in~$C$ and $E_C^*$ respectively. Then we set
\begin{equation}
\label{Wapp49}
\nuo: t\mapsto \left(\Phi_{(T,t)}^{\muo(T)}\circ\pi^1 , \pi^2\right)_{\#}\nu_T(t),
\end{equation}
thus obtaining $\nuo \in AC([0,T];\mathcal{P}_c(C\times E_C^*))$. By its definition we deduce that
\begin{eqnarray}
\label{Wapp53}
\di\nuo(T)(c,p) &\stackrel{\eqref{Wapp49}}= &\di\nu_T(T)(c,p) \stackrel{\eqref{Wapp48}}= \di \sigma_c(T)(p)\di\muo(T)(c) \nonumber \\ &\stackrel{\eqref{Wapp47},\eqref{Wapp46}}= & \di \delta_{-\nabla_{\muo(T)}\varphi(\muo(T))(c)}(p)\di\muo(T)(c) \nonumber \\
& = & \di\left(\left(Id , -\nabla_{\muo(T)}\varphi(\muo(T))\right)_{\#}\muo(T)\right)(c,p).
\end{eqnarray}
Moreover, we have
\begin{eqnarray}
\label{Wapp53bis}
\pi^1_{\#}\nuo(t) \stackrel{\eqref{Wapp49},\eqref{Wapp48}}= (\Phi_{(T,t)}^{\muo(T)})_{\#}\muo(T)=\muo(t) \quad \text{for every }t\in [0,T].
\end{eqnarray}
Now we prove that $\nuo$ is a distributional solution of
\begin{equation}
\label{Wapp50}
\begin{cases}
\dis\dd \nu(t) = - \mathrm{div}_{(c,p)}\left(\Gamma\left(t,\nu(t),c,p,\uo(t,c)\right)\nu(t)\right)   & \text{ in } [0,T), \\
\pi^1_{\#}\nu(t)=\muo(t), \\
\nu(T)= \left(Id , -\nabla_{\muo(T)}\varphi(\muo(T))\right)_{\#}\muo(T),
\end{cases}
\end{equation}
where $\Gamma = (\Gamma_{1}, \Gamma_{2}) \colon  [0,T]\times \mathcal{P}_b(C\times E_C^*) \times C \times E_C^* \times U \to E_C \times E_C^*$ is defined componentwise by
\begin{eqnarray}
\label{Wapp51}
&&\Gamma_1(t,\nu,c,p,u) := A(t,\pi_{\#}^1\nu,c,u) \\
&& \Gamma_2(t,\nu,c,p,u): =-D_c^{*}A(t,\pi_{\#}^1\nu,c,u)[p] - \displaystyle\int_{C\times E_C^*}\nabla_{\pi_{\#}^1\nu}^*A(t,\pi_{\#}^1\nu,\tilde c,u)(c)[\tilde p] \di\nu(\tilde c,\tilde p). \nonumber
\end{eqnarray}

Let $\xi\in C_c^{\infty}(E \times E_C^*;\R)$ with differential $(\Di_E\xi,\Di_{E_C^*}\xi)$. Then
\begin{eqnarray*}
\label{Wapp52}
&& \dd\left[ \int_{C\times E_C^*}\xi(c,p)\di \nuo(t)(c,p) \right] \stackrel{\eqref{Wapp49}}=  \dd\left[ \int_{C\times E_C^*}\xi(\Phi_{(T,t)}^{\muo^T}(c),p)\di \nu_T(t)(c,p) \right]  \\
&& \stackrel{\eqref{Wapp48}} = \dd\left[\int_{C\times E_C^*} \xi(\Phi_{(T,t)}^{\muo^T}(c),p)\di\sigma_c(t)(p)\di\muo(T)(c)\right] \nonumber \\
&& \stackrel{\eqref{Wapp47}}=\dd\left[\int_{C\times E_C^*} \xi(\Phi_{(T,t)}^{\muo^T}(c),\psi_t(c,p)) \di\delta_{-\nabla_{\muo(T)}\varphi(\muo(T))(c)}(p)    \di\muo(T)(c)\right] \nonumber \\
&& \stackrel{\eqref{Wapp53}}=\dd\left[\int_{C\times E_C^*} \xi(\Phi_{(T,t)}^{\muo^T}(c),\psi_t(c,p)) \di\nuo(T)(c,p) \right] \nonumber \\
&& \stackrel{\eqref{flow},\eqref{Wapp46}}= \int_{C\times E_C*} \left\langle \Di_E \xi(\Phi_{(T,t)}^{\muo^T}(c),\psi_t(c,p)), A{\scriptstyle \left(t,\muo(t),\Phi_{(T,t)}^{\muo^T}(c),\uo\left(t,\Phi_{(T,t)}^{\muo^T}(c)\right)\right)}\right\rangle \di\nuo(T)(c,p) \nonumber \\
&& -\int_{C\times E_C*} \left\langle \Di_{E_C^*}\xi(\Phi_{(T,t)}^{\muo^T}(c),\psi_t(c,p)), \Di_{\Phi_{(T,t)}^{\muo^T}(c)}^* A {\scriptstyle \left(t,\muo(t),\Phi_{(T,t)}^{\muo^T}(c),\uo\left(t,\Phi_{(T,t)}^{\muo^T}(c)\right)\right)}[\psi_t(c,p)] \right. \nonumber \\
&& \phantom{\int_{C\times E_C*}} \left.+\int_{C\times E_C^*} \nabla_{\muo(t)}^*A{\scriptstyle \left(t,\muo(t),\Phi_{(T,t)}^{\muo^T}(\tilde c),\uo\left(t,\Phi_{(T,t)}^{\muo^T}(\tilde c)\right)\right)}\left(\Phi_{(T,t)}^{\muo^T}(c)\right)\left[\psi_t(\tilde c,\tilde p) \right] \di\nuo(T)(\tilde c,\tilde p) \right\rangle \di\nuo(T)(c,p) \nonumber \\
&& \stackrel{\eqref{Wapp53},\eqref{Wapp47},\eqref{Wapp48},\eqref{Wapp49}}=\int_{C\times E_C*} \left\langle \Di_E \xi(c,p), A{\scriptstyle\left(t,\muo(t),c,\uo(t,c)\right)}\right\rangle \di\nuo(t)(c,p) \nonumber \\
&& -\int_{C\times E_C*} \left\langle \Di_{E_C^*}\xi(c,p), \Di_{c}^* A {\scriptstyle \left(t,\muo(t),c,\uo\left(t,c\right)\right)}[p] +\int_{C\times E_C^*} \nabla_{\muo(t)}^*A{\scriptstyle \left(t,\muo(t),\tilde c,\uo\left(t,\tilde c\right)\right)}\left(c\right)\left[\tilde p\right] \di\nuo(t)(\tilde c,\tilde p) \right\rangle \di\nuo(t)(c,p) \nonumber \\
&& \stackrel{\eqref{Wapp53bis},\eqref{Wapp51}}= \int_{C\times E_C*} \left\langle \Di_E \xi(c,p), \Gamma_1(t,\nuo(t),c,p,\uo(t,c))\right\rangle \di\nuo(t)(c,p) \nonumber \\
&& \phantom{\stackrel{\eqref{Wapp53bis},\eqref{Wapp51}}=} +\int_{C\times E_C^*} \left\langle \Di_{E_C^*}\xi(c,p), \Gamma_2(t,\nuo(t),c,p,\uo(t,c))\right\rangle \di\nuo(t)(c,p). \nonumber
\end{eqnarray*}
It follows, recalling \eqref{Wapp53bis} and \eqref{Wapp53}, that $\nuo$ is a distributional solution of \eqref{Wapp50}. \\
Finally we define $\K\colon [\tau,T]\to \R$ as
\begin{eqnarray}
\label{Wapp54}
\K(t) &:=& \int_{C\times E_C^*} \left\langle p,\mathcal{F}\left(t,\Phi_{(t,\tau)}^{\muo(t)}(c)\right)\right\rangle \di\nuo(t)(c,p) \nonumber \\ &\stackrel{\eqref{Wapp49},\eqref{Wapp48},\eqref{Wapp47},\eqref{Wapp53}}=& \int_{C\times E_C^*} \left\langle \psi_t(c,p),\mathcal{F}\left(t,\Phi_{(T,\tau)}^{\muo(T)}(c)\right)\right\rangle \di\nuo(T)(c,p),
\end{eqnarray}
where $\mathcal{F}$ is defined by \eqref{Wapp41}. Using that $\psi_t(c,p)$ and $\mathcal{F}(t,c)$ are weak solutions of \eqref{Wapp46} and \eqref{Wapp42} respectively, by the same density argument used in the proof of Theorem \ref{teo1} (cf.~the comments after~\eqref{may6}), recalling the definition of adjoint operator and applying Fubini's Theorem, we obtain
\begin{equation}
\label{Wapp55}
\dd \K(t) = 0 \quad \text{a.e }t\in [\tau,T].
\end{equation}

{\bf Step 4: Conclusion of the proof.} Note that, by \eqref{Wapp55} and since
$$
\K(T) \stackrel{\eqref{Wapp54},\eqref{Wapp53}}= \int_C \left\langle -\nabla_{\muo(T)}\varphi(\muo(T))(c), \mathcal{F}(T,  \Phi_{(T,\tau)}^{\muo(T)} (c)) \right\rangle \di\muo(T)(c) \stackrel{\eqref{Wapp45}}\leq 0,
$$
we deduce that $\K(t) \leq 0$ for $t\in [\tau,T]$. In particular, it holds for every $\omega\in U$ and a.e. $\tau\in [0,T]$ that
\begin{eqnarray*}
\label{Wapp56}
&&\K(\tau) \stackrel{\eqref{Wapp54}}= \int_{C\times E_C^*} \langle p,\mathcal{F}(\tau,c) \rangle \di\nuo(\tau)(c,p)   \\
&& \stackrel{\eqref{Wapp37bis}}= \int_{C\times E_C^*} \langle p,A\left(\tau,\muo(\tau),c,\omega(c)\right) - A\left(\tau,\muo(\tau),c,\uo\left(\tau,c\right)\right) \rangle \di\nuo(\tau)(c,p) \leq 0. \nonumber
\end{eqnarray*}
Thus, recalling \eqref{Wham}, we get \eqref{Wmaxcond}. Finally, it follows from \eqref{form-diff-H} and \eqref{Wapp51} that
$$
\Gamma(t,\nu,c,p,u) = J\nabla_\nu \mathcal H(t,\nu,u)(c,p).
$$
Hence, by \eqref{Wapp50}, $\nuo$ is a distributional solution of \eqref{Wsist}.
\end{proof}

Now we focus on the Bolza problem and we give the infinite-dimensional version of Theorem \ref{PMP}.
Let $L\colon [0,T]\times \mathcal{P}_b(C) \times  C^1_b(C;Z)\to \R$. We consider
\begin{equation}
\label{Bolinf-min}
\min_{u\in\mathcal U} \left\{\int_0^T L(t,\mu(t),u(t))\di t  \right\},
\end{equation}
subject to \eqref{Wmayeq}.
We assume for the running cost $L$ that a function $l\colon [0,T]\times \mathcal{P}_b(C) \times C \times  C^1_b(C;Z)\to \R$ exists such that
$$
L(t,\mu,u)=\int_C l(t,\mu,c,u(c)) \di \mu(c).
$$
Moreover, the following hold.
\begin{tcolorbox}
{\bf (HL)}:
\begin{itemize}
\item[$(a)$] there exists a constant $L> 0$ such that for every $t\in [0,T]$, $\mu,\tilde\mu \in \mathcal{P}_c(C)$ and $u\in U$
\begin{equation*}
|L(t,\mu,u)-L(t,\tilde\mu,u)| \leq L W_1(\mu,\tilde\mu);
\end{equation*}
\item[$(b)$] there exists $M>0$ such that for every $t\in[0,T]$, $\mu \in \mathcal{P}_c(C)$ and $u\in U$ there holds
\begin{equation*}
|L(t,\mu,u)| \leq  M\left( 1 + m_1(\mu)\right);
\end{equation*}
\item[$(c)$] for every $\mu\in\mathcal{P}_c(C)$ and $u\in U$ the map $t\mapsto L(t,\mu,u)$ belongs to $L^1([0,T];\R)$;
\item[$(d)$] for every $t\in [0,T]$ and $\mu\in\mathcal{P}_c(C)$ the map $u\mapsto L(t,\mu,u)$ is such that
$$
\|L(t,\mu,u)-L(t,\mu,\tilde{u})\|_E \leq \omega(\|u-\tilde{u}\|_{C_b^0}),
$$
where $\omega \colon [0,+\infty) \to [0,+\infty)$ is not dependent on $t$ and with $\lim_{s\to 0^+} \omega(s)=0$;
\item[$(e)$] for every $t\in [0,T]$ and $u\in U$ the map $\PC \ni \mu \mapsto L(t,\mu,u)\in \R$ is locally differentiable at any $\mu$ in the sense of the Definition \ref{Wmudiff} with differential $\nabla_\mu L:=\nabla_\mu L(t,\mu,u)$ and the application $\PC\times U\times C \ni (\mu,u,\tilde c) \mapsto \nabla_\mu L(\tilde c) \in E_C^*$ is continuous for every $t\in[0,T]$.
\end{itemize}
\end{tcolorbox}

We define the Hamiltonian $\mathcal{H}\colon [0,T]\times \mathcal{P}_b(C\times E_C^*) \times U \to \R$ for the problem \eqref{Bolinf-min}-\eqref{Wmayeq} as
\begin{equation}
\label{WhamL}
\mathcal{H}(t,\nu,\omega)=\int_{C\times E_C^*} \langle p, A(t,\pi_{\#}^1\nu,c,\omega(c))\rangle \di \nu(c,p) - L(t,\pi_{\#}^1\nu,\omega).
\end{equation}
As seen for the Mayer problem we can compute explicitly $\nabla_\nu \mathcal H(t,\nu,u)\colon  C \times E_C^* \to E_C^*\times E_C$:
\begin{eqnarray}
\label{form-diff-HL}
&&\nabla_\nu \mathcal H(t,\nu,u)(c,p) =  \\
&&\left(\begin{array}{l} D_c^{*}A(t,\pi_{\#}^1\nu,c,u)[p] + \displaystyle\int_{C\times E_C^*}\nabla_{\pi_{\#}^1\nu}^*A(t,\pi_{\#}^1\nu,\tilde c,u)(c)[\tilde p] \di\nu(\tilde c,\tilde p) - \nabla_{\pi_{\#}^1\nu}L(t,\pi_{\#}^1\nu,u)(c) \\
A(t,\pi_{\#}^1\nu,c,u)
\end{array}\right). \nonumber
\end{eqnarray}

Finally we state the Pontryagin maximum principle for the Bolza problem.

\begin{theorem}
\label{mainresL}
Let $(\muo,\uo)\in AC([0,T];\mathcal{P}_c(C))\times\U$ be an optimal pair for \eqref{Bolinf-min}-\eqref{Wmayeq}. Then there exists $\nuo\in AC([0,T]; \mathcal{P}_c(C\times E_C^*))$ which solves in distributional sense
\begin{equation*}\label{WsistL}
\begin{cases}
\dis\dd \nuo(t) = - \mathrm{div}_{(c,p)}\left(\left(J\nabla_{\nuo(t)}\mathcal{H}(t,\nuo(t),\uo(t))(\cdot,\cdot)\right)\nuo(t)\right)    & \text{ in } [0,T), \\
\pi^1_{\#}\nuo(t)=\muo(t), \\
\nuo(T)=\mu(T) \times \delta_{\mathbf{0}} \in \mathcal{P}_c(C\times E_C^*),
\end{cases}
\end{equation*}
where $\mathcal H$ and $\nabla_\nu \mathcal H$ are defined by \eqref{WhamL} and \eqref{form-diff-HL} respectively. Moreover the following maximality condition holds
\begin{equation*}
\label{WmaxcondL}
\mathcal{H}(t,\nuo(t),\uo(t))=\max_{\omega\in U}\{\mathcal{H}(t,\nuo(t),\omega)\} \qquad \text{for a.e. }t\in [0,T].
\end{equation*}
\end{theorem}

\begin{proof}
The proof is an infinite-dimensional adaptation of the proof of Theorem \ref{PMP}. We do not report the details here as they are very similar of \cite[Section 4.2]{BonFra}.
\end{proof}

\section{Model examples}
\label{s:example}

In this section we discuss some examples that fit into and justify the theoretical framework presented in Section~\ref{s3}. The general setting we consider is that of multi-agent multi-label systems, introduced in~\cite{AFMS, MS2020}. Control problems for such systems were studied in~\cite{AAMS}, focusing on existence and finite-particle approximations. As mentioned in the introduction, the state space used to model multi-agent multi-label systems accounts for agents' position $x \in \R^{d}$ and label $\lambda \in \mathcal{P} (K)$, for a metric space $(K, d_{K})$ of pure strategies. We briefly recall the functional setting of~\cite[Section~2.1]{AFMS}. For a Lipschitz function $\varphi:K \to \R$ we define respectively
$$
\mathrm{Lip}(\varphi):= \sup_{k_1\neq k_2\in K} \frac{|\varphi(k_1)-\varphi(k_2)|}{d_K(k_1,k_2)} \qquad\text{and}\qquad \|\varphi\|_{\mathrm{Lip}}:= \sup_{k\in K}|\varphi(k)| + \mathrm{Lip}(\varphi).
$$
We denote by
\begin{equation}\label{P1}
\mathcal{F} (K) :=  \overline{ {\rm span} ( \mathcal{P} (K) ) }^{\|\cdot\|_{\mathrm{BL}}} \subset \mathrm{Lip}(K)^*\,,
\end{equation}
where the closure in~\eqref{P1} is taken with respect to the Bounded Lipschitz norm
\begin{equation*}
\lVert \mu \rVert_{\mathrm{BL}}:= \sup \big\{\langle \mu,\varphi\rangle:  \varphi\in {\rm Lip} (K), \ \|\varphi\|_{\rm Lip} \leq 1\big\} \qquad \text{for $\mu\in {\rm Lip} (K)^*$}\,.
\end{equation*}
This space, also called Arens-Eells space in the literature, is a separable Banach space. In particular, $E=\R^d \times \mathcal{F}(K)$ is a separable Banach space, endowed with the norm $\| c\|_{E} := | x| + \| \lambda\|_{\rm BL}$ for $c = (x, \lambda) \in E$, where $| \cdot |$ denotes the standard Euclidean norm of~$\R^{d}$. This setting makes the state space $C:= \R^{d} \times \mathcal{P}(K)$ a closed and convex subset of $E$. The formulation of mean-field control problems and of optimality conditions thus face the lack of linear structure and smoothness of~$C$, which we addressed in the previous sections. We remark that the Euclidean-type of results of~\cite{BonFra, BonRos, Burger-Tse} cannot be applied.

We start by adapting the examples presented in~\cite{AAMS}. In the Section~\ref{s:entropic} we focus instead on entropy regularised systems, introduced in~\cite{ADEMS, BFS} for multi-agent multi-label systems~\cite{AFMS, MS2020}.

\subsection{Leader-follower dynamics}

In a leader-follower dynamics, the label variable $\lambda$ identifies to which group (e.g., leaders or followers) an agent belongs. We recall here the functional setting described in~\cite[Section~5]{Mor1} (see also \cite{Maas, Mielke}), which describes transition process from one group to another by means of reversible Markov chains on $n$ states. The set of labels $K$ is written as $K = \{ e_{1}, \ldots, e_{n}\}$, where $e_{i}$ are the elements of the canonical base of~$\R^{n}$, endowed with the distance~$d_{K}$
\begin{equation*}
0= d_{K} (e_{i}, e_{i}) \quad \text{for $i = 1, \ldots, n$}, \qquad  1=d_{K}(e_{i} , e_{j}) \quad \text{for $i \neq j$}.
\end{equation*}
The space of probability measures $\cP(K)$ is identified with the closed $(n-1)$-simplex 
\begin{align}
\Lambda_{n}:= \bigg\{ \lambda = (\lambda_{1}, \ldots, \lambda_{n}) \in \R^{n}: \lambda_{i} \geq 0\,, \ \sum_{i=1}^{n} \lambda_{i} = 1 \bigg\}\,.
\end{align}
The state space is represented by the convex subset $C= \R^{d} \times \mathcal{P}(K) \sim \R^{d} \times \Lambda_{n}$ of $E= \R^{d} \times \mathcal{F} (K)$, where $\mathcal{F}(K)$ is defined by \eqref{P1}. We further have that
\begin{displaymath}
E_{C} = \overline{ \R(C - C)} =\R^{d} \times  \{ \mu \in \mathcal{M} (K): \, \mu(K) = 0\}\,.
\end{displaymath}
As $\mathcal{P} (K)$ is identified with~$\Lambda_{n}$, we notice that $E_{C}$ may be identified with $\R^{d} \times \R^{n}_{0}$, where
\begin{displaymath}
\R^{n}_{0} := \bigg\{ \xi \in \R^{n}: \, \sum_{i=1}^{n} \xi_{i} = 0 \bigg\}\,.
\end{displaymath}
Hence,~$E_{C}$ is a reflexive Banach space. We fix $\mathcal{U} : = L^{1}([0, T]; U)$ for $U$ a not empty compact subset of $(C^{1}_{b} (C ; \R^{d}), \| \cdot \|_{C^{1}_{b}})$. For $\Psi \in \cP_{1}(C)$ and $u \in U$ the velocity field~$A(\Psi, c, u)$ in~\eqref{Wmayeq} takes the form
\begin{equation}
\label{e:AAA1000}
A(\Psi, c, u) := ( v_{\Psi} (c) + h_{\Psi}(c)  u \ , \ \cT_{\Psi} (c) )\,.
\end{equation}
We present below an explicit example of the velocity~$v_{\Psi} \colon C \to \R^{d}$ and of the transition operator $\cT_{\Psi} \colon C \to \R^{n}_{0}$ in the special case $n=2$. Here we remark that $\cT_{\Psi}$ takes the form
\begin{displaymath}
\cT_{\Psi} (c) := \mathcal{Q} (x, \Psi) \lambda \qquad \text{for $c = (x, \lambda) \in C$,}
\end{displaymath}
for a matrix-valued map $\mathcal{Q}\colon \R^{d} \times \mathcal{P}_{1} (C)  \to \R^{n \times n}$ satisfying the following:
\begin{itemize}
\item [$(\mathcal{Q}_{0})$] for every $(x, \Psi) \in \R^{d} \times \mathcal{P}_{1}(C)$ and every $i, j = 1, \ldots, n$, $\mathcal{Q}_{ij}(x, \Psi) \geq 0$ for~$i \neq j$, and $\mathcal{Q}_{ii}(x, \Psi) = - \sum_{j \neq i} \mathcal{Q}_{j i} (x, \Psi)$; 

\item[$(\mathcal{Q}_{1})$] for every $(x, \Psi) \in \R^{d} \times \mathcal{P}_{1}(C)$, $\mathcal{Q} ( x, \Psi)$ is \emph{reversible}, that is, there exists a unique $\sigma= \sigma (x, \Psi) \in \Lambda_{n}$ such that
\begin{displaymath}
\mathcal{Q}_{ij} ( x, \Psi ) \sigma_{j} = \mathcal{Q}_{ji} (x, \Psi) \sigma_{i} \qquad \text{for every $i, j = 1, \ldots, n$}\,,
\end{displaymath}
%\item[$(\mathcal{Q}_{2})$] $\Qq$ is locally Lipschitz, that is, for every $R>0$ there exists $L_{\Qq, R}>0$ such that for every $x_{1}, x_{2} \in B_{R}$ and every $\Psi_{1}, \Psi_{2} \in \Pp (\B^{Y}_{R})$
%\begin{displaymath}
%| \Qq( x_{1}, \Psi_{1}) - \Qq (x_{2}, \Psi_{2}) | \leq L_{\Qq, R} \big( | x_{1} - x_{2} | + W_{1} (\Psi_{1}, \Psi_{2}) \big)\,;
%\end{displaymath}
%
%\item[$(\Qq_{3})$] there exists $M_{\Qq} > 0$ such that for every $x \in \R^{d}$ and every $\Psi \in \Pp_{1}(Y)$
%\begin{displaymath}
%| \Qq ( x, \Psi) | \leq M_{\Qq} \big( 1 + |x| + m_{1} ( \Psi ) \big) \,.
%\end{displaymath}
\end{itemize}
together with local Lipschitz continuity and linear growth conditions (see also {\bf(HA1)}).

For the sake of simplicity, we now consider the case $n=2$ and present an explicit example of $ v_{\Psi}$, $h_{\Psi}$, and $\cT_{\Psi}$, following the lines of~\cite[Section~5.1]{AAMS}. We notice that $\Lambda_{2} = [0, 1]$, where $\lambda \sim 0$ indicates a strong leadership. As $K = \{ e_{1}, e_{2}\}$, this means that~$e_{1}$ is the label given to leaders, while $e_{2}$ represents followers. From now on, we will denote such labels with $L$ and~$F$, respectively. To tune the influence of the control $u$, the easiest choice is to fix $h_\Psi(x,\lambda ) = h ( \lambda )$ for a non-negative $C^{1}_{b}$-function $h\colon[0,1]\to\R$. Having in mind applications where a policy maker steers the whole systems by actively controlling only the leaders' population, $h$ may be assumed non-increasing and equal to zero if $\lambda$ is close to~$1$.

In order to define $v_{\Psi}$, we first partition the total population into leaders and followers, according to the value of~$\lambda$. Let us consider a $C^{1}$-function  $g\colon [0,1]\to [0,1]$. For $\Psi\in\cP(\R^d\times[0,1])$ we define the followers and leaders distributions as
\begin{equation}\label{I7}
\begin{split}
& \mu^F_\Psi(B):=\int_{B \times[0,1]} g(\lambda) \,\mathrm{d}\Psi(x, \lambda), \qquad 
 \mu^L_\Psi(B):=\int_{B \times [0,1]} (1-g(\lambda)) \,\mathrm{d}\Psi(x, \lambda),
\end{split}
\end{equation}
for each Borel set $B \subset \mathbb{R}^d$. 
%In particular, the sum $\mu_\Psi^F(B)+\mu_\Psi^L(B)$ coincides with the first marginal of $\Psi$ and therefore it counts the total population contained in $B$.
%In the discrete setting, the leaders and followers distributions in \eqref{I7} are given by
%\begin{equation}\label{I7-discrete}
%\begin{split}
%& \mu_{\Psi^N}^F(B)=\frac1N\sum_{i\,:\,x_i\in B}g(\lambda_i),
%\\
%&
% \mu_{\Psi^N}^L(B)=\frac1N\#\{i:x_i\in B\}-\mu_{\Psi^N}^F(B)=\frac1N\sum_{i\,:\,x_i\in B}(1-g(\lambda_i)).
%\end{split}
%\end{equation}
A rather standard choice for~$g$ is any smooth regularisation of the indicator function of the set $\{\lambda \ge m\}$, for a small threshold $m \ge 0$. In this way, we classify individuals with $\lambda \sim 0$ as strong leaders, while the remaining agents are treated as followers. Different choices of~$g$ may be considered as well, such as the easier $g(\lambda) = \lambda$, measuring the average influence of an agent in~$B$ on the surrounding agents.

The velocity $v_{\Psi}$ combines the effects of leaders and followers. One possibility is to define
\begin{align}
\label{e:vel-ex}
v_{\Psi}(x, \lambda) = & \   g (\lambda)  \big ( K^{FF} * \mu^{F}_{\Psi} (x)  + K^{LF} * \mu^{L}_{\Psi} (x) \big)  
\\
&
+ ( 1 - g (\lambda)) \big ( K^{FL}* \mu^{F}_{\Psi} (x) + K^{LL} * \mu^{L}_{\Psi} (x) \big) \,,\nonumber
\end{align}
where 
\begin{itemize}
\item[$(K)$] $K^{\star \bullet} \colon \R^{d} \to \R^{d}$ for $\star,\bullet\in\{F,L\}$ are $C^{1}_{b}$-kernels of interaction.
\end{itemize}
Clearly, different type of velocities may be designed starting from~\eqref{e:vel-ex}. One may consider weights different from~$g$, as well as convolutions with distributions different from~$\mu^{F}_{\Psi}$ and~$\mu^{L}_{\Psi}$. We refer to~\cite[formula (76)]{AAMS} for more general examples.

Finally, we introduce the transition operator~$\cT_{\Psi}$. Thanks to the identification of $\cP_1(K)$ with $[0,1]$, $\cT_\Psi(x,\lambda)$ is actually identified with a scalar, instead of taking values  in the 2-dimensional space $\mathcal{F} (K)$. We notice, indeed, that the convex constraint~$\lambda \in \mathcal{P} (K)$ requires $\mathcal{T}_{\Psi}$ to have a nontrivial kernel containing all constants. Hence, the first component of $\cT_\Psi(x,\lambda)$ uniquely determines the second one. Thus, we define
\begin{equation}\label{transition}
\cT_\Psi(x,\lambda)=-\alpha_F (x,\Psi)g_1(\lambda) + \alpha_L(x,\Psi)(1-g_1(\lambda)),
\end{equation}
with $\alpha_\bullet$ having the typical form
$$\alpha_\bullet(x,\Psi)=\int_{\R^d\times[0,1]} H_\bullet(x' - x ) \ell_\bullet(\lambda')\,\de\Psi(x',\lambda'),\quad\text{for $\bullet\in\{F,L\}$,}$$
and where 
\begin{itemize}
\item[$(\alpha)$] $g_1\colon[0,1]\to[0,1]$, $H_\bullet\colon\R^d\to\R_+$, and $\ell_\bullet\colon[0,1]\to[0,1]$ are $C^{1}$-functions with
\begin{equation}\label{e:alpha}
\alpha_\bullet\ge 0\,,\quad g_1(0)=0\,,\quad g_1(1)=1\,.
\end{equation}
\end{itemize}
The condition~$(\alpha)$ guarantees that the evolution of  $\lambda$ is confined into $[0,1]$ (see~\cite{MS2020}). Particular choices of~$g_{1}$ are $g_1(\lambda)=\lambda$, leading to  a linear master equation for~$\lambda$ for fixed~$x$ and~$\Psi$, or $g_1=g$, meaning that the switching rates $\alpha_F$ and $\alpha_L$ are activated depending on the population to which an agent belongs. The function $H_\bullet$ can be used to localize the effect of the overall distribution on the transition rates, while $\ell_\bullet$ tunes the influence of the surrounding agents according to their probability of belonging to the populations of followers or leaders.

In the setting described above, it was shown in~\cite{AAMS} that the $A(\Psi, c, u)$ defined in~\eqref{e:AAA1000} complies with {\bf (HA1)}. We now show that also {\bf (HA2)} are satisfied. Condition {\em (v)} of {\bf (HA2)} is trivial, as $A$ is linear with respect to the control variable. Let us show {\em (vi)--(vii)}. To this purpose, we have to compute the C-differential of the map $c \mapsto A(\Psi, c, u(t, c))$ for $c \in C$ and $u \in \mathcal{U}$ fixed. We write it piece-by-piece:
\begin{align}
D_{c} v_{\Psi} (x, \lambda) [c_{1}- c_{2}] = &\ ( \lambda_{1} - \lambda_{2}) g'(\lambda)  \big(  K^{FF} * \mu^{F}_{\Psi} (x)  + K^{LF} * \mu^{L}_{\Psi} (x) \big)  
\\
&
- (\lambda_{1} - \lambda_{2}) g'(\lambda)  \big( K^{FL}* \mu^{F}_{\Psi} (x) + K^{LL} * \mu^{L}_{\Psi} (x) \big)  \nonumber
\\
&
- g(\lambda)  \big( \nabla_{x} K^{FF} * \mu^{F}_{\Psi} (x) + \nabla_{x}K^{LF} * \mu^{L}_{\Psi} (x) \big) \cdot (x_{1} - x_{2}) \nonumber
\\
&
- ( 1 - g(\lambda) ) \big( \nabla_{x} K^{FL}* \mu^{F}_{\Psi} (x) + \nabla_{x} K^{LL} * \mu^{L}_{\Psi} (x)\big) \cdot (x_{1} - x_{2}) \nonumber
\\
D_{c} \mathcal{T}_{\Psi} (x, \lambda) [c_{1}- c_{2}] = &\ \nabla_{x} \alpha_{F} (x, \Psi) g_{1} (\lambda) \cdot (x_{1} - x_{2}) 
\\
&
+ \nabla_{x} \alpha_{L} (x, \Psi) (1 - g_{1} (\lambda)) \cdot (x_{1} - x_{2}) \nonumber \\
&
+ (\lambda_{1} - \lambda_{2}) g'_{1} (\lambda)  \big( \alpha_{L} (x, \Psi) - \alpha_{F} ( x . \Psi) \big) \nonumber 
\\
\nabla_{x} \alpha_{\bullet} (x,\Psi) = & \ - \int_{\R^d\times[0,1]} \nabla_{x} H_\bullet ( x' - x ) \ell_\bullet(\lambda')\,\de\Psi(x',\lambda'),\quad\text{for $\bullet\in\{F,L\}$,}
\end{align}
In particular, the map $(\Psi, c, u) \mapsto D_{c} A ( \Psi, c, u(t, c))$ is continuous on~$\mathcal{P}_{c} (C)  \times C \times U$. 

We now show that {\em (vii)} of {\bf (HA2)} is satisfied. We write the local Wasserstein differential $\nabla_{\Psi} A(\Psi, c, u)$ componentwise, according to Definition~\ref{Wmudiff}. For every $c, c_{1}, c_{2} \in C$ and for $\bullet, \star \in \{F, L\}$ we have
\begin{align}
\label{e:some-grad1000}
& \big[  \nabla_{\Psi}  g(\lambda) K^{\bullet\star} * \mu^{\bullet}_{\Psi} (x) \big](c_{1}) [c_{2} - c_{1}]
\\
&
\qquad = g(\lambda) \big[ g(\lambda_{1}) \nabla_{x} K^{\bullet\star} ( x_{1} - x) [x_{2} - x_{1}] + g'(\lambda_{1} ) (\lambda_{2} - \lambda_{1}) K^{\bullet\star} (x_{1} - x) \big]\,,\nonumber
\end{align}
from which we recover the expression of $\nabla_{\Psi} v_{\Psi} (x, \lambda)$. With the choice $h_{\Psi} (x, \lambda) = h(\lambda)$ we clearly have $\nabla_{\Psi}( h u) = 0$. As for $\nabla_{\Psi} \mathcal{T}_{\Psi}$, we notice that for $\bullet \in \{F, L\}$ and $c, c_{1}, c_{2} \in C$ it holds
\begin{align}
\label{e:some-grad2000}
& \big[\nabla_{\Psi}  g_{1} (\lambda) \alpha_{\bullet} (x, \Psi) \big] (c_{1}) [c_{2} - c_{1}] = g_{1} (\lambda) \big[ \nabla_{\Psi} \alpha_{\bullet} (x, \Psi) \big] (c_{1}) [c_{2} - c_{1}] 
\\
& \qquad = g_{1} (\lambda) \ell_{\bullet} (\lambda_{1} ) \nabla_{x} H_{\bullet} (x_{1} - x) [x_{2} - x_{1}] + g_{1} (\lambda) (\lambda_{2} - \lambda_{1}) \ell'_{\bullet} (\lambda_{1} )  H_{\bullet} (x_{1} - x) \,.\nonumber
\end{align}
The expression of $\nabla_{\Psi} \mathcal{T}_{\Psi}$ can be recovered from~\eqref{e:some-grad2000}. We infer from~\eqref{e:some-grad1000} and~\eqref{e:some-grad2000} that $A$ also satisfies {\em (vii)} of {\bf (HA2)}. Hence, the velocity field~$A$ in~\eqref{e:AAA1000} fits into our theoretical framework.

We conclude with an example of cost functional $L$. Besides a control cost, the cost functional given below features the competition between two effects. The first one represents the distance of the whole system from a desired goal $\bar{x} \in \R^{d}$. The second contribution translates the fact that leaders want to stay close to the followers' population. Hence, we take for $u \in \mathcal{U}$, $t \in [0, T]$, and $\Psi \in \mathcal{P}_{1} (C)$
\begin{align}
\label{e:cost1000}
L(\Psi, u) = & \ \int_{\R^{d} \times [0, 1]}\bigg(  |x-\bar x|^2 + \theta(\lambda) \Big|x - \int_{\R^{d}} x'\,\de\mu_\Psi^F(x') \Big|^2\bigg)  \, \di \Psi(x, \lambda)
\\
&
\nonumber + \int_{\R^{d} \times [0, 1]} | u(t , x, \lambda) |^{p} \, \di \Psi(x, \lambda) \,,
\end{align} 
where the $C^{1}_{b}$-function $\theta\colon[0,1]\to[0,1]$ selects leaders among the agents, so that $\theta(\lambda)=0$ when $\lambda$ is above a given threshold. A possible choice is $\theta(\lambda)=1-g(\lambda)$. Then, $L$ satisfies {\bf (HL)}.

The above scenario may as well be generalised to include any finite number of labels. Examples with two populations of competing leaders and one of followers have been provided in~\cite[Section~5.2]{AAMS}.

\subsection{Replicator-type models with entropy regularisation}
\label{s:entropic}

We consider here another example fitting into our theoretical framework, namely that of entropy regularised multi-label systems~\cite{ADEMS}. We fix a compact metric space $(V, d_{V})$ of pure strategies, a probability measure $\eta \in \mathcal{P}(V)$ with $\supp(\eta) = V$, and $p \in (1, +\infty)$. We define as ambient space $E := \R^{d} \times L^{p} (V, \eta)$, where
\begin{displaymath}
L^{p}(V, \eta) := \bigg\{ \ell \colon V \to \R: \, \int_{V} | \ell(v)|^{p} \, \di \eta(v) <+\infty \bigg\}.
\end{displaymath}
Denoting by $| \cdot |$ the Euclidian norm in $\R^d$ and by~$\| \cdot\|_{p}$ the $L^{p}$-norm in~$L^{p}(V, \eta)$, we endow~$E$ with the norm $\| \cdot\|_{E} = | \cdot| + \| \cdot\|_{p}$. With such choice,~$E$ is a separable and reflexive Banach space. For $0 < r < R <+\infty$ we set
\begin{displaymath}
C_{r, R} := \R^{d} \times \{ \ell \in L^{p} (V, \eta) \cap \mathcal{P}(V) : \, r \leq \ell(v) \leq R \ \text{for $\eta$-a.e.~$v \in V$}\}.
\end{displaymath}
In particular, $C_{r, R}$ is a convex and closed subset of~$E$. To shorten the notation, we denote by $c =(x, \ell)$ the generic element of~$E$.

Let us fix $0 < r < R < +\infty$. As in Section \ref{s3}, we consider $\mathcal{U} = L^{1}([0, T]; U)$ for $U$ a not empty compact subset of $(C^{1}_{b} (C_{r, R}; \R^{d}), \| \cdot \|_{C^{1}_{b}})$. As velocity field, for $\Psi \in \cP_{1}(C_{r, R})$ and $u \in U$ we consider
\begin{equation}
\label{e:AAA}
A(\Psi, c, u) := ( v_{\Psi} (c) + u , \cT_{\Psi} (c) + \varepsilon \mathcal{S} (\ell)),
\end{equation}
for some $\varepsilon>0$. Here, $v_{\Psi} \colon E \mapsto \R^{d}$ takes the form
\begin{displaymath}
v_{\Psi} (c) :=  \int_{C_{r,R}} K(x' - x) \, \di \Psi(x', \ell'),
\end{displaymath}
for a suitable interaction kernel $K \colon \R^{d} \to \R^{d}$ such that
\begin{enumerate}[label=(KE), ref=KE]
\item \label{KKK} $K \in C^{1}_{b} (\R^{d}; \R^{d})$.
\end{enumerate}
As for the operators $\cT_{\Psi} \colon C_{r, R} \to L^{p} (V, \eta)$ and $\mathcal{S} \colon C_{r, R} \to L^{0} (V, \eta)$ we consider
\begin{align*}
&
 \cT_{\Psi} (c) := \bigg( \int_{C_{r,R}} J(x, \cdot , x') \, \di \Psi(x', \ell') - \int_{V} \int_{C_{r,R}}  J(x, v' , x') \, \ell(v')\, \di \Psi(x', \ell')  \, \di \eta(v')\bigg) \ell, \\
 & \mathcal{S} (\ell) := \bigg( \int_{V} \ell(v)\,\log(\ell (v)) \,\de\eta(v) - \log (\ell) \bigg) \ell,
\end{align*}
for a payoff function $J \colon \R^{d} \times V \times \R^{d} \to \R$ satisfying
\begin{enumerate}[label=(J.\arabic*), ref=J.\arabic*]
\item \label{J1} $J$ is bounded on $\R^{d} \times V \times \R^{d}$ and there exists a constant $L> 0$ such that for every $x_{1}, x'_{1}, x_{2}, x'_{2} \in \R^{d}$ and every $v \in V$
\begin{equation*}
| J (x_{1} , v , x'_{1} ) - J (x_{2} , v , x'_{2} )| \leq L \big( | x_{1} - x_{2}| + | x'_{1} - x'_{2}| \big);
\end{equation*}
\item \label{J2} for every $v \in V$, the map $(x, x') \mapsto J(x, v, x')$ is differentiable with $(x, x') \mapsto (\nabla_{x} J(x, v, x'), \nabla_{x'} J(x, v, x'))$ continuous in  $\R^{d}\times \R^{d}$.
\end{enumerate}
The operator~$\cT_{\Psi}$ introduced above fits in the theoretical framework developed in~\cite{ADEMS} and covers the paramount example of the replicator dynamics, described in~\cite{BFS} in the entropy regularised setting. Under the assumptions~\eqref{KKK} and \eqref{J1}--\eqref{J2}, it has been shown in~\cite[Proposition~3.2]{ADEMS} and in~\cite[Proposition~5.8]{MS2020} that for every $\varepsilon>0$ there exists $0 < r_{\varepsilon} < R_{\varepsilon} <+\infty$ such that, setting $C := C_{r_{\varepsilon}, R_{\varepsilon}}$, the velocity field $A$ defined in~\eqref{e:AAA} complies with assumptions  {\em(i)}--{\em (iv)} in {\bf (HA1)}.
Moreover, we have that
\begin{displaymath}
E_{C} = \overline{\R( C - C)} = \R^{d} \times \bigg\{ \ell \in L^{p} (V, \eta): \, \int_{V} \ell (v) \, \di \eta(v) = 0\bigg\}\,.
\end{displaymath}
Hence, $E_{C}$ is reflexive.

We now verify that $A$ also satisfies the set of assumptions~{\bf (HA2)}. As before, condition~{\em (v)} of {\bf (HA2)} is trivial in view of the linear dependence of~$A$ on the control variable~$u$. As for the $C$-differentiability of $A$, we write the $C$-differential piece-by-piece as
\begin{align*}
& D_{c} v_{\Psi} (c) [c_{1} - c_{2}] =  - \int_{C} \nabla K(x' - x)(x_{1} - x_{2}) \, \di \Psi(x', \ell'), \\
& D_{c} \cT_{\Psi} (c) [c_{1} - c_{2}] =  \bigg( \int_{C} \nabla_{x} J(x, \cdot , x') \cdot (x_{1} - x_{2})  \, \di \Psi(x', \ell')
 \\
 &
\qquad - \int_{V} \int_{C}  \nabla_{x} J(x, v' , x')\cdot (x_{1} - x_{2}) \, \ell(v')\, \di \Psi(x', \ell')  \, \di \eta(v')\bigg) \ell
 \\
 &
 \qquad + \bigg( \int_{C} J(x, \cdot , x') \, \di \Psi(x', \ell') - \int_{V} \int_{C}  J(x, v' , x') \, \ell(v')\, \di \Psi(x', \ell')  \, \di \eta(v')\bigg) ( \ell_{1} - \ell_{2})
 \\
 &
\qquad  - \ell \int_{V} \int_{C}  J(x, v' , x') \, (\ell_{1}(v') - \ell_{2}(v') ) \, \di \Psi(x', \ell')  \, \di \eta(v'),
\\
&
D_{c} \mathcal{S} (c) [c_{1} - c_{2}] = \bigg( \int_{V} \ell(v)\,\log(\ell (v)) \,\de\eta(v) - \log (\ell) -1 \bigg) (\ell_{1} - \ell_{2})
\\
&
\qquad +  \ell  \int_{V} (\ell_{1} (v) - \ell_{2} (v)) ( 1 + \log(\ell(v)))\, \di \eta(v),
\end{align*}
for every $\Psi \in \cP(C)$ and every $c_{1}, c_{2} \in C$ (recall that $c_{i} = (x_{i}, \ell_{i})$ for $i = 1, 2$). In particular,~\eqref{KKK} and~\eqref{J1}--\eqref{J2} imply that~$A$ satisfies {\em (vi)} of {\bf (HA2)}.

Finally, the map $\PC \ni \Psi \mapsto A(\Psi,c,u)\in E_C$ is locally differentiable at any $\Psi$ in the sense of the Definition~\ref{Wmudiff} with differential
\begin{align*}
& \big[ \nabla_{\Psi} A(\Psi, c , u)\big] (c_{1}) [c_{2} - c_{1}]
\\
&
\qquad = \left(
\begin{array}{cc}
\nabla K(x_{1} - x) [x_{2} - x_{1}] \\[2mm]
 \big[ \nabla_{x'} J(x, \cdot , x_{1}) \cdot( x_{2} - x_{1} ) - \int_{V}  \nabla_{x'} J(x, v, x_{1}) \cdot (x_{2} - x_{1}) \ell(v) \, \di \eta(v) \big]\, \ell(\cdot)
\end{array}\right),
\end{align*}
for every $c, c_{1}, c_{2} \in C$. In particular,~\eqref{KKK} and \eqref{J1}--\eqref{J2} yield~{\em (vii)} in {\bf (HA2)}.

\section*{Appendix}

We first prove the results whose proof we have postponed, then we give two technical results that we used in Section \ref{s3}.

\subsection*{A: Proofs of Lemma \ref{ext-bnd-inf}, Lemma \ref{Dphi-bnd} and Lemma \ref{conv-flussi-mis}}
\setcounter{definition}{0}
\setcounter{equation}{0}
\renewcommand{\theequation}{A.\arabic{equation}}
\renewcommand{\thedefinition}{A.\arabic{definition}}
\renewcommand{\thetheorem}{A.\arabic{theorem}}
\renewcommand{\theremark}{A.\arabic{remark}}
\renewcommand{\thelemma}{A.\arabic{lemma}}

\begin{proof}[Proof of Lemma \ref{ext-bnd-inf}]
Fix $c\in C$. We recall that $\Phi_{(s,t)}^{\mu^s}\in C_b^0(C;C)$. For brevity of notation we set $\tilde c:=\Phi_{(s,t)}^{\mu^s}(c)$. Let $v\in \mathbb{R}(C-C)$. Then $v=\alpha(c_1-c_2)=\alpha(c_1-\tilde c)-\alpha(c_2-\tilde c)$. It follows by definition of $C$-differentiability (see \eqref{ACdiff}) that
\begin{eqnarray*}
&& \|\Di_{\tilde c}A[v]\|_E = |\alpha|\|\Di_{\tilde c}A[c_1-\tilde c] - \Di_{\tilde c}A[c_2-\tilde c]\|_E \\
&&\stackrel{\eqref{Gat}}= |\alpha| \lim_{h\to 0^+} \left\|\frac{1}{h}\left(A{\scriptstyle\left(t,\mu(t),\tilde c+h(c_1-\tilde c),u(t,\tilde c+h(c_1-\tilde c))\right)} - A{\scriptstyle\left(t,\mu(t),\tilde c+h(c_2-\tilde c),u(t,\tilde c+h(c_2-\tilde c))\right)}\right) \right\|_E \\
&& \stackrel{\eqref{flow-cont},{\bf (HA1)}-(i)}\leq |\alpha|L \lim_{h\to 0^+} \left\{ h\|c_1-c_2\|_E + \|u(t,\tilde c+h(c_1-\tilde c))- u(t,\tilde c+h(c_2-\tilde c))\|_{C_b^0} \right\}.
\end{eqnarray*}
Since $u\in U$ which is compact in $C^1_b(C;Z)$, by Remark A.5 of \cite{AFMS}, there exists a positive constant $L_U$ such that, starting from the last inequality, we deduce
$$
\|\Di_{\tilde c}A[v]\|_E \leq L(1+L_U)\|v\|_E.
$$
By an abuse of notation we use the symbol $L$ to denote $L(1+L_U)$, hence, the result follows.
\end{proof}

\begin{proof}[Proof of Lemma \ref{Dphi-bnd}]
Let us first note that, by \eqref{Das-bnd}, the Cauchy problem \eqref{Dphi-edo} is well-defined and admits a unique solution $z_f(t,c)\in AC([s,T];E_C)$. In particular, we can define the family of diffeomorphisms $L_{(s,t)}^c\colon E_C\to E_C$ as $L_{(s,t)}^c[f]:=z_f(t,c)$. Since the differential equation in \eqref{Dphi-edo} is linear, we have that $L_{(s,t)}^c$ is a linear operator for every $(t,c)\in [s,T]\times C$. Moreover, by its definition
\begin{equation}
\label{Wapp29bis}
z_f(t,c)=f + \int_s^t \Di_{\Phi_{(s,\sigma)}^{\mu^s}(c)}A[z_f(\sigma,c)] \di\sigma,
\end{equation}
which, using \eqref{Das-bnd}, implies
$$
\|z_f(t,c)\|_{E}\leq \|f\|_{E} + L\int_s^t \|z_f(\sigma,c)\|_{E} \di\sigma.
$$
Then, after an application of the Gr\"onwall inequality, we obtain
$$
\|L_{(s,t)}^c[f]\|_E = \|z_f(t,c)\|_E \leq \|f\|_E\ex^{LT}.
$$
Hence $L_{(s,t)}^c\in \li(E_C;E_C)$ and $\|L_{(s,t)}^c\|_{\li(E_C;E_C)} \leq \ex^{LT}$,
consequently we deduce that
$$
\text{the map}\quad (t,c)\mapsto L_{(s,t)}^{c} \quad \text{belongs to } L^{\infty}_{\li \times \mu^s}\left([s,T]\times C; \li(E_C;E_C)\right).
$$
Let us now show that $L_{(s,t)}^c$ is exactly the $C$-differential of $\Phi_{(s,t)}^{\mu^s}$ at $c\in C$, thus proving the statement. By definition of $C$-differential, defining
$$
h:=\|c'-c\|_E, \quad z_h(t):=\frac{\Phi_{(s,t)}^{\mu^s}(c')-\Phi_{(s,t)}^{\mu^s}(c)}{h} \quad\text{and}\quad f:=\frac{c'-c}{h},
$$
we just have to prove that
\begin{equation}
\label{Wapp26}
\lim_{c'\to c} \left\|z_h(t) - z_f(t,c)\right\|_E = 0.
\end{equation}
Note that, by its definition, $z_{h}(t)$ solves
$$
\begin{cases}
\dd z_h(t) = \frac{1}{h}\left\{A\left(t,\mu(t),\Phi_{(s,t)}^{\mu^s}(c'),u(t,\Phi_{(s,t)}^{\mu^s}(c'))\right) - A\left(t,\mu(t),\Phi_{(s,t)}^{\mu^s}(c),u(t,\Phi_{(s,t)}^{\mu^s}(c))\right)\right\}   & \text{ in }(s,T], \\
z_h(s)=\frac{c'-c}{h}.
\end{cases}
$$
We can therefore rewrite the equation for $z_h(t)$ as
$$
\dd z_h(t) = \Di_{\Phi_{(s,t)}^{\mu^s}(c)}A[z_h(t)] + r_h(t),
$$
where
\begin{eqnarray*}
&& r_h(t):= \frac{\|\Phi_{(s,t)}^{\mu^s}(c')-\Phi_{(s,t)}^{\mu^s}(c)\|_E}{h} \times \\
&&\frac{A{\scriptstyle\left(t,\mu(t),\Phi_{(s,t)}^{\mu^s}(c'),u(t,\Phi_{(s,t)}^{\mu^s}(c'))\right)} - A{\scriptstyle\left(t,\mu(t),\Phi_{(s,t)}^{\mu^s}(c),u(t,\Phi_{(s,t)}^{\mu^s}(c))\right)} - \Di_{\Phi_{(s,t)}^{\mu^s}(c)}A[\Phi_{(s,t)}^{\mu^s}(c')-\Phi_{(s,t)}^{\mu^s}(c)]}{\|\Phi_t^{\mu^s}(c')-\Phi_t^{\mu^s}(c)\|_E},
\end{eqnarray*}
leading to
\begin{equation}
\label{Wapp29}
z_h(t)=f + \int_s^t \Di_{\Phi_{(s,\sigma)}^{\mu^s}(c)}A[z_h(\sigma)] \di\sigma + \int_s^t r_h(\sigma)\di\sigma.
\end{equation}
By \eqref{flow}, {\bf (HA1)}-$(i)$, and again applying the Gr\"onwall inequality, we obtain
\begin{equation}
\label{Wapp27}
\|\Phi_{(s,t)}^{\mu^s}(c')-\Phi_{(s,t)}^{\mu^s}(c)\|_E \leq h \ex^{LT},
\end{equation}
which, together to {\bf (HA2)}-$(vi)$, implies that
$$
\lim_{h\to 0} \|r_h(t)\|_E = 0 \qquad \text{for }t\in [s,T].
$$
Moreover, using again {\bf (HA1)}-$(i)$, \eqref{Das-bnd} and \eqref{Wapp27}, it easy to check that
$$
\|r_h(t)\|_E\leq 2L\ex^{LT} \qquad \text{for }t\in [s,T],
$$
so that, from an application of the Lebesgue dominated convergence theorem, we have
\begin{equation}
\label{Wapp28}
r_h(t)\to 0 \quad \text{in} \quad L^1([s,T];E_C) \qquad \text{as }h\to 0.
\end{equation}
Finally, by \eqref{Wapp29bis} and \eqref{Wapp29}, we deduce that
\begin{eqnarray*}
\|z_h(t)-z_f(t,c)\|_E &\leq& \int_s^t \|\Di_{\Phi_{(s,\sigma)}^{\mu^s}(c)}A[z_h(\sigma)-z_f(\sigma,c)]\|_E \di\sigma + \int_s^t \|r_h(\sigma)\|_E \di\sigma \\
&\stackrel{\eqref{Das-bnd}}\leq& L \int_s^t \|z_h(\sigma)-z_f(\sigma,c)\|_E \di\sigma + \delta_h,
\end{eqnarray*}
where
\begin{equation}
\label{Wapp30}
\delta_h:= \int_s^T \|r_h(\sigma)\|_E \di\sigma \stackrel{\eqref{Wapp28}}\to 0 \qquad \text{as }h\to 0.
\end{equation}
Applying again the Gr\"onwall inequality, we conclude that
$$
\|z_h(t)-z_f(t,c)\|_E \leq \delta_h\ex^{LT}.
$$
Therefore, recalling that $h=\|c'-c\|_E$ and by \eqref{Wapp30}, we get \eqref{Wapp26}.
\end{proof}

\begin{proof}[Proof of Lemma \ref{conv-flussi-mis}]
We divide the proof in agreement with the three statements. \\
\textit{Proof of $(a)$ and $(b)$.} Let $\hat{\mu}_\eps$ be the solution of \eqref{Wmayeq} with $u$ replaced by $u_\eps$ and with initial datum $\mu^0$. Let $\Phi_{(0,t)}^{\eps,\mu^0}$ be the family of non-local flows associated to $\hat{\mu}_\eps$. We know by \eqref{flow-cont} that $\Phi_{(0,t)}^{\mu^0}, \Phi_{(0,t)}^{\eps,\mu^0}, \Phi_{(0,t)}^{\eps,\mu^0_\eps}$ belong to $C_b^0([0,T]\times C;C)$. Fix $c\in C$. It follows from triangle inequality that
\begin{equation}
\label{appen1}
\left\| \Phi_{(0,t)}^{\eps,\mu^0_\eps}(c) - \Phi_{(0,t)}^{\mu^0}(c) \right\|_E \leq \left\| \Phi_{(0,t)}^{\eps,\mu^0_\eps}(c) - \Phi_{(0,t)}^{\eps,\mu^0}(c) \right\|_E + \left\| \Phi_{(0,t)}^{\eps,\mu^0}(c) - \Phi_{(0,t)}^{\mu^0}(c) \right\|_E.
\end{equation}
We focus on the second term of the right-hand side of \eqref{appen1}. We define
$$
F_\eps(t):= \sup_{C} \left\| \Phi_{(0,t)}^{\eps,\mu^0}(c) - \Phi_{(0,t)}^{\mu^0}(c) \right\|_E \quad \text{for }t\in [0,T].
$$
It follows from \eqref{flow} and {\bf (HA1)}-$(i)$ that
\begin{eqnarray}
\label{appen2}
&&\|\Phi_{(0,t)}^{\eps,\mu^0}(c) - \Phi_{(0,t)}^{\mu^0}(c)\|_E \\
&& \leq \int_0^t \left\|A{\scriptstyle\left(s,\hat{\mu}_\eps(s),\Phi_{(0,s)}^{\eps,\mu^0}(c),u_\eps(s,\Phi_{(0,s)}^{\eps,\mu^0}(c))\right)} - A{\scriptstyle\left(s,\mu(s),\Phi_{(0,s)}^{\mu^0}(c),u(s,\Phi_{(0,s)}^{\mu^0}(c))\right)}\right\|_E \di s \nonumber \\
&& \leq L \int_0^t \left\{W_1(\hat{\mu}_\eps(s),\mu(s)) + \|\Phi_{(0,s)}^{\eps,\mu^0}(c) - \Phi_{(0,s)}^{\mu^0}(c)\|_E + \|u_\eps{\scriptstyle(s,\Phi_{(0,s)}^{\eps,\mu^0}(c))} - u{\scriptstyle(s,\Phi_{(0,s)}^{\mu^0}(c))}\|_{C_b^0} \right\}\di s \nonumber \\
&& \leq L \int_0^t W_1(\hat{\mu}_\eps(s),\mu(s)) \di s + L\int_0^t \|\Phi_{(0,s)}^{\eps,\mu^0}(c) - \Phi_{(0,s)}^{\mu^0}(c)\|_E \di s \nonumber \\
&& \phantom{\leq} + L\int_0^t \|u_\eps{\scriptstyle(s,\Phi_{(0,s)}^{\eps,\mu^0}(c))} - u_\eps{\scriptstyle(s,\Phi_{(0,s)}^{\mu^0}(c))}\|_{C_b^0} \di s + L\int_0^t \|u_\eps{\scriptstyle(s,\Phi_{(0,s)}^{\mu^0}(c))} - u{\scriptstyle(s,\Phi_{(0,s)}^{\mu^0}(c))}\|_{C_b^0} \di s. \nonumber
\end{eqnarray}
Now we estimate the terms on the right-hand side of \eqref{appen2}. As for the first term, by definition of Wasserstein distance, we deduce that
\begin{equation}
\label{appen3}
\int_0^t W_1(\hat{\mu}_\eps(s),\mu(s)) \di s \leq \int_0^t \int_C \| \Phi_{(0,s)}^{\eps,\mu^0}(c) - \Phi_{(0,s)}^{\mu^0}(c) \|_E\di \mu_0(c) \di s \leq \int_0^t F_\eps(s) \di s.
\end{equation}
Regarding the third term, since $u_\eps,u\in U$ which is compact in $(C^1_b(C;Z),\|\cdot\|_{C^1_b})$, there exists a positive constant $L_U$ such that
\begin{equation}
\label{appen4}
\int_0^t \|u_\eps{\scriptstyle(s,\Phi_{(0,s)}^{\eps,\mu^0}(c))} - u_\eps{\scriptstyle(s,\Phi_{(0,s)}^{\mu^0}(c))}\|_{C_b^0} \di s \leq L_U\int_0^t F_\eps(s) \di s.
\end{equation}
Finally, by assumption, $u_\eps \to u$ in $L^1([0,T]; (U,\|\cdot\|_{C_b^1}))$, then
\begin{equation}
\label{appen5}
\int_0^t \|u_\eps{\scriptstyle(s,\Phi_{(0,s)}^{\mu^0}(c))} - u{\scriptstyle(s,\Phi_{(0,s)}^{\mu^0}(c))}\|_{C_b^0} \di s \to 0 \quad \text{as }\eps\to 0.
\end{equation}
Combining \eqref{appen3}, \eqref{appen4} and \eqref{appen5} with \eqref{appen2}, we obtain that
$$
F_\eps(t) \leq L(2+L_U)\int_0^t F_\eps(s) \di s + \delta_\eps,
$$
where $\delta_\eps$ is a positive constant not depending on $c$ and $t$ which goes to 0 as $\eps\to 0$. Hence, applying Gr\"onwall inequality, we conclude that
\begin{equation}
\label{appen6bis}
F_\eps(t)\leq \delta_\eps \ex^{L(2+L_U)T},
\end{equation}
whence
\begin{equation}
\label{appen6}
\sup_C \left\| \Phi_{(0,t)}^{\eps,\mu^0}(c) - \Phi_{(0,t)}^{\mu^0}(c) \right\|_E \to 0 \quad \text{as }\eps \to 0 \quad\text{uniformly in }t\in [0,T].
\end{equation}
As regards the first term on the right-hand side of \eqref{appen1}, since by Theorem 3.3 of \cite{AFMS} we have for a positive constant $M$
\begin{equation}
\label{appen7bis}
W_1(\mu_\eps(t),\hat{\mu}_\eps(t))\leq \ex^{Mt}W_1(\mu_\eps^0,\mu^0) \quad \text{for } t\in [0,T],
\end{equation}
and recalling that, by assumption, $W_1(\mu_\eps^0,\mu^0)\to 0$ as $\eps \to 0$, we can proceed in the same way used for the second term to obtain that
\begin{equation}
\label{appen7}
\sup_C \left\| \Phi_{(0,t)}^{\eps,\mu^0_\eps}(c) - \Phi_{(0,t)}^{\eps,\mu^0}(c) \right\|_E \to 0 \quad \text{as }\eps \to 0 \quad\text{uniformly in }t\in [0,T].
\end{equation}
Hence, combining \eqref{appen6} and \eqref{appen7} with \eqref{appen1}, we deduce $(a)$. The proof of $(b)$ is an easy consequence of what we have just seen. Indeed, by triangle inequality,
\begin{eqnarray*}
W_1(\mu_\eps(t),\mu(t)) & \leq & W_1(\mu_\eps(t),\hat{\mu}_\eps(t)) + W_1(\hat{\mu}_\eps(t),\mu(t)) \\
& \stackrel{\eqref{appen7bis}}\leq & \ex^{MT}W_1(\mu_\eps^0,\mu^0) + \int_C \| \Phi_{(0,t)}^{\eps,\mu^0}(c) - \Phi_{(0,t)}^{\mu^0}(c) \|_E \di \mu^0(c) \\
& \stackrel{\eqref{appen6bis}}\leq &  \ex^{MT}W_1(\mu_\eps^0,\mu^0) + \delta_\eps \ex^{L(2+L_U)T},
\end{eqnarray*}
so $(b)$ follows. \\
\textit{Proof of $(c)$.} Fix $t\in [0,T]$ and $c\in C$. We define for a fixed $f\in E_C$
$$
z_\eps(t,c):=\Di_c\Phi_{(0,t)}^{\eps,\mu_\eps^0}[f] \quad\text{and}\quad z(t,c):=\Di_c\Phi_{(0,t)}^{\mu^0}[f].
$$
By Lemma \ref{Dphi-bnd} we know that $\|z(t,c)\|_E\leq M\|f\|_E$ for a positive constant $M$ and that
\begin{align}
\label{appen8}
& \|z_\eps(t,c)-z(t,c)\|_E \\
& \leq \int_0^t \| \Di_{\Phi_{(0,s)}^{\eps,\mu_\eps^0}(c)}A{\scriptstyle(s,\mu_\eps(s),\Phi_{(0,s)}^{\eps,\mu_\eps^0}(c),u_\eps(s,\Phi_{(0,s)}^{\eps,\mu_\eps^0}(c)) )}[z_\eps(s,c)-z(s,c)]\|_E \di s \nonumber \\
& + \int_0^t \|(\Di_{\Phi_{(0,s)}^{\eps,\mu_\eps^0}(c)}A{\scriptstyle(s,\mu_\eps(s),\Phi_{(0,s)}^{\eps,\mu_\eps^0}(c),u_\eps(s,\Phi_{(0,s)}^{\eps,\mu_\eps^0}(c)) )} - \Di_{\Phi_{(0,s)}^{\mu^0}(c)}A{\scriptstyle(s,\mu(s),\Phi_{(0,s)}^{\mu^0}(c),u(s,\Phi_{(0,s)}^{\mu^0}(c)))})[z(s,c)]\|_E \di s \nonumber \\
& \stackrel{\eqref{Das-bnd}}\leq L\int_0^t \|z_\eps(s,c)-z(s,c)\|_E \di s \nonumber \\
& + M\|f\|_E \int_0^t \|\Di_{\Phi_{(0,s)}^{\eps,\mu_\eps^0}(c)}A{\scriptstyle(s,\mu_\eps(s),\Phi_{(0,s)}^{\eps,\mu_\eps^0}(c),u_\eps(s,\Phi_{(0,s)}^{\eps,\mu_\eps^0}(c)) )} - \Di_{\Phi_{(0,s)}^{\mu^0}(c)}A{\scriptstyle(s,\mu(s),\Phi_{(0,s)}^{\mu^0}(c),u(s,\Phi_{(0,s)}^{\mu^0}(c)))}\|_{\li(E_C;E_C)} \di s. \nonumber
\end{align}
Furthermore, by assumption {\bf (HA2)}-$(vi)$ and using $(a)$ and the convergence of~$u_{\varepsilon}$ to $u$ in $L^{1} ( [ 0 , T ] ; (U, \| \cdot \|_{C^{1}_{b}} ) )$, we have for a.e.~$t \in [0, T]$ that as $\eps\to 0$
$$
\left\|\Di_{\Phi_{(0,t)}^{\eps,\mu_\eps^0}(c)}A{\scriptstyle(t,\mu_\eps(t),\Phi_{(0,t)}^{\eps,\mu_\eps^0}(c),u_\eps(t,\Phi_{(0,t)}^{\eps,\mu_\eps^0}(c)) )} - \Di_{\Phi_{(0,t)}^{\mu^0}(c)}A{\scriptstyle(t,\mu(t),\Phi_{(0,t)}^{\mu^0}(c),u(t,\Phi_{(0,t)}^{\mu^0}(c)))} \right\|_{\li(E_C;E_C)} \to 0.
$$
Since, by Lemma \ref{ext-bnd-inf}, both $\Di_{\Phi_{(0,t)}^{\eps,\mu_\eps^0}(c)}A$ and $\Di_{\Phi_{(0,t)}^{\mu^0}(c)}A$ are uniformly bounded in $[0,T]$, we can apply the Lebesgue theorem to the last term on the right-hand side of \eqref{appen8} obtaining that
$$
\|z_\eps(t,c)-z(t,c)\|_E \leq L\int_0^t \|z_\eps(s,c)-z(s,c)\|_E \di s + M\|f\|_E \delta_\eps(c),
$$
where $\delta_\eps(c)$ is a positive constant depending on $c$ which goes to 0 as $\eps\to 0$. It follows by an application of the Gr\"onwall inequality that
$$
\|z_\eps(t,c)-z(t,c)\|_E \leq M\|f\|_E \delta_\eps(c)\ex^{LT}.
$$
Hence, recalling the definitions of $z_\eps$ and $z$, we conclude that
$$
\displaystyle\sup_{E_C\ni f\neq 0}\frac{\|(\Di_c\Phi_{(0,t)}^{\eps,\mu_\eps^0}-\Di_c\Phi_{(0,t)}^{\mu^0})[f]\|_E}{\|f\|_E} \leq M\delta_\eps(c)\ex^{LT} \to 0 \quad \text{as }\eps \to 0,
$$
thus finding $(c)$.
\end{proof}

\subsection*{B: Auxiliary lemmas}
\setcounter{definition}{0}
\setcounter{equation}{0}
\renewcommand{\theequation}{B.\arabic{equation}}
\renewcommand{\thedefinition}{B.\arabic{definition}}
\renewcommand{\thetheorem}{B.\arabic{theorem}}
\renewcommand{\theremark}{B.\arabic{remark}}
\renewcommand{\thelemma}{B.\arabic{lemma}}

In what follows let $X$ be a closed and convex subset of a separable Banach space $Y$ and let $E_X$ be the closure of the vector subspace $\R(X - X)$ in~$Y$.

\begin{lemma}
\label{ODE-continua}
Let $\mu$ belong to $\mathcal{P}_c(X)$. Let $B_1\colon [0,T]\times X \to \li(E_X;E_X)$ and $B_2\colon [0,T]\times X \times X \to \li(E_X;E_X)$ be operators satisfying
\begin{itemize}
\item[$(i)$] the map $x\mapsto B_1(t,x)$ belongs to $C(X;\li(E_X;E_X))$ for a.e. $t\in [0,T]$ and $B_1\in L^{\infty}_{\li\times \mu}([0,T]\times X;\li(E_X;E_X))$;
\item[$(ii)$] the map $(x,\tilde x)\mapsto B_2(t,x,\tilde x)$ belongs to $C(X\times X;\li(E_X;E_X))$ for a.e. $t\in [0,T]$ and $B_2\in L^{\infty}_{\li\times \mu\times \mu}([0,T]\times X\times X;\li(E_X;E_X))$.
\end{itemize}
Then there exists a unique continuous weak solution $w\colon [0,T]\times X \to E_X$ of
\begin{equation*}
\label{}
\begin{cases}
\displaystyle \dd w(t,x) = B_1(t,x)[w(t,x)] + \int_X B_2(t,x,\tilde x)[w(t,\tilde x)] \di\mu(\tilde x) & \text{in }(0,T], \\
w(0,x) = y \in E_X.
\end{cases}
\end{equation*}
\end{lemma}

\begin{proof}
We consider the metric space $C_b^0([0,T]\times X;E_X)$ equipped with the norm
$$
\|w\|_{\alpha}:= \sup_{[0,T]\times X} \ex^{-2\alpha t}\|w(t,x)\|_Y.
$$
Set $F:=(C_b^0([0,T]\times X;E_X),\|\cdot\|_{\alpha})$. Since $E_X$ is a closed subset of $Y$ which is a Banach space, it follows that $E_X$ is a complete metric space. Then $F$ is a complete metric space. We define the operator $S\colon F\to F$ as
$$
S(w)(t,x):= y + \int_0^t B_1(s,x)[w(s,x)]\di s + \int_0^t \int_X B_2(s,x,\tilde x)[w(s,\tilde x)]\di \mu(\tilde x) \di s.
$$
Thanks to the assumptions $(i)$-$(ii)$, $S$ is well-defined. Following the same reasoning as in the first part of the proof of Proposition 5 in \cite{BonRos}, it is easy to prove that, with a suitable choice of $\alpha$, $S$ is a contraction mapping. Then, applying the Banach-Caccioppoli theorem, the result follows.
\end{proof}

\begin{lemma}
\label{calc-diff-Ham}
Let $\mathcal H\colon \mathcal{P}_b(X) \to \R$ be
$$
\mathcal{H}(\mu)=\int_X H(\mu,x)\di\mu(x),
$$
where $H\colon \mathcal{P}_b(X)\times X \to \R$. Assume that $H$ is $X$-differentiable in the sense of \eqref{ACdiff} with $X$-differential satisfying
\begin{itemize}
\item[$(i)$] the map $(\mu,x)\mapsto D_xH(\mu,x)$ belongs to $C(\mathcal{P}_b(X)\times X;E_X^*)$;
\item[$(ii)$] there exists a positive constant $L$ not depending on $\mu\in\mathcal{P}_b(X)$ and $x\in X$ such that $\|D_xH(\mu,x)\|_{E_C^*}\leq L$ for every $\mu\in\mathcal{P}_b(X)$ and $x\in X$.
\end{itemize}
Moreover, assume that $H$ is differentiable in the sense of Definition \ref{Wmudiff} with $\mu$-differential satisfying
\begin{itemize}
\item[$(iii)$] the map $(\mu,\tilde x,x)\mapsto \nabla_\mu H(\mu,\tilde x)(x)$ belongs to $C(\mathcal{P}_c(X)\times X\times X;E_X^*)$;
\item[$(iv)$] the map $(\tilde x,x)\mapsto \nabla_\mu H(\mu,\tilde x)(x)$ belongs to $L^{\infty}_{\mu\times\mu}(X\times X;E_X^*)$ for every $\mu\in \mathcal{P}_c(X)$.
\end{itemize}
Then, $\mathcal{H}$ is differentiable at any $\mu\in\mathcal{P}_c(X)$ in the sense of Definition \ref{Wmudiff} and
\begin{equation*}
\label{}
\nabla_\mu\mathcal{H}(\mu)(x)=D_x H(\mu,x) + \int_X\nabla_\mu H(\mu,\tilde x)(x)\di\mu(\tilde x)
\end{equation*}
\end{lemma}

\begin{proof}
Let $R>0$ and $\eta\in \mathcal{P}(B_\mu(R))$. It follows from the disintegration theorem that for $\gamma\in\Gamma(\mu,\eta)$
\begin{eqnarray}
\label{appen10}
&& \mathcal{H}(\eta)-\mathcal{H}(\mu) = \int_X H(\eta,x)\di\eta(x)-\int_X H(\mu,x)\di\mu(x)  \\
&& = \int_{X\times X} \left(H(\eta,x_2)-H(\mu,x_1)\right)\di\gamma(x_1,x_2) \nonumber \\
&&= \underbrace{\int_{X\times X} \left(H(\eta,x_2)-H(\eta,x_1)\right)\di\gamma(x_1,x_2)}_{=:I_1} +\underbrace{\int_{X}\left(H(\eta,x_1)-H(\mu,x_1)\right)\di\mu(x_1)}_{=:I_2}. \nonumber
\end{eqnarray}
Now we focus on $I_1$. Recalling that $X$ is convex and using $(i)$ and $(ii)$, we have for $x(s)=x_1+s(x_2-x_1)\in X$ with $s\in[0,1]$ and applying Fubini theorem for Bochner integral
\begin{eqnarray}
\label{appen11}
I_1 & =& \int_{X\times X} \int_0^1 \langle \Di_{x(s)}H(\eta,x(s)),x_2-x_1\rangle\di s \di\gamma(x_1,x_2) \\
& = & \int_0^1 \int_{X\times X} \langle \Di_{x(s)}H(\eta,x(s)),x_2-x_1\rangle \di\gamma(x_1,x_2) \di s \nonumber \\
& = & \int_{X\times X} \langle \Di_{x_1}H(\mu,x_1),x_2-x_1\rangle \di \gamma(x_1,x_2) \nonumber \\
&& + \underbrace{\int_{X\times X}\langle \Di_{x_1}H(\eta,x_1)-\Di_{x_1}H(\mu,x_1),x_2-x_1\rangle \di\gamma(x_1,x_2)}_{=: I_{1,1}} \nonumber \\
&& + \underbrace{\int_0^1 \int_{X\times X} \langle \Di_{x(s)}H(\eta,x(s)) - \Di_{x_1}H(\eta,x_1),x_2-x_1\rangle \di\gamma(x_1,x_2) \di s   }_{=:I_{1,2}} \nonumber.
\end{eqnarray}
As for $I_{1,1}$, using H\"older inequality, we deduce that
\begin{equation*}
|I_{1,1}| \leq \left(\int_X \|\Di_{x_1} H(\eta,x_1)-\Di_{x_1}H(\mu,x_1)\|_{E_X^*}^2 \di \mu(x_1)\right)^{\frac{1}{2}} W_{2,\gamma}(\mu,\eta).
\end{equation*}
Since, $\gamma\in\mathcal{P}(X\times X)$, we have that if $W_{2,\gamma}(\mu,\eta)\to 0$ then $W_1(\mu,\eta)\to 0$. Hence, using $(i)$ and $(ii)$, we have
$$
\begin{cases}
\Di_{x_1} H(\eta,x_1) \to \Di_{x_1}H(\mu,x_1) \quad \text{as } W_{2,\gamma}(\mu,\eta)\to 0, \\
\|\Di_{x_1} H(\eta,x_1)-\Di_{x_1}H(\mu,x_1)\|_{E_X^*} \leq 2L.
\end{cases}
$$
Therefore, by an application of the Lebesgue theorem, we get
\begin{equation}
\label{appen12}
I_{1,1} = o_R\left(W_{2,\gamma}(\mu,\eta)\right).
\end{equation}
As regards $I_{1,2}$, using again $(i)$ and $(ii)$, we know that for every $\eps>0$ there exists $\delta>0$ depending on $\eta$ such that
\begin{eqnarray*}
|I_{1,2}| & \leq & \eps\int_{\{\|x_1-x_2\|_Y < \delta\}} \|x_1-x_2\|_Y\di\gamma(x_1,x_2) + 2L \int_{\{\|x_1-x_2\|_Y\geq \delta\}} \|x_1-x_2\|_Y \di\gamma(x_1,x_2) \nonumber \\
& \leq & \eps W_{2,\gamma}(\mu,\eta) + \frac{2L}{\delta}W_{2,\gamma}^2(\mu,\eta).
\end{eqnarray*}
It follows that, if $W_{2,\gamma}(\mu,\eta) \to 0$, then we can choose $W_{2,\gamma}(\mu,\eta)\leq \frac{\eps\delta}{2L}$ so that from the last inequality we deduce for every $\eps >0$ that $|I_{1,2}| \leq 2\eps W_{2,\gamma}(\mu,\eta)$. Hence,
\begin{equation}
\label{appen13}
I_{1,2} = o_R\left(W_{2,\gamma}(\mu,\eta)\right).
\end{equation}
Finally we focus on $I_2$. We have using $(iii)$ and $(iv)$ and applying Fubini theorem for Bochner integral that
\begin{eqnarray}
\label{appen14}
I_2 & = & \int_X \int_{X\times X} \langle \nabla_\mu H(\mu,x_1)(x_3) , x_4-x_3\rangle \di\gamma(x_3,x_4) \di\mu(x_1) + o_R\left(W_{2,\gamma}(\mu,\eta)\right) \nonumber \\
& =& \int_{X\times X} \left\langle \int_X \nabla_\mu H(\mu,x_1)(x_3)\di\mu(x_1) , x_4-x_3 \right\rangle \di\gamma(x_3,x_4) + o_R\left(W_{2,\gamma}(\mu,\eta)\right).
\end{eqnarray}
Combining \eqref{appen11},\eqref{appen12},\eqref{appen13} and \eqref{appen14} with \eqref{appen10} we obtain the result.
\end{proof}

\subsection*{Acknowledgements}
The work of S. Almi was funded by the FWF Austrian Science Fund through the Projects ESP-61 and
10.55776/P35359 and by the University of Naples Federico II through FRA Project "ReSinApas".
\par
R. Durastanti has been supported by the Italian Ministry of University and Research under PON “Ricerca e Innovazione” 2014-2020 (PON R\&I, D.M. 1062/21) - AZIONE IV.6 – Contratti di Ricerca su tematiche Green – CUP E65F21003200003, and, his work has been carried out in collaboration with CRdC Tecnologie Scarl as part of the "Embodied Social Experiences in Hybrid Shared Spaces (SHARESPACE)" project - http://sharespace.eu funded by the European Union under Horizon Europe, grant number 101092889.
\par
The work of R. Durastanti and F. Solombrino has been also supported by Gruppo Nazionale per l’Analisi Matematica, la Probabilità e le loro Applicazioni (GNAMPA-INdAM, Project 2024 ``Problemi di controllo ottimo nello spazio di Wasserstein delle misure definite su spazi di Banach'', CUP E53C23001670001).
\par
The work of F. Solombrino and S. Almi is part of the MUR - PRIN 2022, project Variational Analysis of Complex Systems in Materials Science, Physics and Biology, No. 2022HKBF5C, funded by European Union NextGenerationEU.
\par
F. Solombrino also acknowledges support by project Starplus 2020 Unina Linea 1 "New challenges in the variational modeling of continuum
mechanics" from the University of Naples Federico II and Compagnia di San Paolo.

%\subsection*{Declarations of interest:} none.

\end{document}